%% file: Test_HiC.tex
\documentclass[11pt,a4paper,reqno]{amsart}

\usepackage{a4wide, amsfonts, amsmath, amssymb, amsthm, bbm, natbib, xcolor,url,hyperref,tabularx,ifthen,twoopt,enumerate,natbib,amsaddr}

\usepackage[latin1]{inputenc}
\usepackage[T1]{fontenc}
\usepackage{graphicx}

\newcommand{\1}{\mathbbm{1}}
\newcommand{\E}{\mathbb{E}}
\newcommand{\Esp}[1]{\mathbb{E}\left[#1\right]}
\renewcommand{\P}{\mathbb{P}}
\newcommand{\R}{\mathbb{R}}

\newcommand{\jp}{j'}
\newcommand{\kp}{k'}
\newcommand{\n}{n}
\newcommand{\nl}{\n_{\ell}}
\newcommand{\nlun}{\n_{\ell+1}}
\newcommand{\nun}{\n_{1}}
\newcommand{\Rbar}{\overline{R}}
\newcommand{\Rbarli}{\Rbar_{\ell}^{(i)}}
\newcommand{\Rji}{R_{j}^{(i)}}
\newcommand{\SnunL}{S_{\n}\left(\nun,\ldots,n_L\right)}
\newcommand{\Xij}{X_{i,j}}
\newcommand{\Xijp}{X_{i,\jp}}
\newcommand{\Xik}{X_{i,k}}
\newcommand{\Xikp}{X_{i,\kp}}

\theoremstyle{plain}
\newtheorem{lem}{Lemma}

\newtheorem{thm}{Theorem}
\theoremstyle{remark}
\newtheorem{rem}{Remark}

\makeatletter
\newcommand{\labitem}[2]{%
\def\@itemlabel{\textbf{#1}}
\item
\def\@currentlabel{#1}\label{#2}}
\makeatother
\title[Nonparametric homogeneity tests and multiple change-point estimation]{Nonparametric homogeneity tests and multiple change-point estimation
       for analyzing large Hi-C data matrices}
\author{V.\,Brault, S.\,Ouadah, L.\,Sansonnet and C.\,L\'evy-Leduc}
 \address{UMR MIA-Paris, AgroParisTech, INRA, Université 
 Paris-Saclay, France}
\email{vincent.brault@agroparistech.fr}

\date{}

\begin{document}
%----------------------------------------------------------------------
%----------------------------------------------------------------------

\begin{abstract}
We propose a novel nonparametric approach for estimating the location
of block boundaries (change-points) of non-overlapping blocks in a
random symmetric matrix which consists of random variables having their distribution changing from one block to the other.
Our method is based on a nonparametric two-sample homogeneity test for matrices that we extend to the more general case of several groups.
We first provide some theoretical results for the two associated test statistics
and we explain how to derive change-point location estimators.
Then, some numerical experiments are given in order to support our claims.
Finally, our approach is applied to Hi-C data which are used in molecular
biology for better understanding 
the influence of the chromosomal conformation on the cells functioning.
\end{abstract}

\thanks{Vincent Brault would like to thank the French National
  Research Agency ANR which supported this research through the ABS4NGS project (ANR-11-BINF-0001-06).}

\keywords{Nonparametric tests, change-point estimation, Hi-C data}

\maketitle

% \tableofcontents

% \noindent
% \emph{\CL{The comments/remarks of C\'eline are in }{blue}\LS{, the ones of Laure in }{green}\SO{, the ones of Sarah in }{purple }\VB{ and those of Vincent in }{red}\CL{.}{}}

\section{Introduction}
\input{intro_CS.tex}

%\iffalse
\section{Homogeneity tests and multiple change-point estimation} \label{sec:homgeneity_tests}

\subsection{Two-sample homogeneity test} \label{stat2}
\input{stat2_L.tex}

\subsection{Multiple-sample homogeneity test} \label{statl}
\input{statL_S.tex}

%\fi

\subsection{Multiple change-point estimation} \label{sec:change-point}
\input{detection_CV.tex}

\section{Numerical experiments} \label{sec:num_exp}

\input{simus.tex}

\section{Application to real data} \label{sec:real_data}

\input{real.tex}

\newpage
\ \\

\section{Conclusion}

In this paper, we designed a novel nonparametric and fully automated method for retrieving the block boundaries of non-overlapping blocks in
large matrices modeled as symmetric matrices of random variables
having their distribution changing from one block to the other. Our approach is implemented in the R package \textsf{MuChPoint} which will be available from the
Comprehensive R Archive Network (CRAN). In the course of this study, we have
shown that our method, inspired by a generalization of nonparametric multiple sample tests to multivariate data, 
has two main features which make it very attractive.
Firstly, it is a nonparametric approach which showed very good statistical performances from a practical point of view.
Secondly, its low computational burden makes its use possible on
large Hi-C data matrices.

\section{Proofs} \label{sec:proofs}

In this section, we prove Theorems \ref{thm_OP1_1rupture},
\ref{Thm:Multiple} and Equation (\ref{eq:dyn_recurs}).
The proofs of Theorems \ref{thm_OP1_1rupture} and
\ref{Thm:Multiple} given below use technical lemmas established in Section \ref{sec:tech_lemmas}.

\input{proof2_SL.tex}
\input{proofL_SL.tex}
\input{proof_Sec2-3_CV.tex}

\section{Technical lemmas}\label{sec:tech_lemmas}
\input{lemma_SL.tex}

%\section*{Acknowledgments}

% Vincent Brault would like to thank the French National Research Agency ANR, which
% supported this research through the ABS4NGS project (ANR-11-BINF-0001-06).

\bibliographystyle{chicago}
\bibliography{Biblio}

%----------------------------------------------------------------------
\end{document}

%% file: intro_CS.tex
Detecting and localizing changes in the distribution of random variables is a major
statistical issue that arises in many fields such as
the surveillance of industrial processes, see \cite{Basseville:1993},
the detection of anomalies in internet traffic data, see
\cite{Tartakovsky:2006} and \cite{Levy:Roueff:2009}
or in molecular biology. In the latter field, several change-point detection
methods have been designed for dealing with different kinds of data such as
CNV (Copy Number Variation), see
\cite{picard:robin:lavielle:vaisse:daudin:2005,vert:bleakley:2010},
RNAseq data, see \cite{Cleynen:2013} and more recently Hi-C data which
motivated this work.

The Hi-C technology corresponds to one of the most recent chromosome conformation capture method
that has been developed to better understand the influence of the chromosomal
conformation on the cells functioning.
This technology is based on a deep sequencing approach
and provides read pairs corresponding to pairs of genomic loci that
physically interacts in the nucleus, see \cite{lieberman2009comprehensive}.
The raw measurements provided by Hi-C data are often summarized as a square matrix where each entry
at row $i$ and column $j$ stands for the total number of read pairs matching in position
$i$ and position $j$, respectively, see \cite{dixon2012topological}
for further details. 
%Positions refer here to a sequence of non-overlapping windows of equal sizes covering the genome.
Blocks of different intensities arise among this matrix, revealing
interacting genomic regions among which some have already been confirmed to host co-regulated
genes. The purpose of the statistical analysis is then to provide a fully automated and efficient
strategy to determine a decomposition of the matrix in non-overlapping blocks, which gives, as a by-product,
a list of non-overlapping interacting chromosomic regions.  It has
to be noticed that this issue has already been addressed by
\cite{levy2014two} in the particular framework where the mean of the observations changes from
one diagonal block to the other and is constant everywhere else. In
this latter work, the authors use a parametric approach based on the
maximization of the likelihood. In the following, we shall address the
case where the non-overlapping
blocks are not diagonal anymore by using a nonparametric method. Our goal will thus be to design
an efficient, nonparametric and fully automated method to find the block boundaries, also called change-points, of non-overlapping blocks in
large matrices which can be modeled as matrices of random variables
having their distribution changing from one block to the other.

To the best of our knowledge the most recent paper dealing with the nonparametric change-point estimation issue is the one of
\cite{matteson:2014}. Their approach allows them to retrieve
change-points within $n$ $K$-dimensional multivariate observations where
$K$ is fixed and $n$ may be large. It is based on the use of an
empirical divergence measure derived from the divergence measure
introduced by \cite{Szekely:Rizzo:2005}. Note that this methodology cannot be used in
our framework since we have to deal with matrices having both their
rows and columns that may be large. Another approach based on ranks has also been
proposed by \cite{lun15} in the same framework as
\cite{matteson:2014}. More precisely, the approach proposed by
\cite{lun15} consists in extending the classical Wilcoxon and
Kruskal-Wallis statistics (\cite{lehmann2006nonparametrics}) to the multivariate case.

In this paper, we propose a nonparametric change-point estimation
approach based on nonparametric homogeneity tests. More precisely, we
shall generalize the approach of \cite{lun15} to the case where we
have to deal with large matrices instead of fixed multidimensional vectors.

The paper is organized as follows. We first propose in Sections
\ref{stat2} and \ref{statl} nonparametric homogeneity
tests for two-samples and several samples, respectively. In Section
\ref{sec:change-point}, we deduce from these tests a nonparametric procedure for
estimating the block boundaries of a matrix of random variables
having their distribution changing from one block to the other. These
methodologies are then illustrated by
some numerical experiments in Section \ref{sec:num_exp}. An
application to real Hi-C data is also given in Section
\ref{sec:real_data}.
Finally, the proofs of our theoretical results are given in Section \ref{sec:proofs}.

% \paragraph{our specific context and contibution}
% \textit{Nonparametric setting}\bigskip

% \textit{Homogeneity tests}\\
% We extend/adapt the multivariate rank statistic introduced in Lung-Yut-Fong et al. ?, for random variables in $\mathbb{R}^{K}$ with $K$ being fixed to our case where the space dimension $K$ equals the obervations number $n$.\medskip

% We propose a nonparametric homogeneity test which can be seen as (in some way is) a generalization of the Wilcoxon rank test and study its asymptotic properties.\bigskip

% \textit{Change-point detection}\\
% We propose an automatic approach for estimating the boundaries of these homogeneous blocks.\medskip

% We provide an automatic and efficient estimation method of the limits of these blocks.\bigskip

% \paragraph{bibliography in our specific context}

% \textit{Change-point detection}\\
% Dixon et al. ?, L\'evy-Leduc et al. ?, propose block detection methods on the diagonal. Our work aim at retrieving the boundaries by taking the extra-diagonal interactions into account. Besides our statistical framework is a nonparametric one.\bigskip

%%% Local Variables:
%%% mode: latex
%%% eval: (TeX-PDF-mode 1)
%%% TeX-master: "Test_HiC.tex"
%%% ispell-local-dictionary: "en_US"
%%% eval: (flyspell-mode 1)
%%% End:

%% file: stat2_L.tex
\subsubsection{Statistical framework}\label{sec:stat_fram_two_sample}

Let $\textbf{X}=(X_{i,j})_{1\leq i,j\leq n}$ be a symmetric matrix such that the $X_{i,j}$'s are independent random variables
when $i \geq j$. 
Observe that $\textbf{X}$ can be rewritten as follows: $\textbf{X}=(\textbf{X}^{(1)},\ldots,\textbf{X}^{(n)})$, where $\textbf{X}^{(j)}=(X_{1,j},\ldots,X_{n,j})'$
denotes the $j$th column of $\textbf{X}$.

Let $n_1$ be a given integer in $\{1,\dots,n\}$. The goal of this section is to propose a statistic to test the null hypothesis
$(H_0)$: ``$(\textbf{X}^{(1)},\ldots,\textbf{X}^{(n_1)})$ and $(\textbf{X}^{(n_1+1)},\ldots,\textbf{X}^{(n)})$ are identically distributed random vectors''
against the alternative hypothesis $(H_1)$:
``$(\textbf{X}^{(1)},\ldots,\textbf{X}^{(n_1)})$ has the distribution $\P_1$ and 
$(\textbf{X}^{(n_1+1)},\ldots,\textbf{X}^{(n)})$ has the distribution $\P_2$, where $\P_1 \neq \P_2$''.
Note that the hypotheses $(H_0)$ and $(H_1)$ can be reformulated as
follows. The null hypothesis $(H_0)$ means that for all
$i\in\{1,\ldots,n\}$, $X_{i,1},\ldots,X_{i,n}$ are 
independent and identically distributed (i.i.d) random variables and
the alternative hypothesis $(H_1)$ means that there exists
$i\in\{1,\ldots,n\}$ such that $X_{i,1},\ldots,X_{i,n_1}$ have the
distribution $\P^{i}_1$ and $X_{i,n_1+1},\ldots,X_{i,n}$ have the distribution $\P^i_2$, with $\P^i_1\neq \P^i_2$.

For deciding whether $(H_0)$ has to be rejected or not, we propose to use a test statistic inspired by the one designed by \cite{lun15}
which extends the well-known Wilcoxon-Mann-Whitney rank-based test to deal with multivariate data.
Our statistical test can thus be seen as a way to decide whether $n_1$ can be considered as a potential change in the distribution of the $\Xij$'s or not.
% We propose a homogeneity test for $\textbf{X}$ splitted in two groups: $(\textbf{X}^{(1)},\ldots,\textbf{X}^{(n_1)})$ and $(\textbf{X}^{(n_1+1)},\ldots,\textbf{X}^{(n)})$, where the potential change-point $n_1$ is assumed to be given and $\textbf{X}^{(j)}=(X_{1,j},\ldots,X_{n,j})'$ is the $j$-th column of $\textbf{X}$. For this classical statistical test framework, the null hypothesis is
% $(H_0)$: "$(\textbf{X}^{(1)},\ldots,\textbf{X}^{(n)})$ are identically distributed random vectors"
% and the alternative hypothesis is
% $(H_1)$: "$(\textbf{X}^{(1)},\ldots,\textbf{X}^{(n_1)})$ are distributed under $\P_1$ and $(\textbf{X}^{(n_1+1)},\ldots,\textbf{X}^{(n)})$ under $\P_2$, with $\P_1 \neq \P_2$".
%
% \begin{rem}
% We can also formulate the test hypotheses in the following way. The null hypothesis $(H_0)$ means that for all $i\in\{1,\ldots,n\}$, $X_{i,1},\ldots,X_{i,n}$ are identically distributed random variables and the alternative hypothesis $(H_1)$ means that there exists $i\in\{1,\ldots,n\}$ such that $X_{i,1},\ldots,X_{i,n_1}$ are distributed under $\P^{i}_1$ and $X_{i,n_1+1},\ldots,X_{i,n}$ under $\P^i_2$, with $\P^i_1\neq \P^i_2$.
% \end{rem}
%
% We introduce a test statistic inspired by the one proposed by \cite{lun15} and which extends the well-known Wilcoxon-Mann-Whitney rank-based criterion to multivariate data.
More precisely, the test statistic that we propose for assessing the presence of the potential change $n_1$ is defined by
\begin{equation} \label{def:sn1}
S_n(n_1) = \sum_{i=1}^n U_{n,i}^2(n_1),
\end{equation}
where
\[U_{n,i}(n_1) = \frac{1}{\sqrt{nn_1(n-n_1)}} \sum_{j_0=1}^{n_1}\sum_{j_1=n_1+1}^{n} h(X_{i,j_0},X_{i,j_1}),\]
with $h(x,y)=\1_{\{x \leq y\}}-\1_{\{y \leq x\}}$.

The great difference between our framework and the one considered by
\cite{lun15} is that, in their framework, the vectors $\textbf{X}^{(j)}$ are $K$-dimensional with $K$ fixed
whereas, in our framework, the vectors are $n$-dimensional where $n$ may
be large. 
%Hence, in our framework, the $U_{n,i}(n_1)$ are $n$-dimensional vectors.
% \begin{rem}
% We stress that the multivariate observations in our $U_{n,i}(n_1)$ are $n$-dimensional where $n$ may be large (see the asymptotical results hereafter) whereas in \cite{lun15} they are $K$-dimensional with $K$ fixed. 
% \end{rem}

Note that the statistic $U_{n,i}$ can also be written by using the rank of $X_{i,j}$ among $\left(X_{i,1},\ldots,X_{i,n}\right)$. Indeed,
\begin{equation}\label{eq:Uni_rank}
U_{n,i}(n_1) = \frac{2}{\sqrt{nn_1(n-n_1)}} \sum_{j_0=1}^{n_1}\left(\frac{n+1}{2}-R_{j_0}^{(i)}\right)=\frac{2}{\sqrt{nn_1(n-n_1)}} \sum_{j_1=n_1+1}^{n}\left(R_{j_1}^{(i)}-\frac{n+1}{2}\right),
\end{equation}
where
\begin{equation}\label{rank}
\Rji=\sum_{k=1}^{n}\1_{\{X_{i,k}\leq X_{i,j}\}}
\end{equation}
is the rank of $X_{i,j}$ among $\left(X_{i,1},\ldots,X_{i,n}\right)$. 
This alternative form of $U_{n,i}$ will be used % to define the generalization of $S_n(n_1)$ in the multiple change-point case
in Section~\ref{statl} in order to extend the two-sample homogeneity
test to deal with the multiple sample case.

\subsubsection{Theoretical results}

% From now on, we shall assume that the cumulative distribution function
% of the $X_{i,j}$'s is continuous. Under this assumption, the following
% theorem states that the test statistic $S_n(n_1)$ is
% bounded in probability.

If the cumulative distribution function of the $X_{i,j}$'s is assumed to be continuous
then the following theorem establishes that the test statistic $S_n(n_1)$ is properly normalized,
namely $S_n(n_1)$ is bounded in probability as $n$ tends to infinity.
\begin{thm}\label{thm_OP1_1rupture}
Let $\textbf{X}=(X_{i,j})_{1\leq i,j\leq n}$ be a symmetric matrix of
random variables $X_{i,j}$ such that the $X_{i,j}$'s are i.i.d. when $i \geq j$.
Assume that the cumulative distribution function of the
$X_{i,j}$'s is continuous and that there exists $\tau_1 \in (0,1)$ such that $n_1/n \rightarrow \tau_1$ as $n\to\infty$.
Then,
\begin{equation*}%\label{eq:T_n1}
T_n(n_1):=n^{-1/2}\big(S_n(n_1)-\E(S_n(n_1)\big)=O_P(1) \mbox{ as }
n\to\infty,
\end{equation*}
%$n^{-1/2}\left(S_n(n_1)-\E(S_n(n_1))\right)$ is bounded in probability as $n\to\infty$
where
\[\E(S_n(n_1))=\frac{n+1}{3}.\]
\end{thm}
The proof of Theorem \ref{thm_OP1_1rupture} is given in Section~\ref{proof2}.

Observe that the assumptions under which Theorem
\ref{thm_OP1_1rupture} is established correspond to the null
hypothesis $(H_0)$ described in Section
\ref{sec:stat_fram_two_sample}. Hence, we shall reject this null
hypothesis when 
\begin{equation}\label{eq:reject_region}
T_n(n_1)>s,
\end{equation}
where $s$ is a threshold. A way of computing this
threshold in practical situations will be given in
Section \ref{sec:pract_calibration}.

%%% Local Variables:
%%% mode: latex
%%% eval: (TeX-PDF-mode 1)
%%% TeX-master: "Test_HiC.tex"
%%% ispell-local-dictionary: "en_US"
%%% eval: (flyspell-mode 1)
%%% End:

%% file: statL_S.tex
The goal of this section is to extend the two-sample homogeneity test of the previous section to deal with the multiple sample case.

%In this section, we extend the two-sample homogeneity test established previously to a multiple-sample one.

\subsubsection{Statistical framework}

Let us assume that $\textbf{X}=(X_{i,j})_{1\leq i,j\leq n}$ is still a symmetric matrix 
such that the $X_{i,j}$'s are independent random variables when $i \geq j$. Let $0=n_0<n_1<\ldots<n_L<n_{L+1}=n$ be $L$
integers given in $\{1,\dots,n-1\}$. We propose in this section a statistic to test the null hypothesis:
``$(\textbf{X}^{(1)},\ldots,\textbf{X}^{(n_1)})$, $(\textbf{X}^{(n_1+1)},\ldots,\textbf{X}^{(n_2)})$, \ldots,$(\textbf{X}^{(n_L+1)},\ldots,\textbf{X}^{(n)})$ have 
the same distribution'' against the alternative hypothesis: 
``there exists $\ell\in\{1,\ldots,L\}$ such that
$(\textbf{X}^{(n_{\ell-1}+1)},\ldots,\textbf{X}^{(n_\ell)})$ has the
distribution $\P_\ell$ and
$(\textbf{X}^{(n_\ell+1)},\ldots,\textbf{X}^{(n_{\ell+1})})$ has the
distribution $\P_{\ell+1}$, where $\P_\ell \neq \P_{\ell+1}$''.

% We present a homogeneity test for $\bX$, the matrix we considered in the previous section, splitted now in $L+1$ groups, where the $L$ potential change points $0=n_0<n_1<\ldots<n_L<n_{L+1}=n$ are assumed to be given.
% The null hypothesis is
% $(H_0)$: "$(\textbf{X}^{(1)},\ldots,\textbf{X}^{(n)})$ are identically distributed random vectors" and the alternative hypothesis is
% $(H_1)$: "there exists $\ell\in\{1,\ldots,L\}$ such that $(\textbf{X}^{(n_{\ell-1}+1)},\ldots,\textbf{X}^{(n_\ell)})$ are distributed under $\P_l$ and $(\textbf{X}^{(n_\ell+1)},\ldots,\textbf{X}^{(n_{\ell+1})})$ under $\P_{\ell+1}$, with $\P_\ell \neq \P_{\ell+1}$".

The homogeneity test presented in the previous section for two groups can be extended in order to deal with $L+1$ groups instead of two by using the 
following statistic:
%\noindent The previous statistic of the two-sample homogeneity test is extended to the following one which deal with more than one change point in $\bX$.
\begin{eqnarray}\label{Def:Stat}
\SnunL=
\frac{4}{\n^2}
\sum_{\ell=0}^{L}\left(\nlun-\nl\right)
\sum_{i=1}^{\n}\left(\Rbarli-\frac{\n+1}{2}\right)^2,
\end{eqnarray}
with
\begin{eqnarray}\label{Def:Rang}
\Rbarli&=&\frac{1}{\nlun-\nl}\sum_{j=\nl+1}^{\nlun}\Rji,
\end{eqnarray}
where the rank $\Rji$ of $X_{i,j}$ is defined by \eqref{rank} and $\Rbarli$ is its mean in the group $\ell$.

Let us observe that (\ref{Def:Stat}) can be seen as a natural extension of the classical Kruskal-Wallis statistic for univariate observations
to the multivariate case, see \cite[p. 181]{vdV2000}.

\begin{rem}
Note that when $L=1$, $S_n(n_1)$ defined in (\ref{Def:Stat}) boils down to  $S_n(n_1)$ defined in (\ref{def:sn1}) since
\begin{align*}
&\frac{4}{\n^2}\left[n_1\sum_{i=1}^{\n}\left(\frac{1}{n_1}\sum_{j=1}^{n_1}\Rji-\frac{\n+1}{2}\right)^2
+(n-n_1)\sum_{i=1}^{\n}\left(\frac{1}{n-n_1}\sum_{j=n_1+1}^{n}\Rji-\frac{\n+1}{2}\right)^2\right]\\
&=\frac{4}{\n^2 n_1}\sum_{i=1}^{\n}\left\{\sum_{j=1}^{n_1}\left(\Rji-\frac{\n+1}{2}\right)\right\}^2
+\frac{4}{\n^2 (n-n_1)}\sum_{i=1}^{\n}\left\{\sum_{j=n_1+1}^{n}\left(\Rji-\frac{\n+1}{2}\right)\right\}^2\\
&=\frac1n\left[\sum_{i=1}^{\n}n_1\left\{\frac{1}{\sqrt{n n_1 (n-n_1)}}\sum_{j=1}^{n_1}\left(\Rji-\frac{\n+1}{2}\right)\right\}^2\right.\\
&\left.+\sum_{i=1}^{\n}(n-n_1)\left\{\frac{1}{\sqrt{n n_1 (n-n_1)}}\sum_{j=n_1+1}^{n}\left(\Rji-\frac{\n+1}{2}\right)\right\}^2\right]
=\sum_{i=1}^n U_{n,i}^2(n_1),
\end{align*}
by using (\ref{eq:Uni_rank}), which corresponds to (\ref{def:sn1}).
\end{rem}

\subsubsection{Theoretical results}

If the cumulative distribution function of the $X_{i,j}$'s is assumed to be continuous
then the following theorem establishes that the test statistic $\SnunL$ is properly normalized,
namely $\SnunL$ is bounded in probability as $n$ tends to infinity.

\begin{thm}\label{Thm:Multiple}
Let $\textbf{X}=(X_{i,j})_{1\leq i,j\leq n}$ be a symmetric matrix of
random variables $X_{i,j}$ such that the $X_{i,j}$'s are i.i.d when $i \geq j$.
Assume that the cumulative distribution function of the $X_{i,j}$'s is continuous and that there exist $0<\tau_1<\tau_2<\ldots<\tau_L<1$ 
such that for all $\ell\in\{1,\ldots,L\}$, $\nl/n \rightarrow \tau_\ell$ as $n\to\infty$.
Then,
\[n^{-1/2}\big(\SnunL-\Esp{\SnunL}\big)=O_P(1) \mbox{ as } n\to\infty,\]
with
\[\Esp{\SnunL}=\frac{L(\n+1)}{3}.\]
\end{thm}

The proof of Theorem~\ref{Thm:Multiple} is given in Section~\ref{proofL}

Note that the $n_\ell$'s can be seen as the boundaries of groups of random variables having different distributions.
We shall explain in the next section how to derive from this theorem a methodology for estimating the $n_\ell$'s when they are assumed to be unknown.

%%% Local Variables:
%%% mode: latex
%%% eval: (TeX-PDF-mode 1)
%%% TeX-master: "Test_HiC.tex"
%%% ispell-local-dictionary: "en_US"
%%% eval: (flyspell-mode 1)
%%% End:

%% file: detection_CV.tex
We propose in this section to use the test statistic (\ref{Def:Stat}) defined in
Section \ref{statl} to derive the location of the block
boundaries $n_1<n_2<\cdots<n_L$.
More precisely, we propose to estimate $(n_1,n_2,\cdots,n_L)$ as follows:
\begin{equation}\label{eq:max_chge_point}
(\widehat{n}_1,\cdots,\widehat{n}_L):=\textrm{Argmax}_{0<n_1<\ldots<n_L<n}\; S_n(n_1,\ldots,n_L),
\end{equation}
where $S_n(n_1,\ldots,n_L)$ is defined in (\ref{Def:Stat}).

In practice, directly maximizing (\ref{eq:max_chge_point}) is
computationally prohibitive as it corresponds to a task which
complexity exponentially grows with $L$. However, thanks to the
additive structure of (\ref{Def:Stat}), it is possible to use a
dynamic programming strategy as we shall explain hereafter. We refer here to the classical dynamic programming approach 
described in \cite{kay:1993} which can be traced back to
the note of \cite{Bellman:1961}. 

Let us introduce  the following notations
$$
\Delta(n_{\ell}+1:n_{\ell+1})=(n_{\ell+1}-n_\ell)\sum_{i=1}^{\n}\left(\Rbarli-\frac{\n+1}{2}\right)^2,
$$ 
where $\Rbarli$ is defined by (\ref{Def:Rang}) and
\begin{equation}\label{eq:dyn:def}
I_{L}(p)=\max_{1<n_1<\dots<n_{L}<n_{L+1}=p}\sum_{\ell=0}^{L}\Delta(n_{\ell}+1:n_{\ell+1}), 
\end{equation}
for $L\in\{0,1,\dots,L_{\textrm{max}}\}$ and $p\in\{2,\ldots,n\}$,
where $L_{\textrm{max}}$ is assumed to be a known upper bound for the number of block boundaries.
Observe that $I_L(p)$ satisfies the following recursive formula:
\begin{equation}\label{eq:dyn_recurs}
I_{L}(p)=\max_{n_{L}}\left\{I_{L-1}(n_{L})+\Delta(n_{L}+1:p)\right\}\;,
\end{equation}
which is proved  in Section \ref{proof_Sec.2.3}.
% \LS{}{Si on définit $\Delta$ avant, je propose, à la place de "More precisely ...", qch du genre :
% For using this strategy, we need to deal with the quantity $I_{L}(p)$ for any $L$ in ??? and $p$ in $\{2,\ldots,n\}$, defined as
% \begin{equation}\label{eq:dyn:def}
% I_{L}(p)=\max_{1<n_1<\dots<n_{L}<n_{L+1}=p}\sum_{\ell=0}^{L}\Delta(n_{\ell}+1:n_{\ell+1})\;.
% \end{equation}
% \begin{prop} \label{prop:recform}
% For the implementation of the dynamic programming approach, we use the recursive form of $I_{L}(p)$ which is
% \[I_{L}(p)=\max_{n_{L}}\left\{I_{L-1}(n_{L})+\Delta(n_{L}+1:p)\right\}\;.\]
% \end{prop}
% The proof of Proposition~\ref{prop:recform} is given in Subsection~\ref{proof_Sec.2.3}.}
Thus, for solving the optimization problem (\ref{eq:max_chge_point}), we
proceed as follows. We start by computing the $\Delta(i:j)$ for all
$(i,j)$ such that
$1\leq i<j\leq n$. All the $I_0(p)$ are thus
available for $p=2,\dots,n$. Then $I_1(p)$ is computed by using
the recursion (\ref{eq:dyn_recurs}) and so on. Hence the complexity of
our algorithm is $O(n^3)$.
%The overall numerical
%complexity of the procedure is thus proportional to $L \times n^2$ only.

Figure \ref{fig:time} displays the computational times in seconds
associated with
our multiple change-point estimation strategy based on the dynamic programming
algorithm. We observe from this figure the polynomial computational time
of our procedure. For instance, it takes 15
minutes to our algorithm for processing a
$500\times 500$ matrix. 

\begin{figure}[!h]
\begin{center}
\includegraphics*[height=8cm,width=12cm]{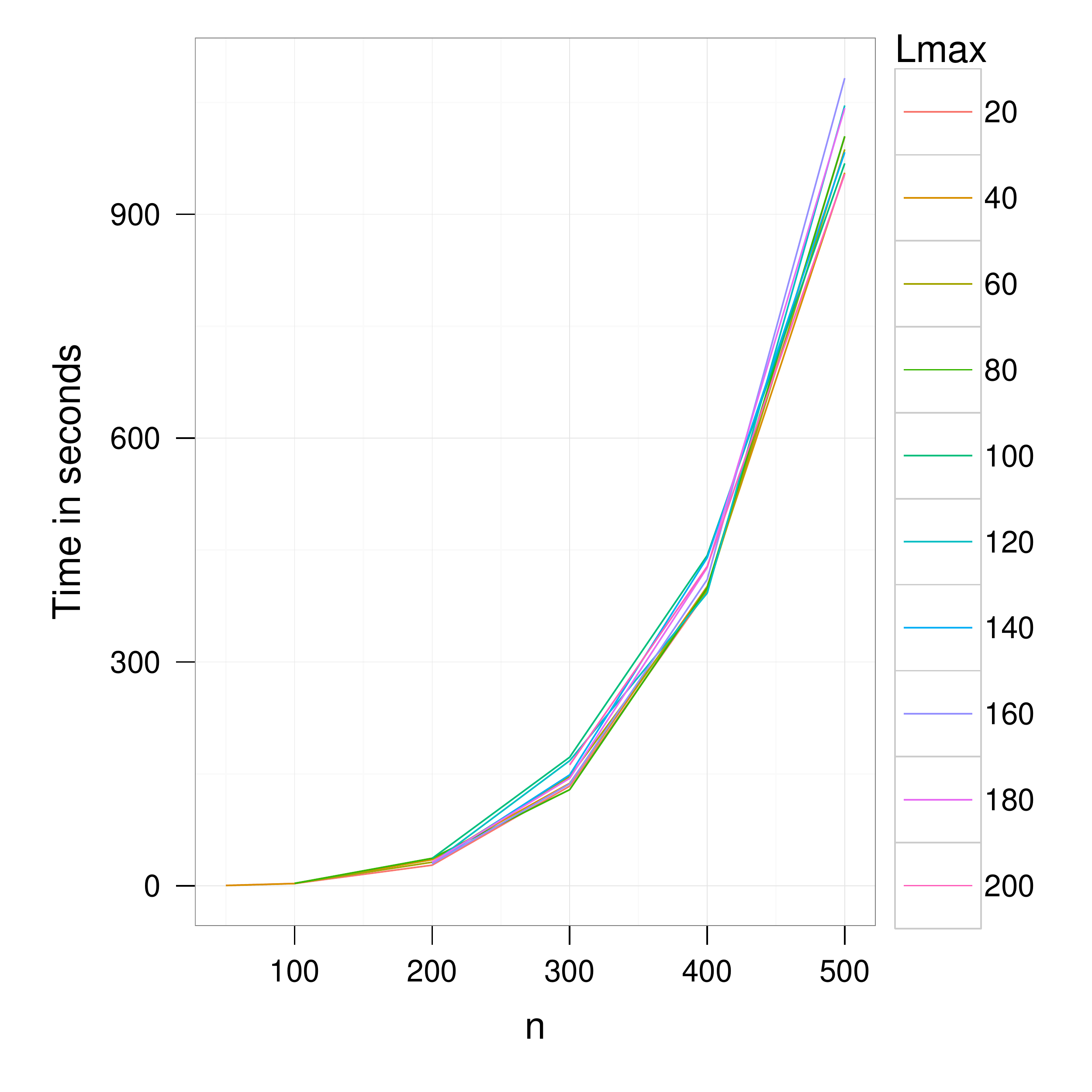}
\caption{Computational times in seconds for the dynamic programming algorithm
  described in Section \ref{sec:change-point} as a function of
  $n$ for different values of $L_{\textrm{max}}$.\label{fig:time}}
\end{center}
\end{figure}

%%% Local Variables:
%%% mode: latex
%%% eval: (TeX-PDF-mode 1)
%%% TeX-master: "Test_HiC.tex"
%%% ispell-local-dictionary: "en_US"
%%% eval: (flyspell-mode 1)
%%% End:

%% file: simus.tex
\subsection{Statistical performance of the two-sample homogeneity
  test}

\subsubsection{Practical calibration of the rejection region}\label{sec:pract_calibration}

We propose hereafter a procedure for calibrating the threshold $s$ of
the rejection region $T_n(n_1)>s$ defined in (\ref{eq:reject_region}).
For ensuring that the two-sample homogeneity test is of level
$\alpha$, an estimation of the $(1-\alpha)$ quantile of $T_n(n_1)$ has
to be provided. In the sequel, such an estimation is given in the case where $\alpha=0.05$.

We generated $10^4$
$n\times n$ symmetric matrices $\boldsymbol{X}=(\Xij)$ with $\n\in\{50,100,500,1000\}$. More precisely,
the $(\Xij)_{i\geq j}$'s are independent random variables
distributed as a zero mean standard Gaussian distribution
($\mathcal{N}(0,1)$), a Cauchy distribution with 0 and 1 location
and scale parameters ($\mathcal{C}au(0,1)$), respectively or an
Exponential distribution of parameter 2 ($\mathcal{E}xp(2)$). We shall
consider two values for $n_1$: $n_1=\lfloor 0.1 n\rfloor$
 and $n_1=\lfloor 0.5 n\rfloor$, where $\lfloor x\rfloor$ denotes the
 integer part of $x$.

The  empirical $0.95$ quantiles of
$T_n(n_1)$ are given in Table \ref{Table:Quantiles}. 
We observe from this table that the empirical $0.95$ quantiles 
do not seem to be sensitive
neither to the values of $n_1$ and $n$ nor to the distribution of the
observations since they slightly vary around 0.8.

\input{Tableau.tex}

%Parler de la methode bootstrap.
% \input{TableauBootstrap.tex}

\subsubsection{Power of the test statistic}

In this section, we study the power of the two-sample homogeneity test
defined in Section \ref{sec:stat_fram_two_sample}.
We generated $10^4$
$n\times n$ symmetric matrices $\boldsymbol{X}=(\Xij)$ split into four blocks defined as follows and
$\n\in\{50,100,500,1000\}$. 
Let
$$
\mathcal{I}_1=\{(i,j): 1\leq j\leq i \leq n_1\},\;\;
\mathcal{I}_2=\{(i,j): 1\leq j\leq n_1,\; n_1+1\leq i \leq n\},
$$
and
$$  
\mathcal{I}_3=\{(i,j): n_1+1\leq j\leq i \leq n\}.
$$
In the sequel, we assume that $(\Xij)_{(i,j)\in\mathcal{I}_1}\stackrel{iid}{\sim}\mathcal{L}_1$, $(\Xij)_{(i,j)\in\mathcal{I}_2}\stackrel{iid}{\sim}\mathcal{L}_2$ 
and $(\Xij)_{(i,j)\in\mathcal{I}_3}\stackrel{iid}{\sim}\mathcal{L}_3$
and we take the following values for $n_1$: $n_1=\lfloor 0.1 n\rfloor$
 and $n_1=\lfloor 0.5 n\rfloor$.

Figure \ref{fig:power} displays the power curves of the two-sample
homogeneity test defined in Section \ref{sec:stat_fram_two_sample} in
the case where $\mathcal{L}_1=\mathcal{L}_3=\mathcal{N}(0,1)$ and 
$\mathcal{L}_2=\mathcal{N}(\mu,1)$ where $\mu$ belongs to the set 
$\{0,0.01,0.02,\dots,0.99,1\}$.

We can see from this figure that for large values of $n$ our testing procedure appears to be powerful whatever the value of $\mu$.
For small values of $n$, we observe that our testing procedure is all the more powerful that $\mu$ is large.

\begin{figure}[!h]
\begin{center}
\begin{tabular}{cc}
\includegraphics*[height=7cm,width=8cm]{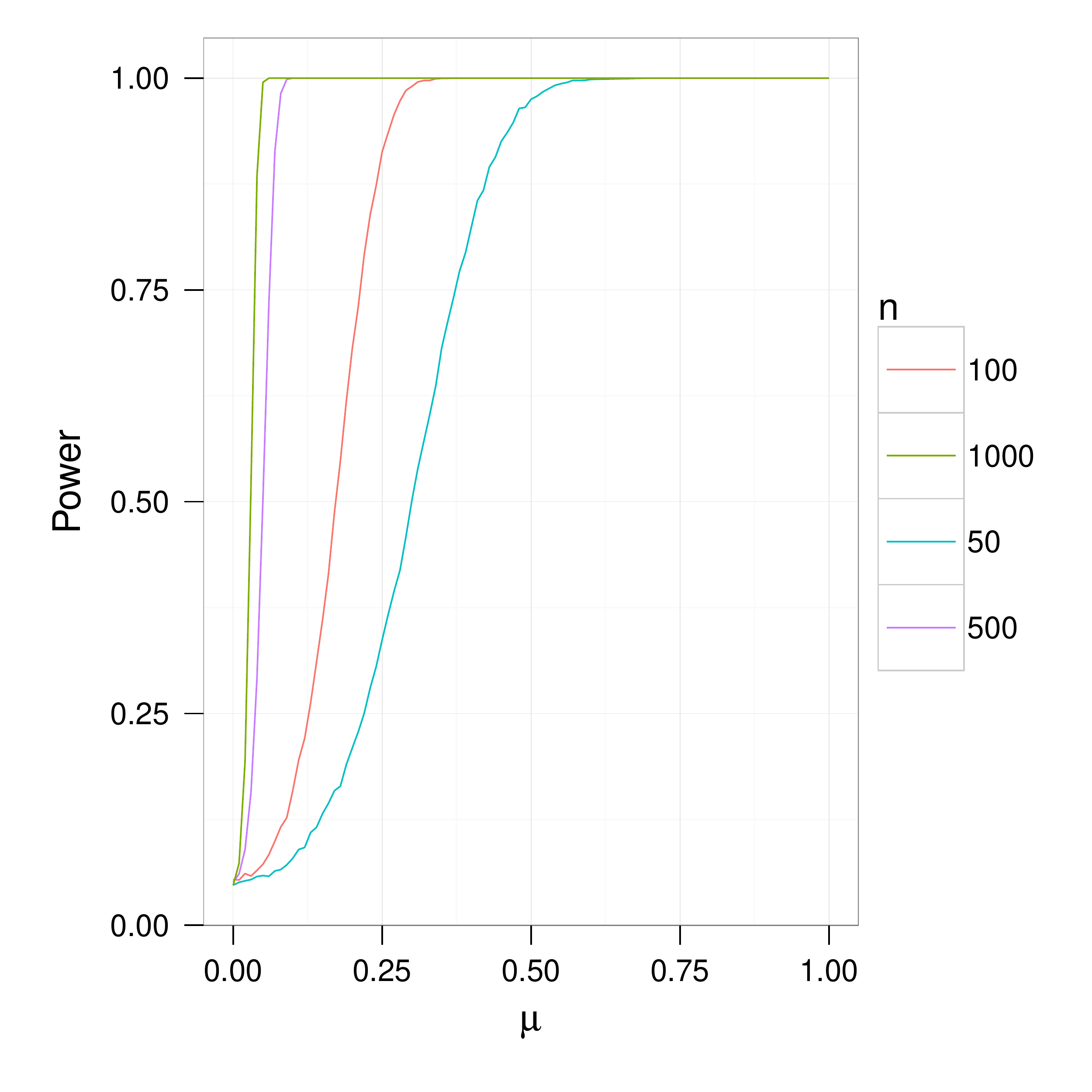}&
\hspace{-9mm}
\includegraphics*[height=7cm,width=8cm]{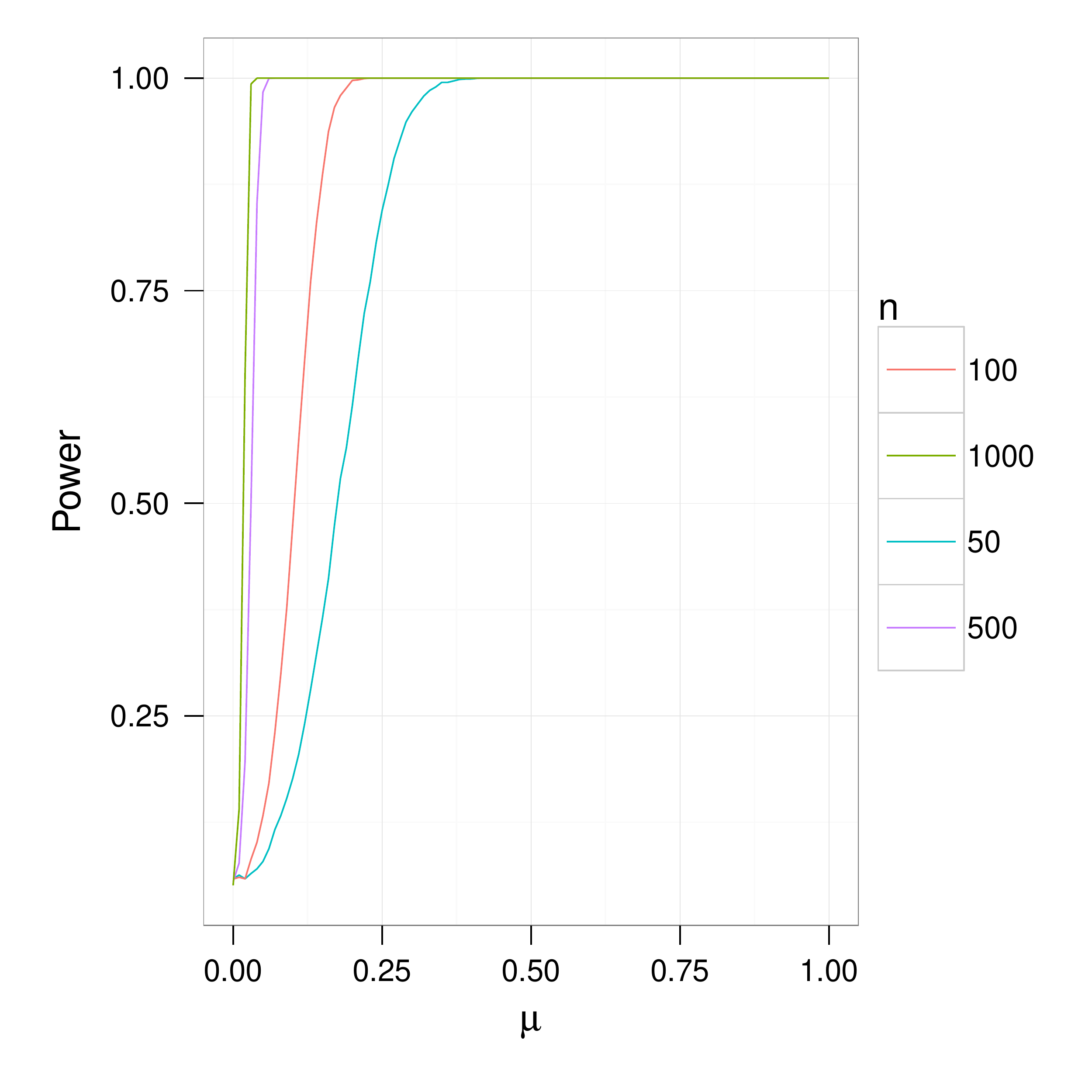}
\end{tabular}
\caption{Power curves for the two-sample homogeneity test as a function
  of $\mu$ for different values of $n$,  $n_1=\lfloor 0.1 n\rfloor$
  (left) and $n_1=\lfloor 0.5 n\rfloor$
  (right).\label{fig:power}}
\end{center}
\end{figure}

\subsection{Statistical performance of the multiple change-point
  estimation procedure}

In this section, we study the statistical performance of the multiple change-point
  estimation procedure described in Section
  \ref{sec:change-point}. This method is implemented in the R package
  \textsf{MuChPoint}, which will be available on the Comprehensive R
  Archive Network (CRAN).

We generated 10000 $n\times n$ symmetric matrices $\textbf{X}=(X_{i,j})$ where
  $\n\in\{50,100,200,300,400\}$ with different block configurations
  and $L=10$ block boundaries (change-points).

We shall first consider the \textsf{Block Diagonal} configuration. In this case, the matrix consists
of diagonal blocks of size $n/10$. Within each of these diagonal blocks, the
$\Xij$'s such that $i\geq j$ are independent and have the distribution $\mathcal{L}_1$. The $\Xij$'s lying in the extra-diagonal
part of the lower triangular part of $\textbf{X}$ are independent and
have the distribution $\mathcal{L}_2$, which is assumed to be
different from $\mathcal{L}_1$. The upper triangular part of
$\textbf{X}$ is then derived by symmetry.

We shall also consider the \textsf{Chessboard} configuration. In this
case, the matrix consists of non overlapping blocks of size $n/10$.
The $\Xij$'s belonging to two blocks
sharing a boundary have different distributions. 
This configuration implies that only two distributions $\mathcal{L}_1$
and $\mathcal{L}_2$ are at stake.
The distribution of the upper left block is denoted by $\mathcal{L}_1$
in the sequel.

For these two configurations, we shall consider for $\mathcal{L}_1$ a
$\mathcal{N}(1,\sigma^2)$, a $\mathcal{E}xp(2)$ or a $\mathcal{C}au(1,a)$
distribution where $\sigma$ and $a$ are in $\{1,2,5\}$. The $\mathcal{L}_2$ distributions associated with each of
them are $\mathcal{N}(0,\sigma^2)$, $\mathcal{E}xp(\lambda)$ and
$\mathcal{C}au(0,a)$ where $\lambda\in\{1,0.5,4\}$.
We display in Figure \ref{fig:examples} some examples of the \textsf{Block Diagonal} and
\textsf{Chessboard} configurations for the Gaussian, Exponential and
Cauchy distributions. In these plots, large values are
displayed in red and small values in blue.

\begin{figure}[!h]
\begin{center}
\begin{tabular}{ccc}
\includegraphics*[height=4cm,width=4.5cm]{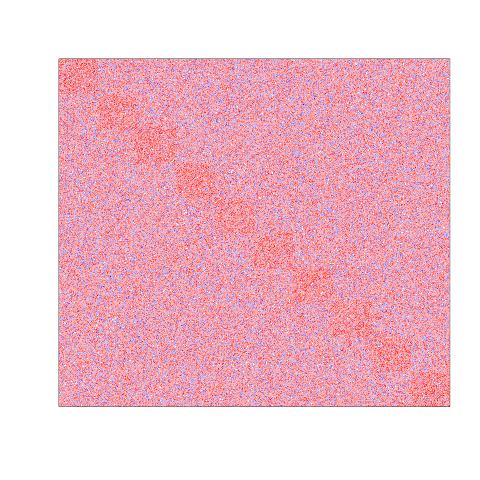}
&\hspace{-48mm}\includegraphics*[height=4cm,width=4.5cm]{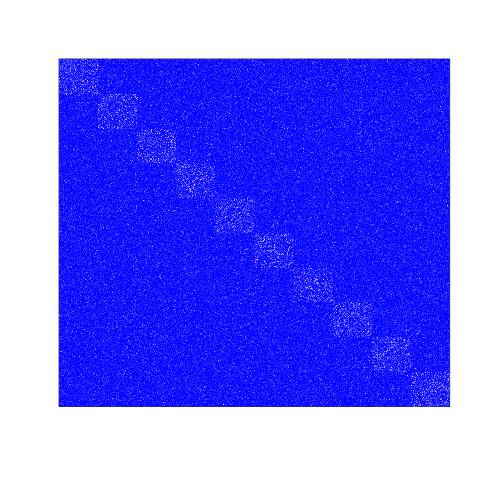}
&\hspace{-50mm}\includegraphics*[height=4cm,width=4.5cm]{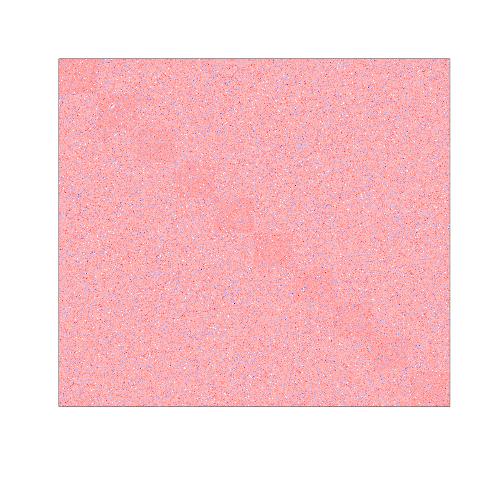}\\
\includegraphics*[height=4cm,width=4.5cm]{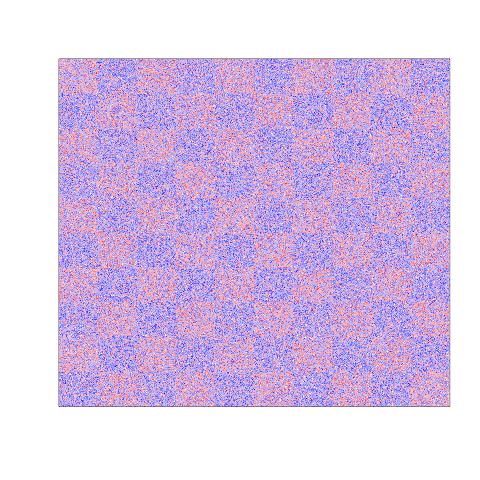}
&\includegraphics*[height=4cm,width=4.5cm]{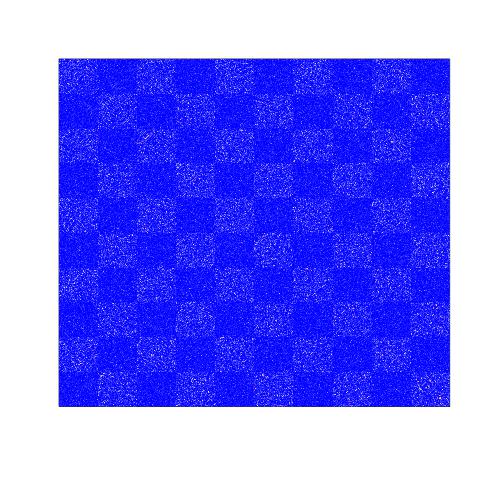}
\includegraphics*[height=4cm,width=4.5cm]{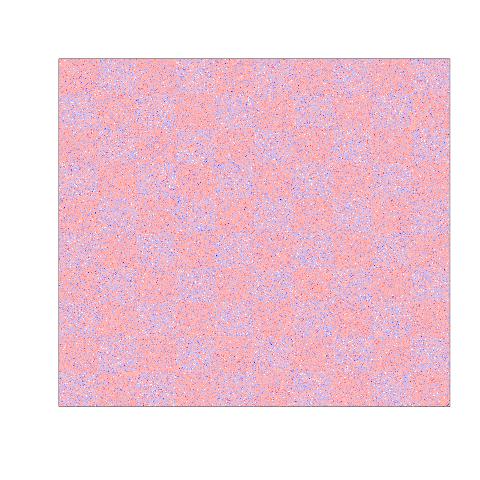}
\end{tabular}
\caption{Examples of $400\times 400$ matrices $\textbf{X}$. Top: \textsf{Block
    Diagonal} configuration. Bottom: \textsf{Chessboard}
  configuration. Left: $\mathcal{L}_1=\mathcal{N}(1,4)$,
  $\mathcal{L}_2=\mathcal{N}(0,4)$, middle:
  $\mathcal{L}_1=\mathcal{E}xp(2)$, $\mathcal{L}_2=\mathcal{E}xp(1)$ 
and right: $\mathcal{L}_1=\mathcal{C}au(1,1)$,
$\mathcal{L}_2=\mathcal{C}au(0,1)$. \label{fig:examples}}
\end{center}
\end{figure}

In the Gaussian \textsf{Chessboard} configuration, Figure \ref{fig:barplots} displays the frequency of the number of
times where each position in $\{1,\dots,n-1\}$ has been estimated as a
change-point. We can see from this figure that the true change-point
positions are in general properly retrieved by our approach even in
cases where the change-points are not easy to detect with the naked eye. However, we observe that
in the cases where $\sigma$ increases, some spurious
change-points appear close to the true change-point positions.
% Since we obtained similar results for the other configuration and the
% other distributions, the corresponding barplots are not displayed.
% They are available upon request.

\begin{figure}[!h]
\begin{center}
\begin{tabular}{ccc}
$\sigma=1$ & $\sigma=2$ & $\sigma=5$\\
\includegraphics*[height=4cm,width=4.5cm]{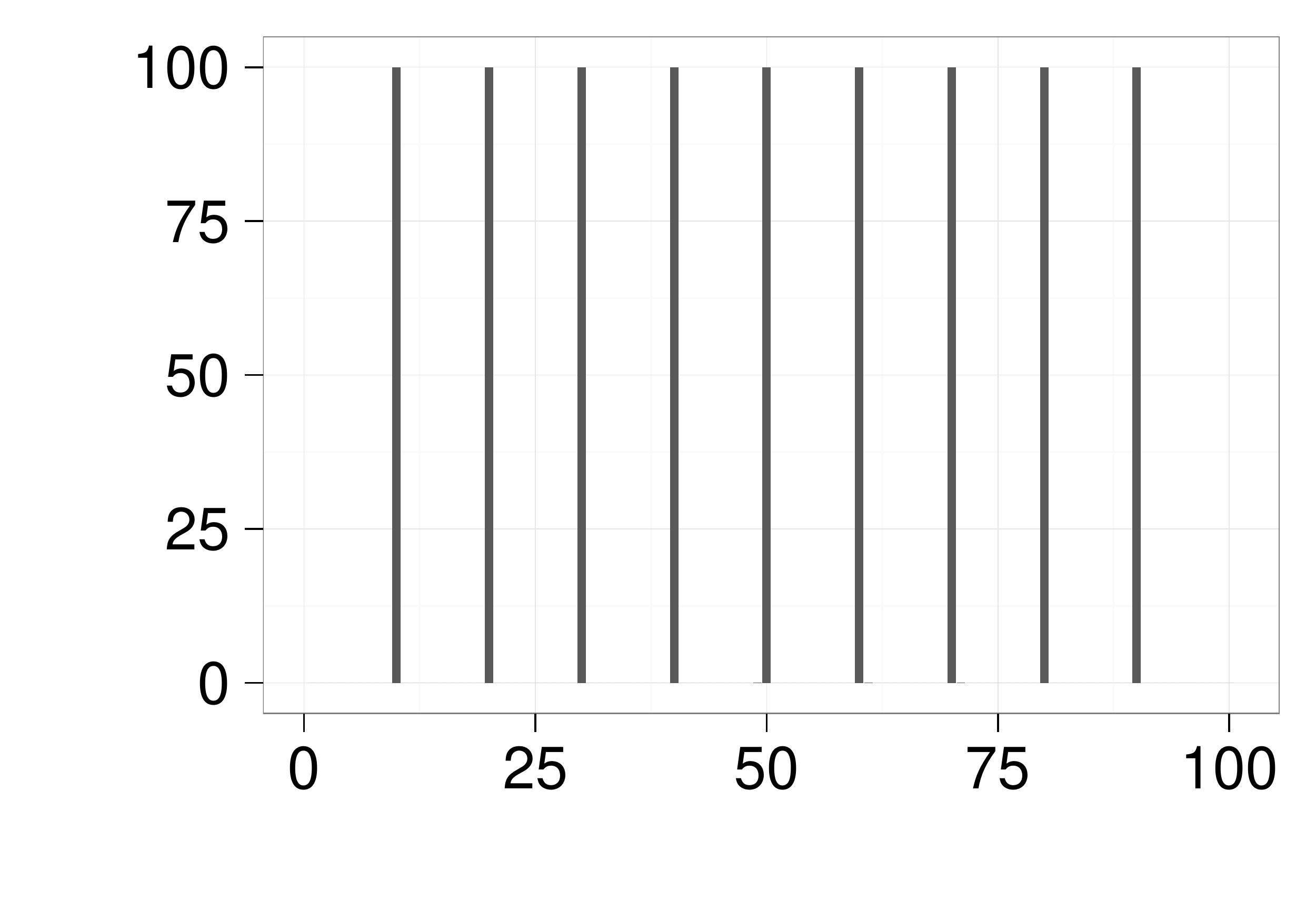}&
\hspace{-5mm}
\includegraphics*[height=4cm,width=4.5cm]{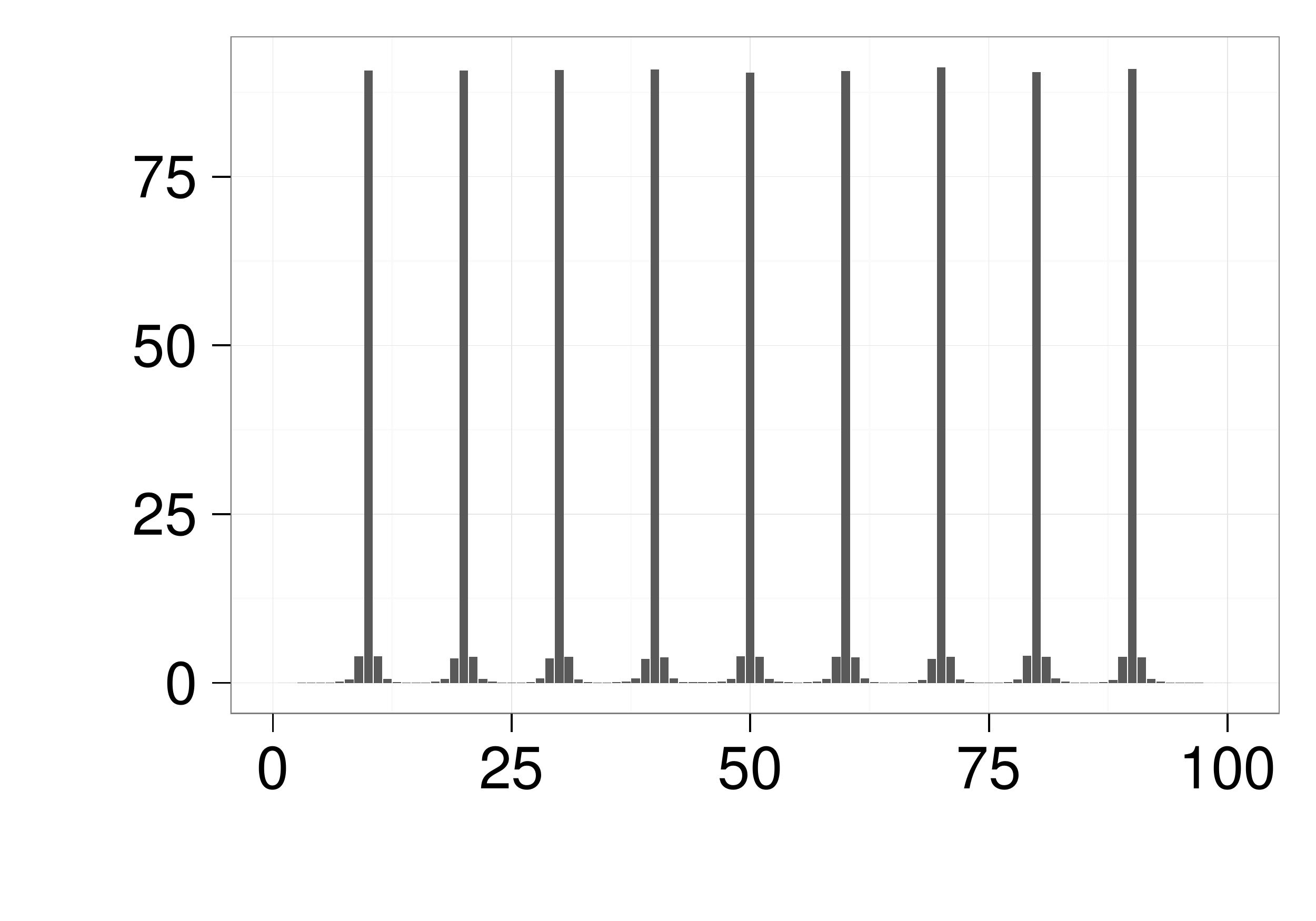}&
\hspace{-5mm}
\includegraphics*[height=4cm,width=4.5cm]{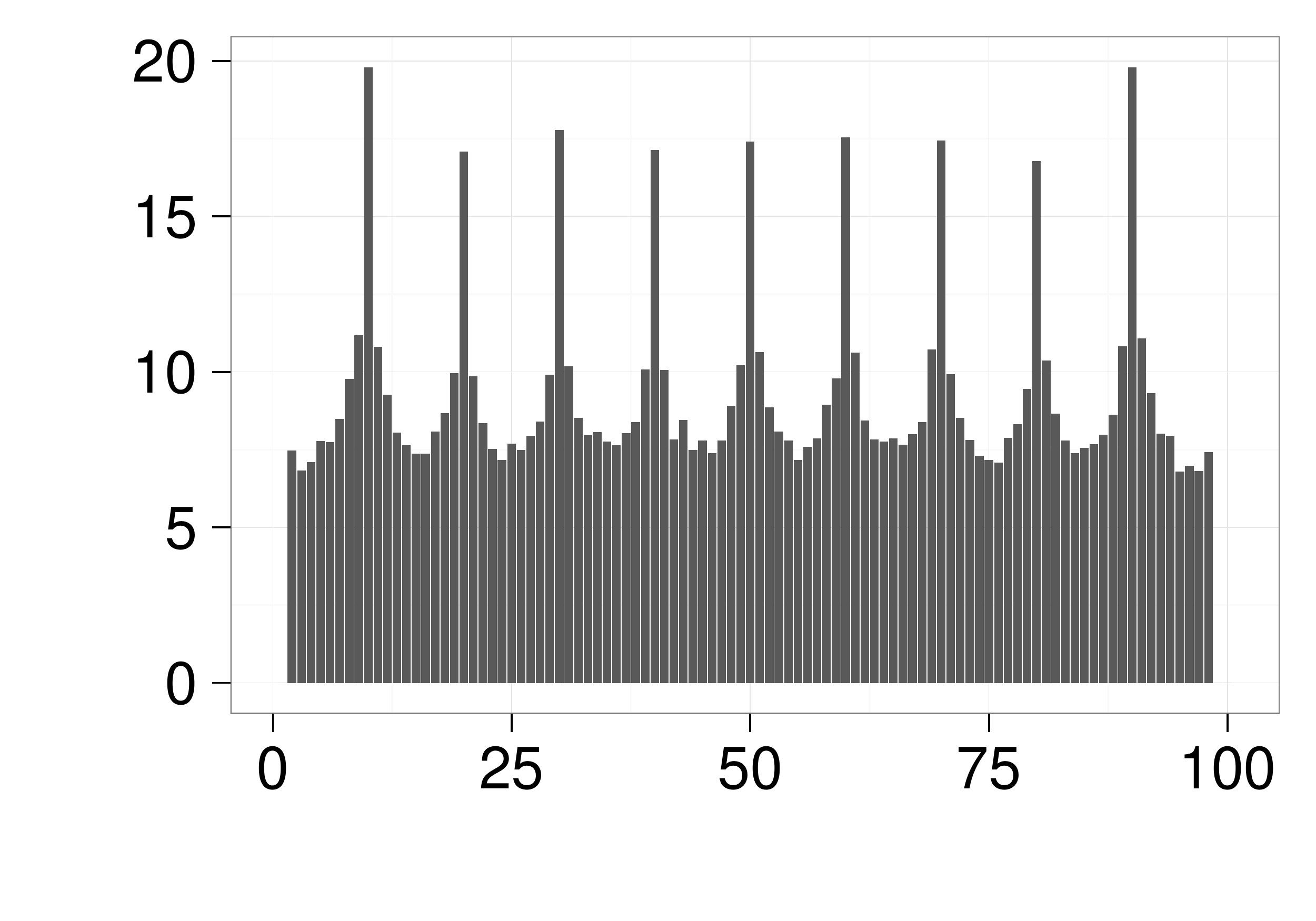}\\
\includegraphics*[height=4cm,width=4.5cm]{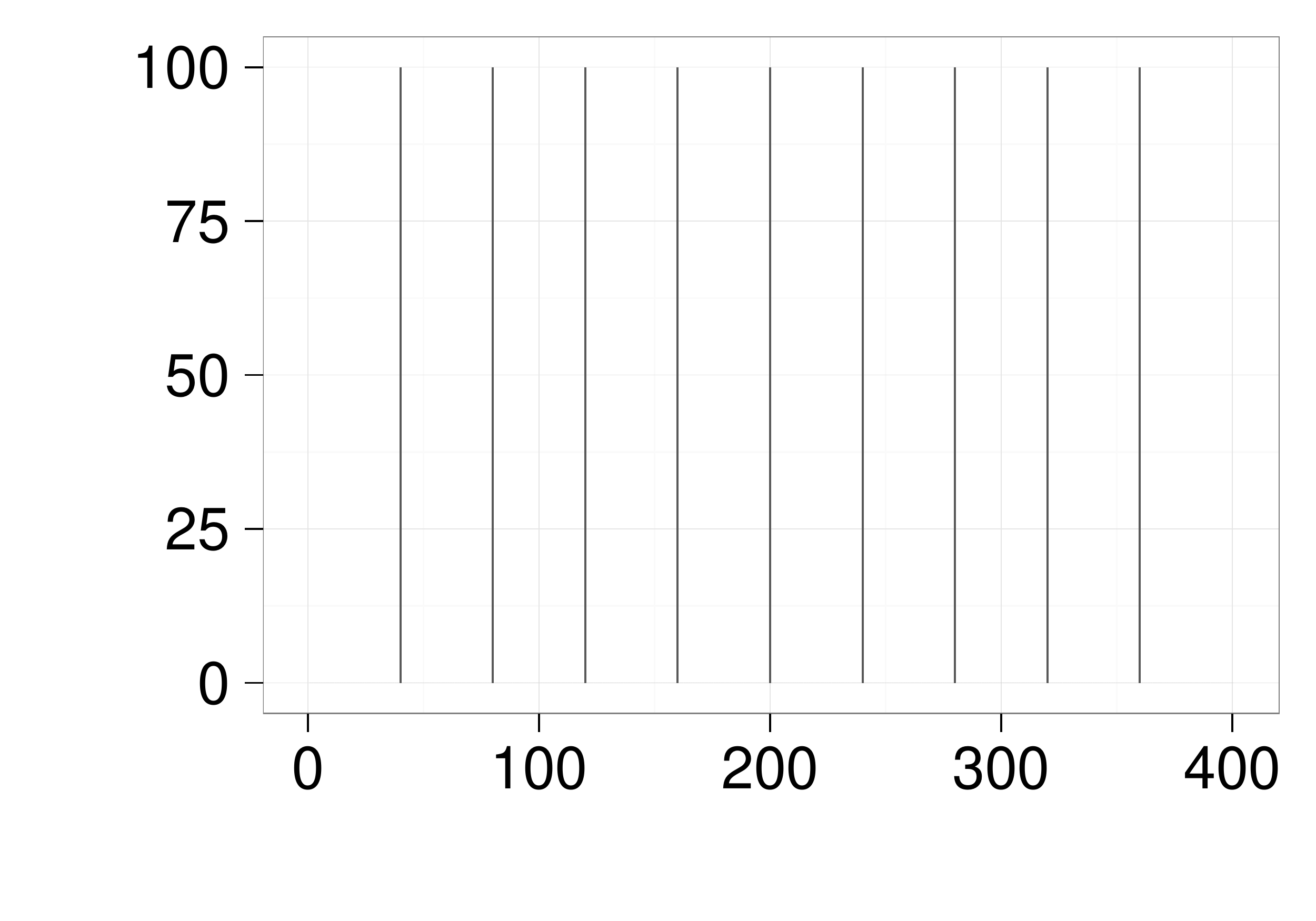}&
\hspace{-5mm}
\includegraphics*[height=4cm,width=4.5cm]{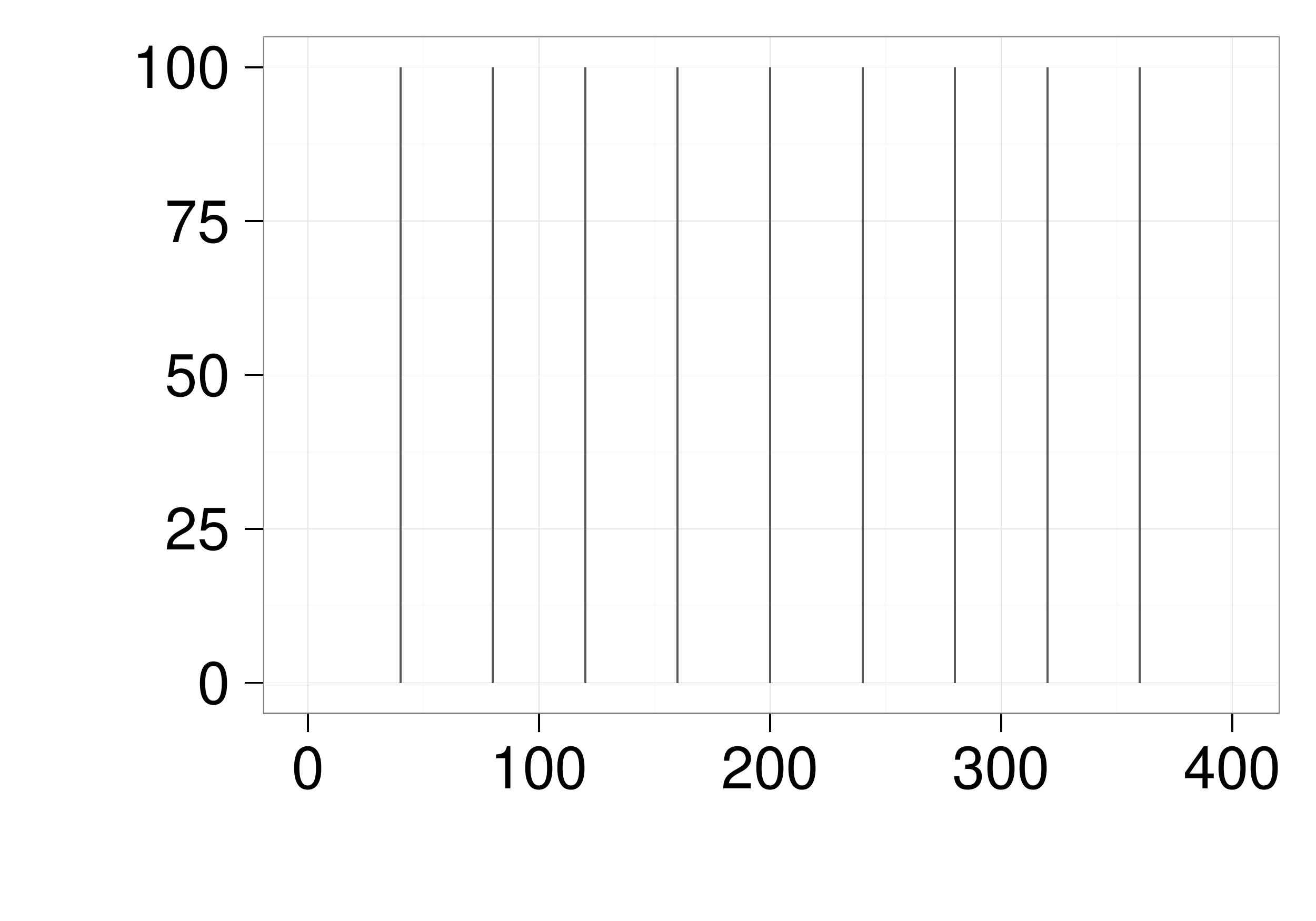}&
\hspace{-5mm}
\includegraphics*[height=4cm,width=4.5cm]{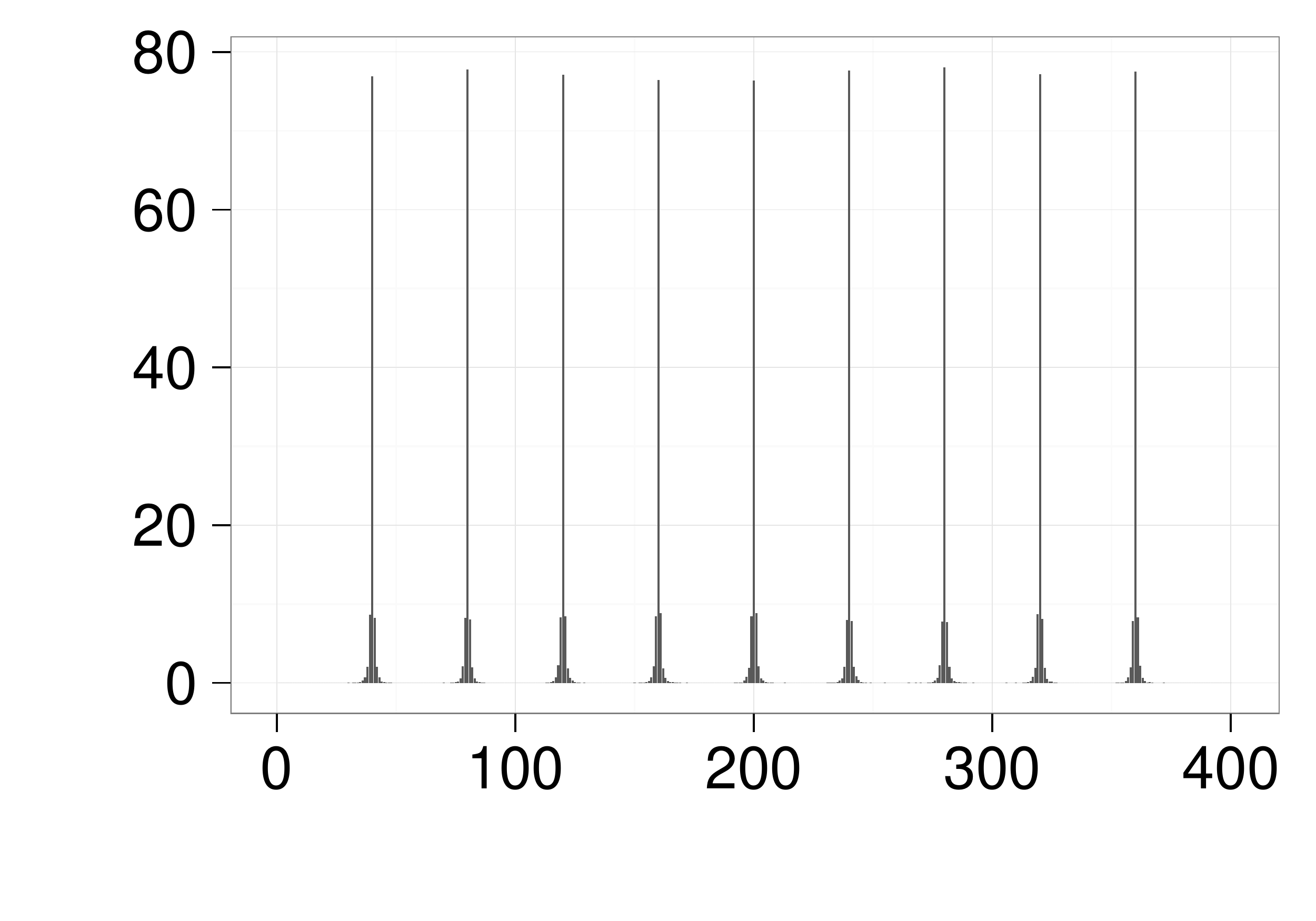}\\
\end{tabular}
\caption{Barplots associated with the multiple change-point estimation
  procedure for $n=100$ (top), $n=400$ (bottom),
  $\mathcal{L}_1=\mathcal{N}(1,\sigma^2)$ and
$\mathcal{L}_2=\mathcal{N}(0,\sigma^2)$ for different
  values of $\sigma$. The true positions of the change-points are
  located at the multiples of $n/10$.  \label{fig:barplots}}
\end{center}
\end{figure}

We also compared our multiple change-point estimation strategy
(\textsf{MuChPoint}) to the
one devised by \cite{matteson:2014} (\textsf{ecp}), which is, to the best of our
knowledge, the most recent approach proposed for solving this issue.
The results are gathered in Figures \ref{fig:comparaisons} and
\ref{fig:comparaisons_diag} which display the boxplots of the
distance $D$, defined in \eqref{eq:D}, between the change-points provided by these procedures in the \textsf{Block Diagonal} and
\textsf{Chessboard} configurations for the Gaussian, Exponential and
Cauchy distributions. These boxplots
are obtained from 100 replications of $n\times n$ symmetric matrices
where $n\in\{50,100,200,300,400\}$. 
More precisely, the distance $D$ is defined as follows
\begin{equation}\label{eq:D}
D(\widehat{\textbf{t}},\textbf{t}^\star)=\frac1n\sqrt{\sum_{k=1}^{K^\star}(\widehat{t}_k-t_k^\star)^2},
\end{equation}
where $\textbf{t}^\star=(t_1^\star,\dots,t_{K^\star}^\star)$ denotes the vector of
the true $K^\star$ change-point positions and
$\widehat{\textbf{t}}=(\widehat{t}_1,\dots,\widehat{t}_{K^\star})$
its estimation either obtained by \textsf{MuChPoint} or \textsf{ecp}.
Note that, it actually corresponds to the usual
$\ell_2$-norm of the vector
$\boldsymbol{\tau}^\star-\widehat{\boldsymbol{\tau}}$ where
$\boldsymbol{\tau}^\star=(\tau_1^\star,\dots,\tau_{K^\star}^\star)$,
$\widehat{\boldsymbol{\tau}}=(\widehat{\tau}_1,\dots,\widehat{\tau}_{K^\star})$
with $t_k^\star=\lfloor n\tau_k^\star\rfloor$ and
$\widehat{t}_k=\lfloor n\widehat{\tau}_k\rfloor$.
In order to benchmark these methodologies, we provide to both of them
the true value $K^\star$ of the number of change-points, which is here
equal to 10.

\begin{figure}[!h]
\begin{center}
\begin{tabular}{ccc}
$\sigma=1$&\hspace{-48mm}$a=1$&\hspace{-50mm}$\lambda=1/2$\\
\includegraphics*[height=4cm,width=4.5cm]{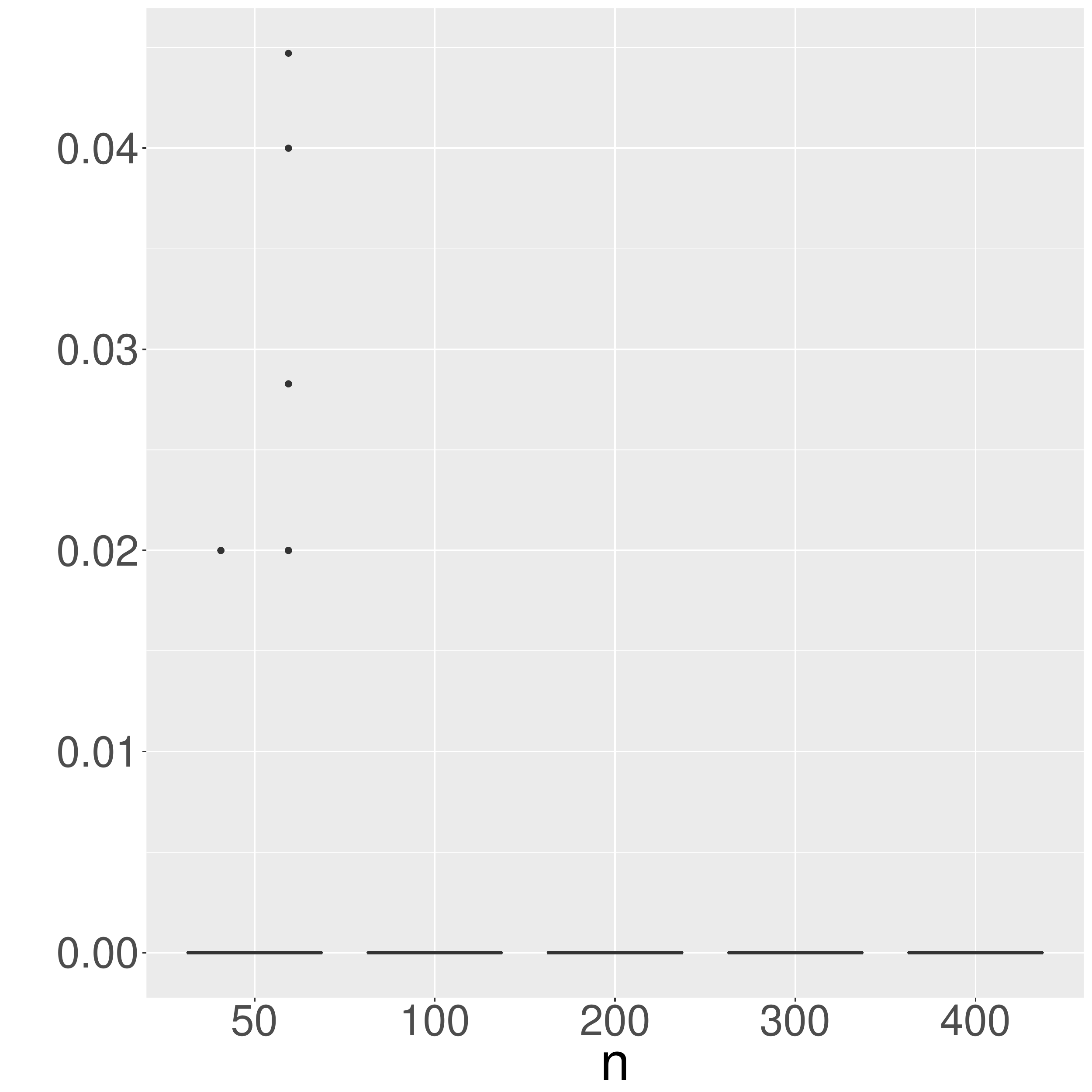}
&\hspace{-48mm}\includegraphics*[height=4cm,width=4.5cm]{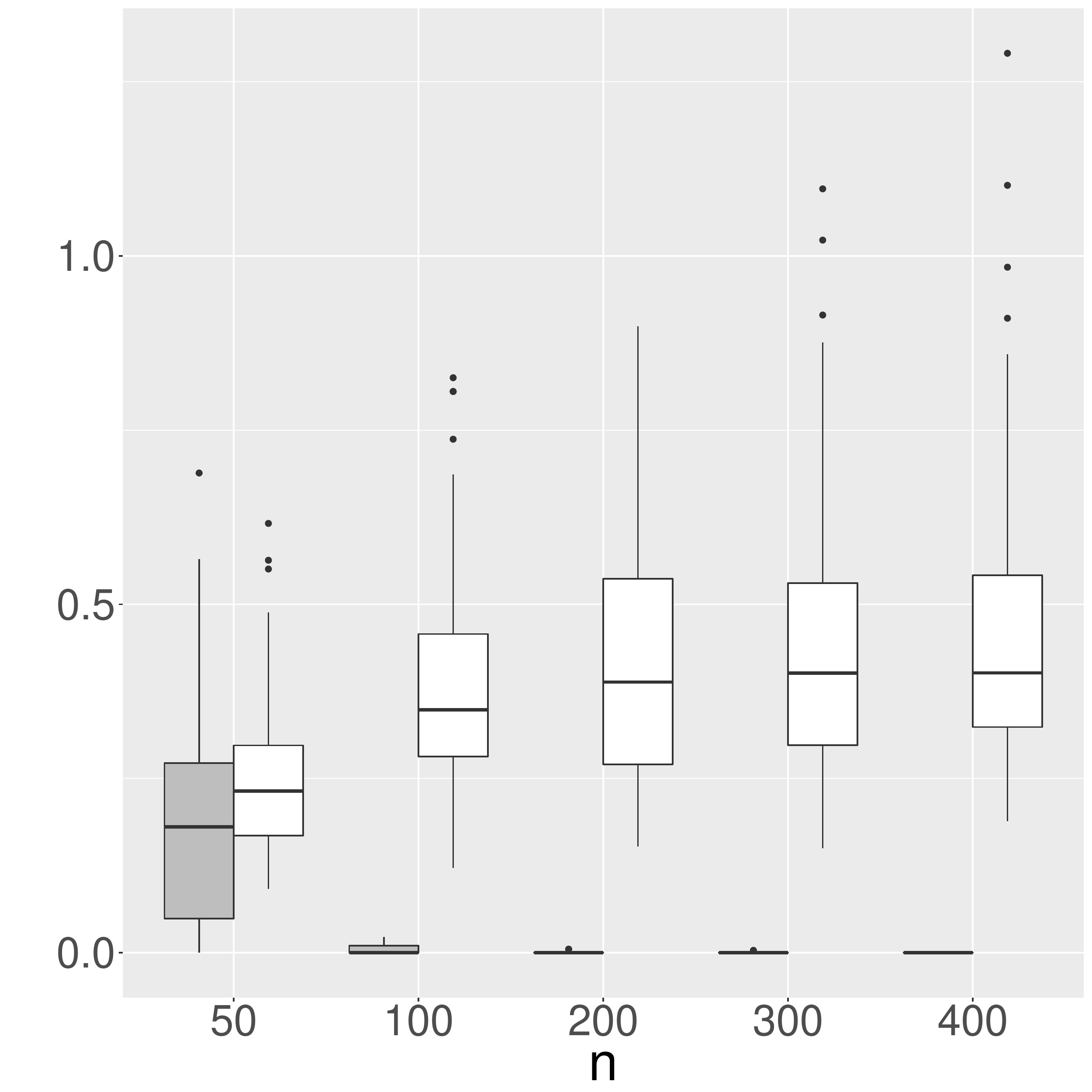}
&\hspace{-50mm}\includegraphics*[height=4cm,width=4.5cm]{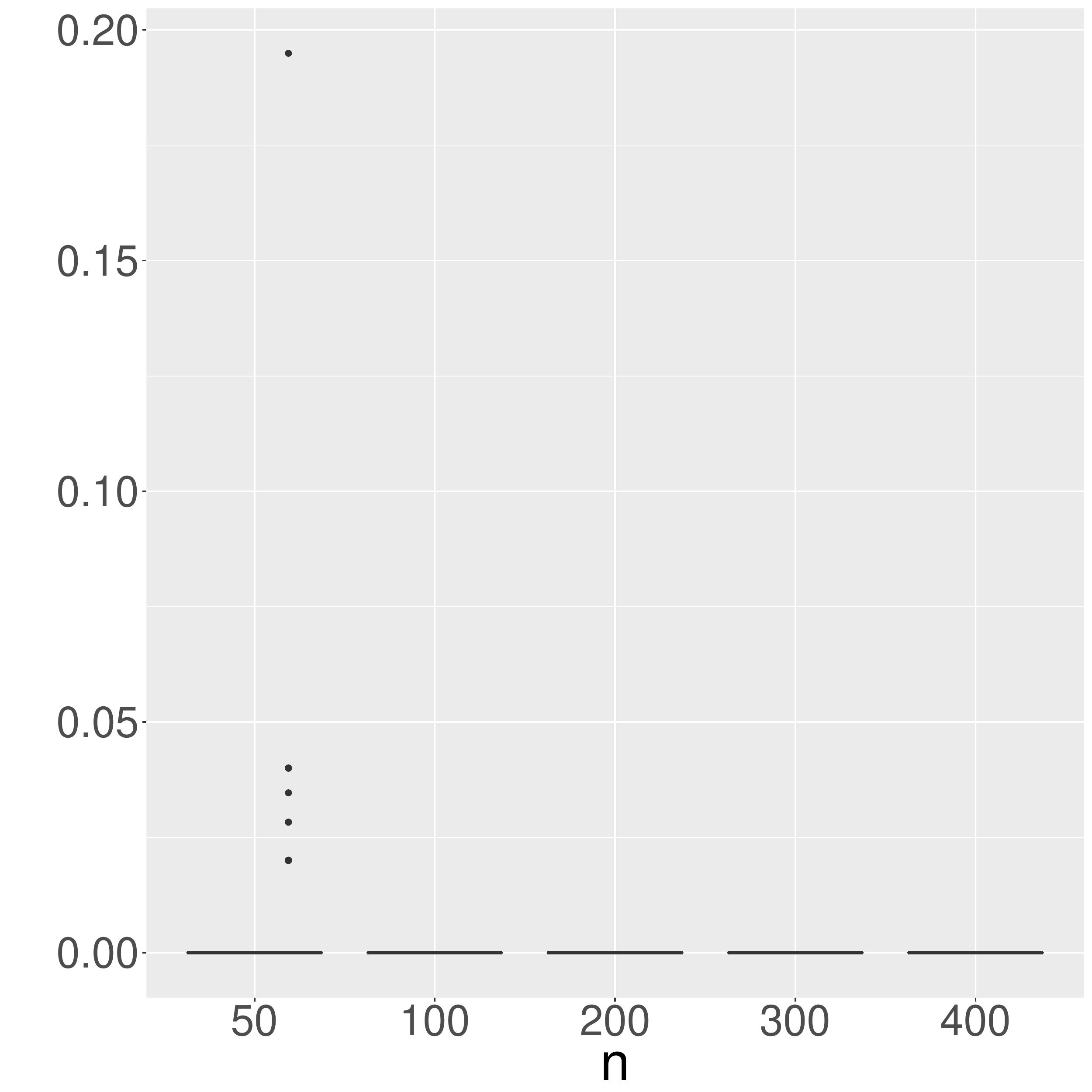}\\
$\sigma=2$&\hspace{-48mm}$a=2$&\hspace{-50mm}$\lambda=1$\\
\includegraphics*[height=4cm,width=4.5cm]{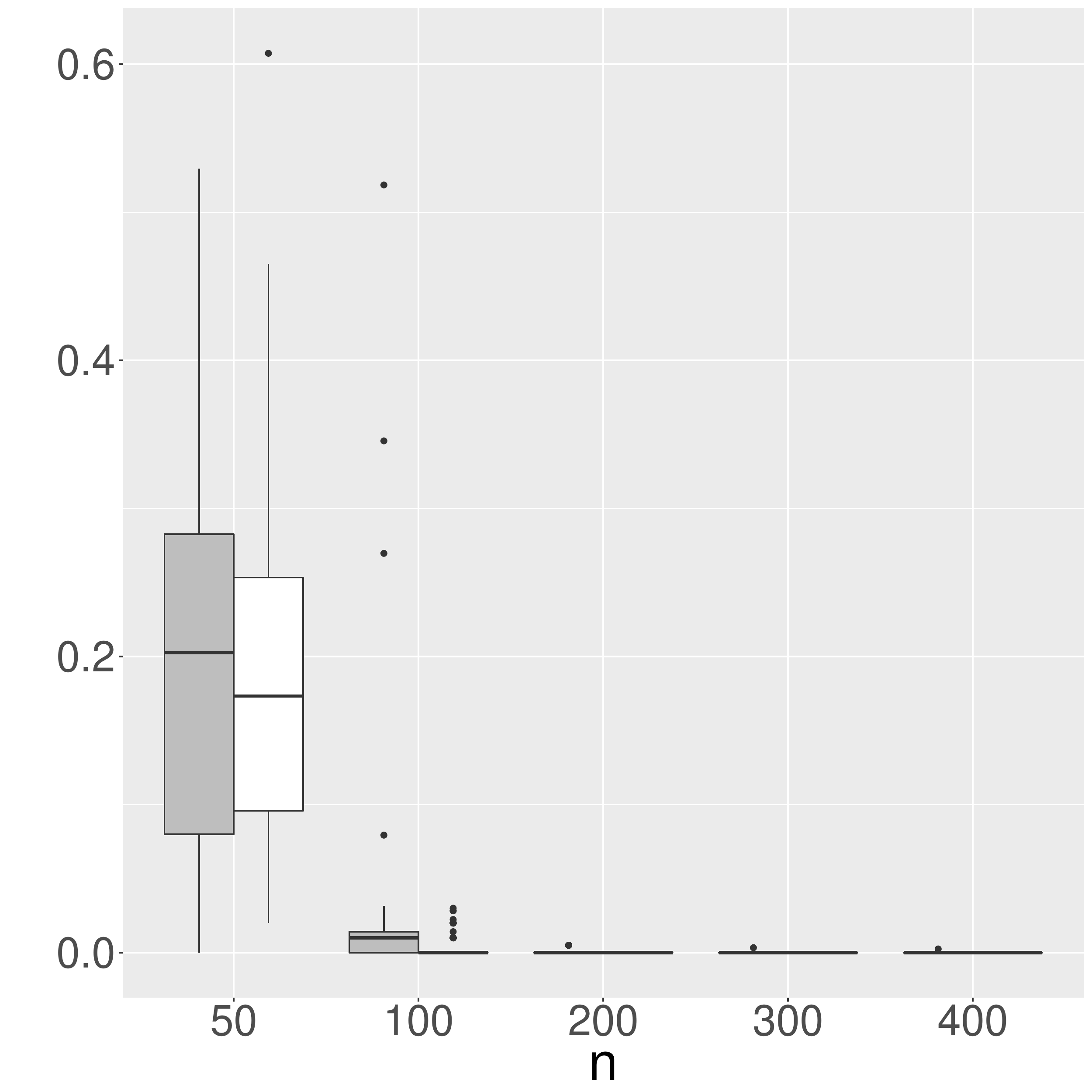}
&\includegraphics*[height=4cm,width=4.5cm]{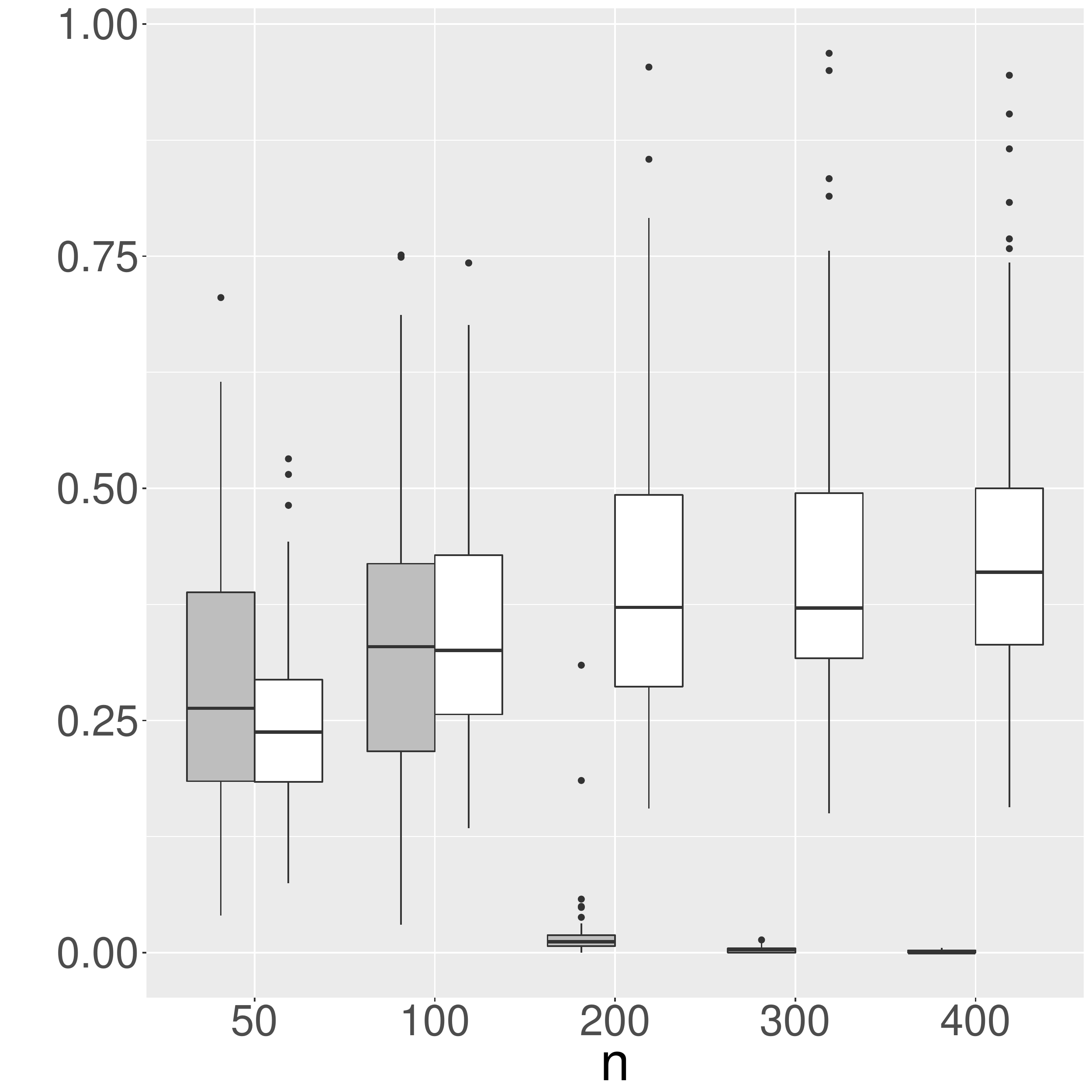}
\includegraphics*[height=4cm,width=4.5cm]{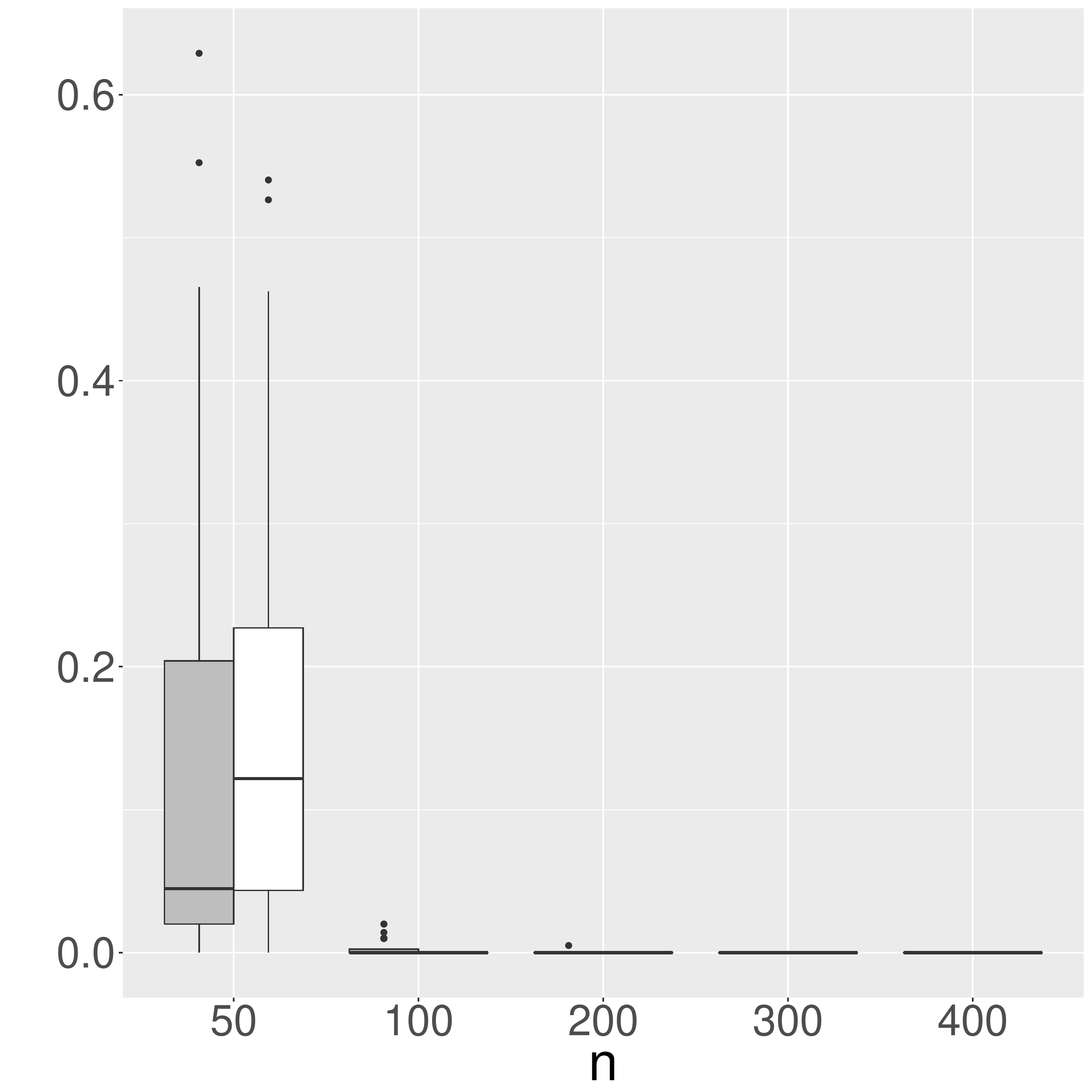}\\
$\sigma=5$&\hspace{-48mm}$a=5$&\hspace{-50mm}$\lambda=4$\\
\includegraphics*[height=4cm,width=4.5cm]{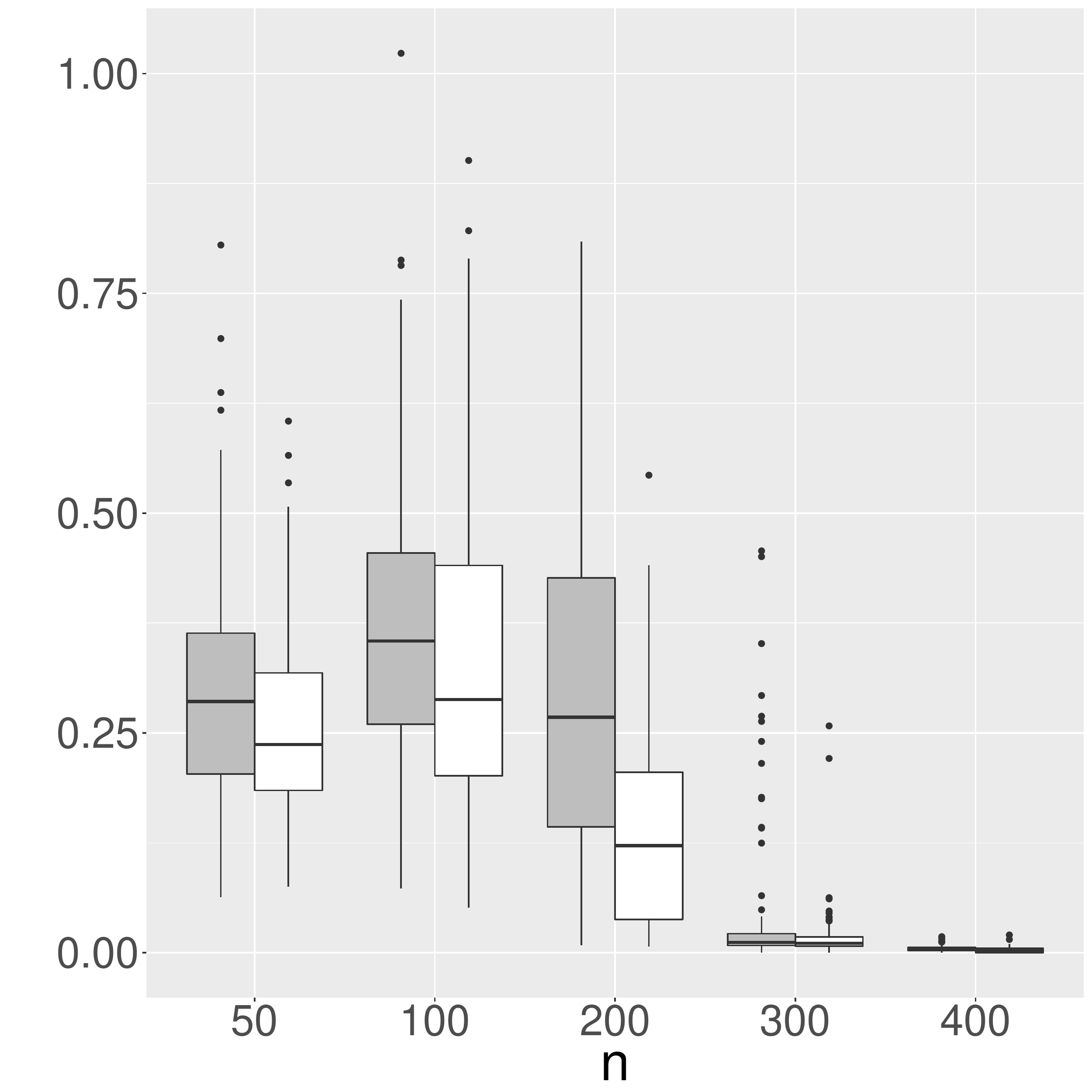}
&\includegraphics*[height=4cm,width=4.5cm]{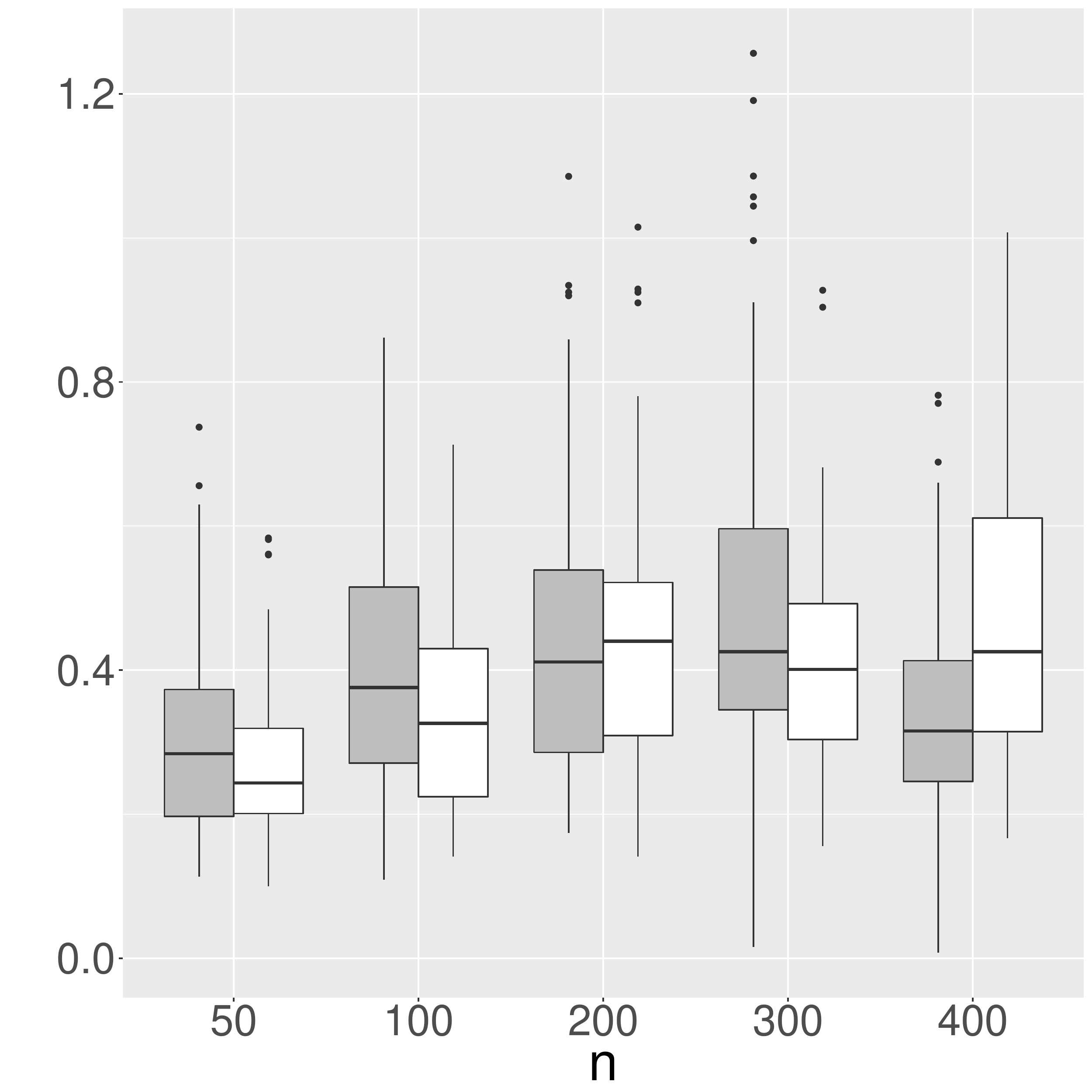}
\includegraphics*[height=4cm,width=4.5cm]{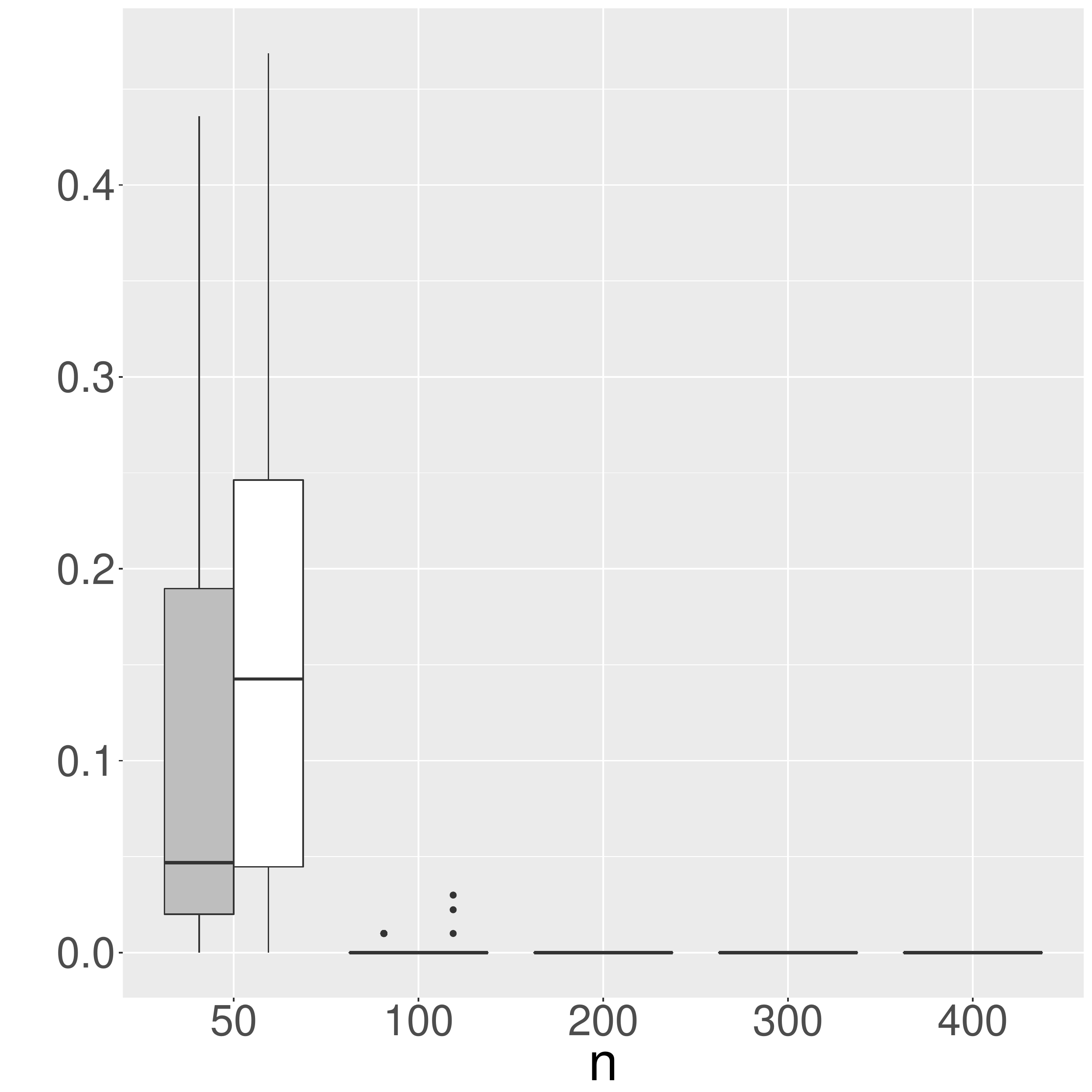}\\
\end{tabular}
\caption{Boxplots of the distances $D$ for \textsf{MuChPoint} and
  \textsf{ecp} in the \textsf{Chessboard} configuration.  Left: $\mathcal{L}_1=\mathcal{N}(1,\sigma^2)$,
  $\mathcal{L}_2=\mathcal{N}(0,\sigma^2)$, middle: $\mathcal{L}_1=\mathcal{C}au(1,a)$,
$\mathcal{L}_2=\mathcal{C}au(0,a)$ and right:
  $\mathcal{L}_1=\mathcal{E}xp(2)$,
  $\mathcal{L}_2=\mathcal{E}xp(\lambda)$ for different values of
  $\sigma$, $\lambda$ and $a$. The boxplots associated with
  \textsf{MuChPoint} are displayed in gray and the ones of
  \textsf{ecp} in white.\label{fig:comparaisons}}
\end{center}
\end{figure}

\begin{figure}[!h]
\begin{center}
\begin{tabular}{ccc}
$\sigma=1$&\hspace{-48mm}$a=1$&\hspace{-50mm}$\lambda=1/2$\\
\includegraphics*[height=4cm,width=4.5cm]{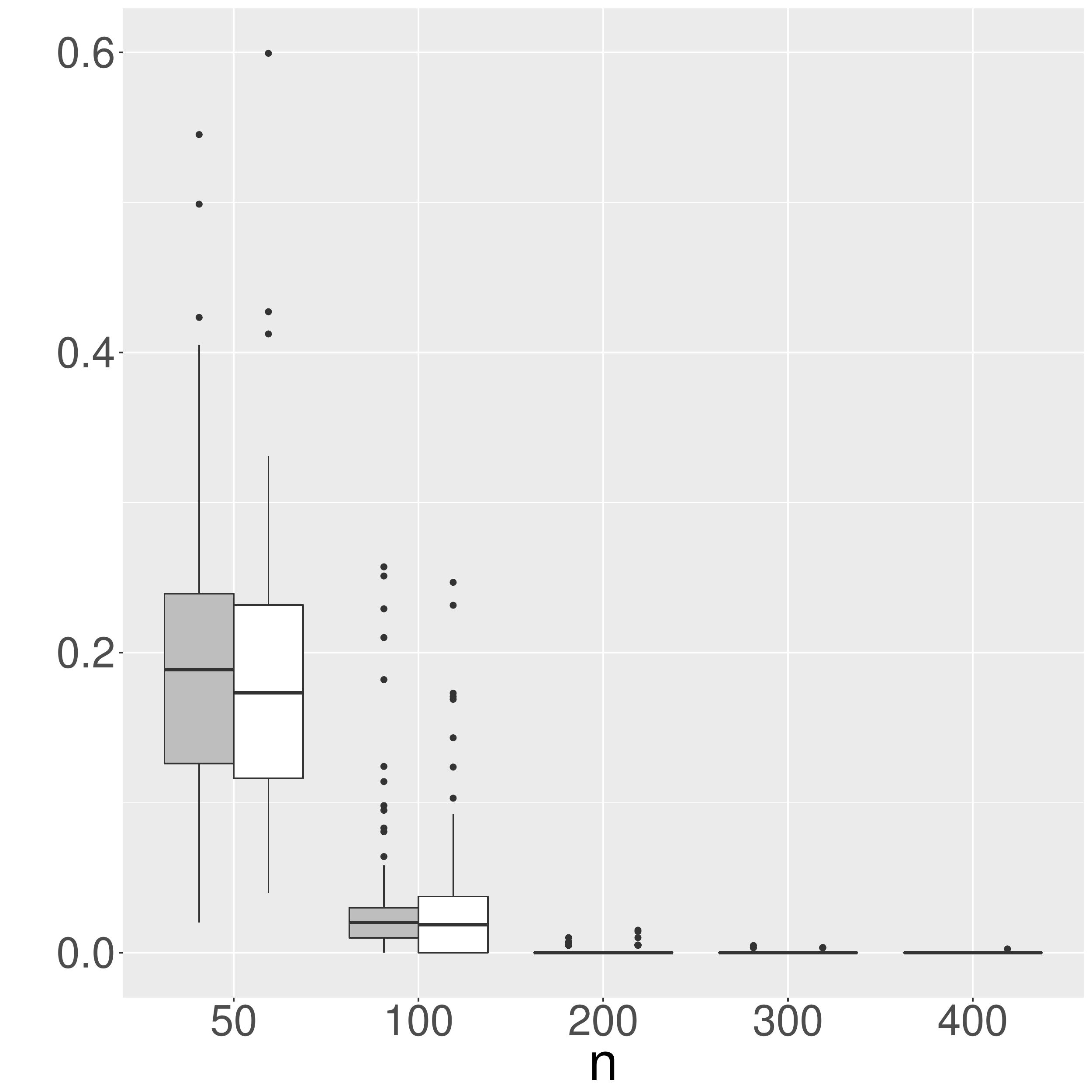}
&\hspace{-48mm}\includegraphics*[height=4cm,width=4.5cm]{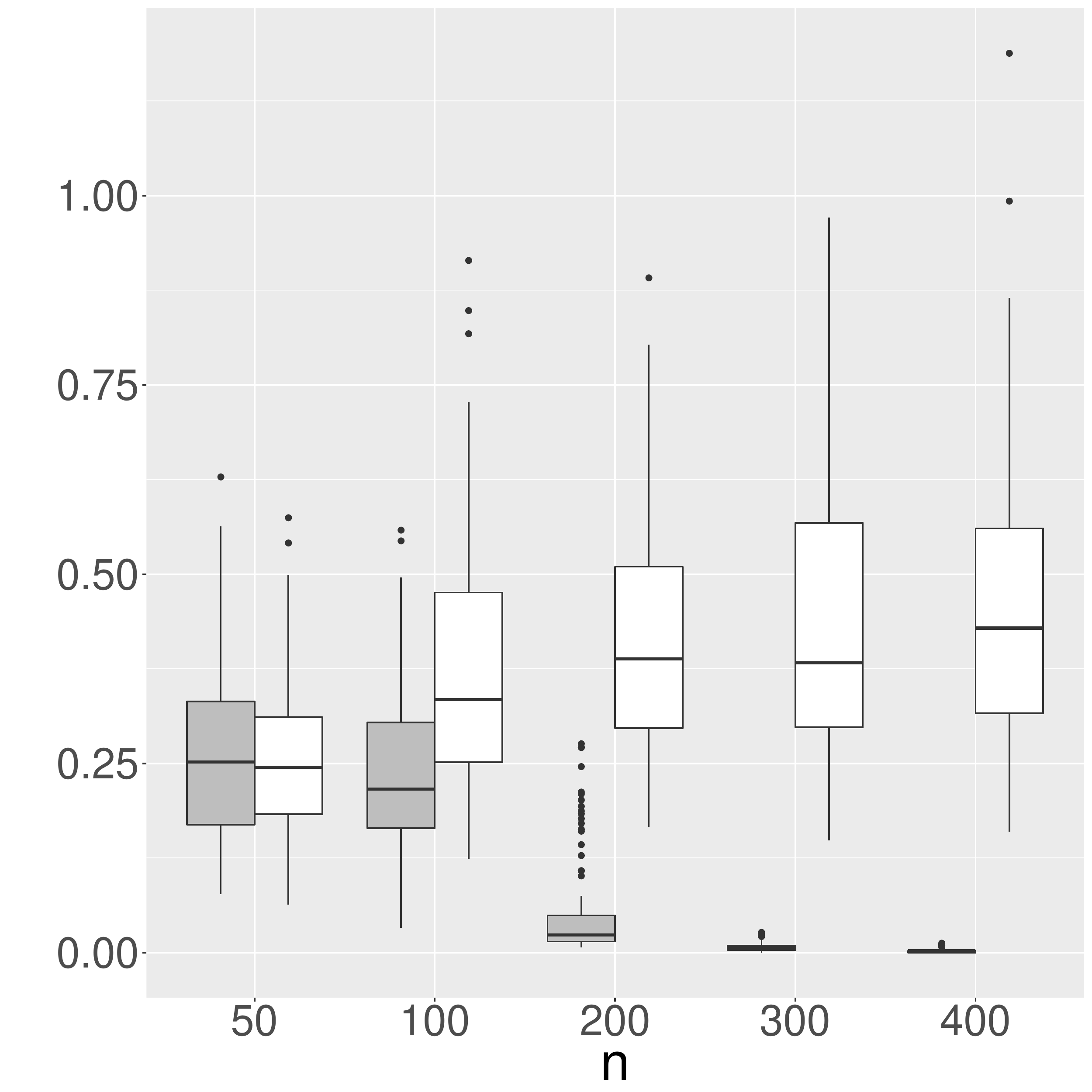}
&\hspace{-50mm}\includegraphics*[height=4cm,width=4.5cm]{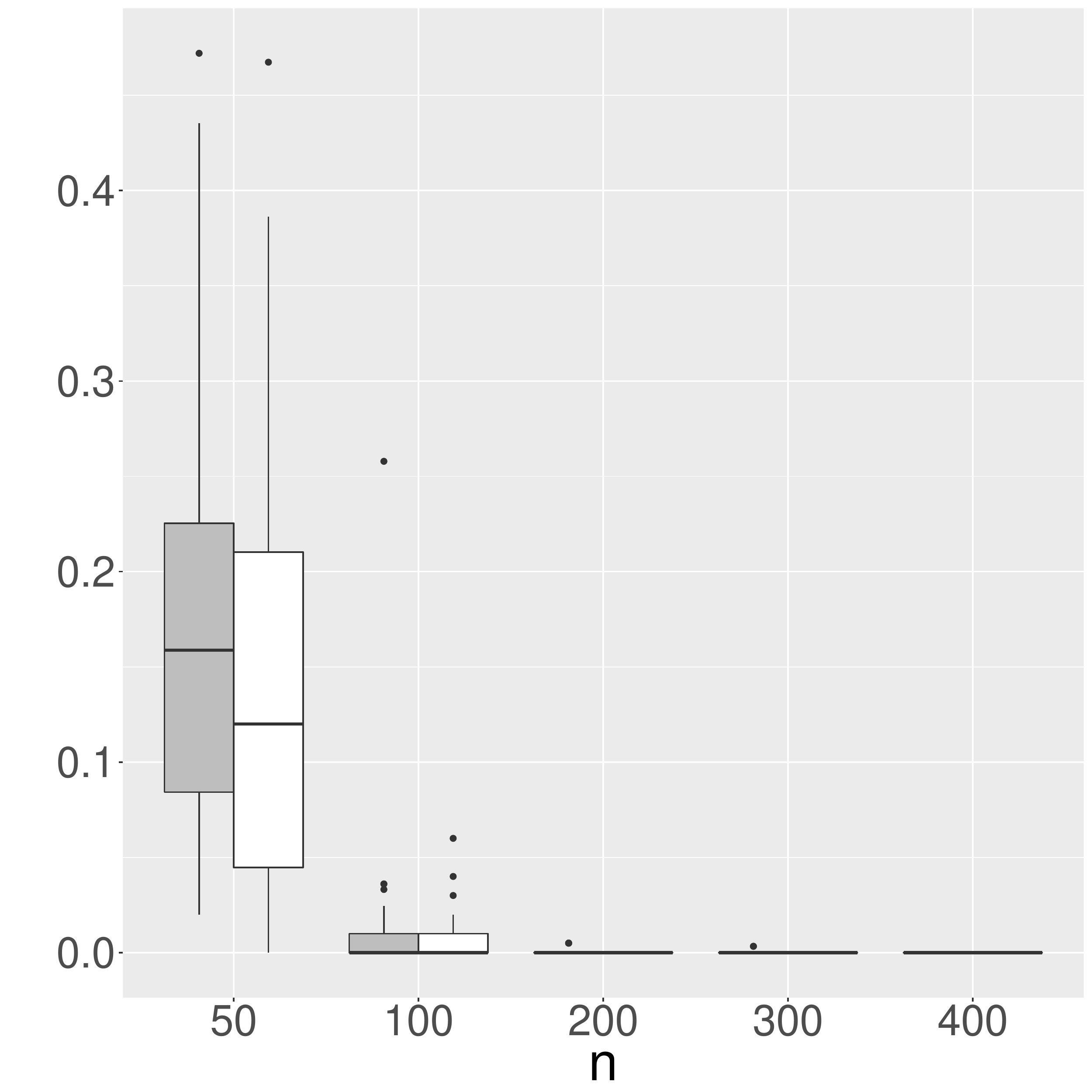}\\
$\sigma=2$&\hspace{-48mm}$a=2$&\hspace{-50mm}$\lambda=1$\\
\includegraphics*[height=4cm,width=4.5cm]{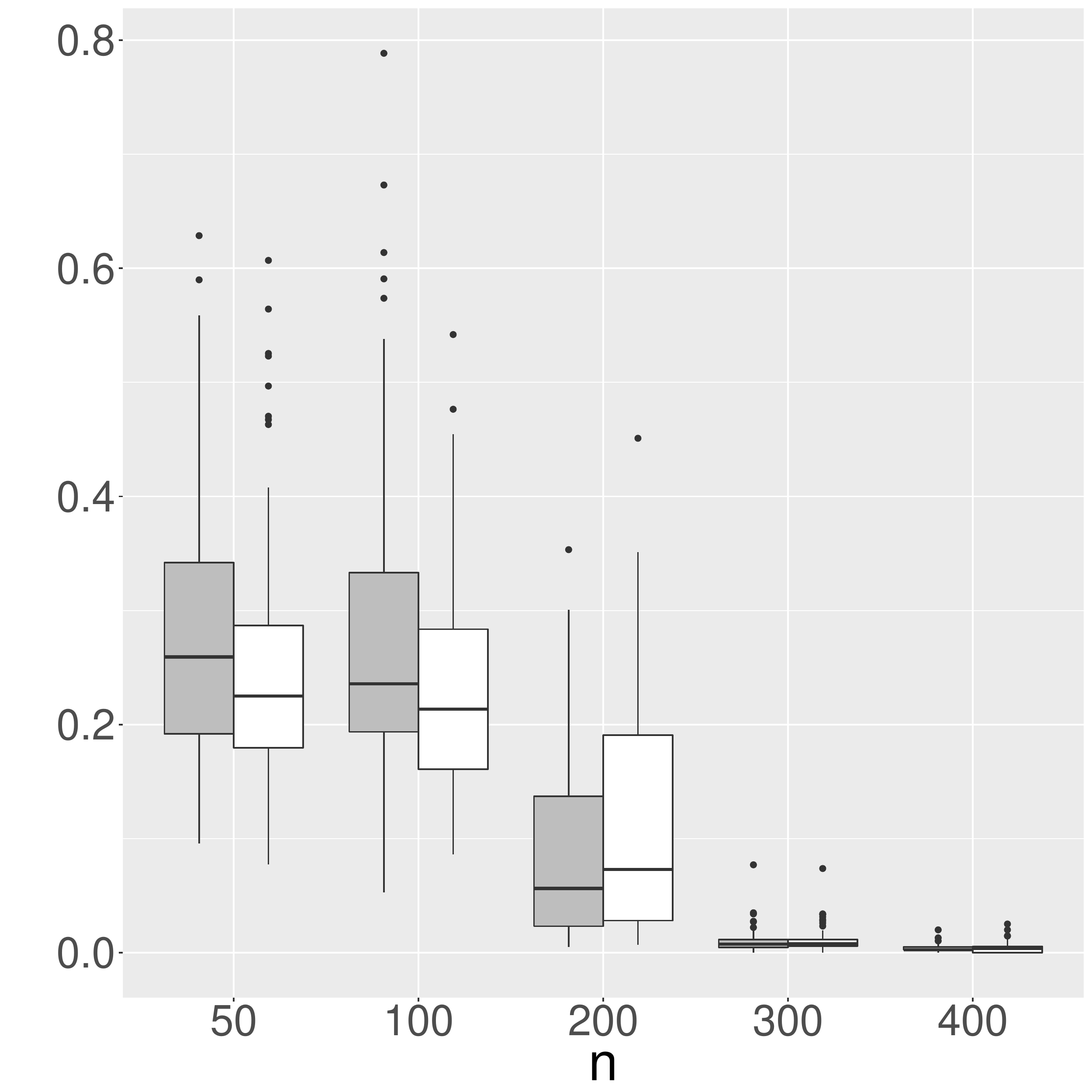}
&\includegraphics*[height=4cm,width=4.5cm]{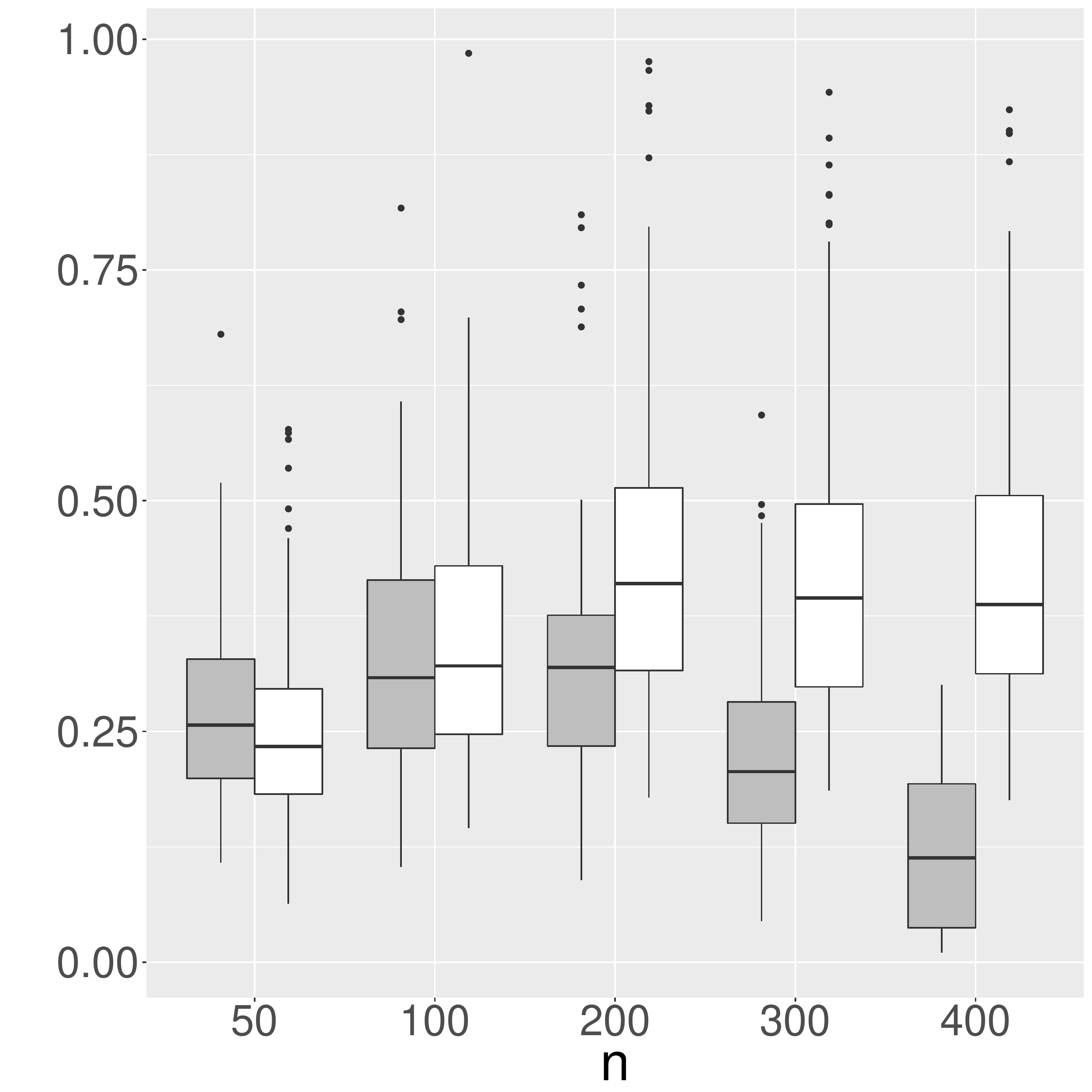}
\includegraphics*[height=4cm,width=4.5cm]{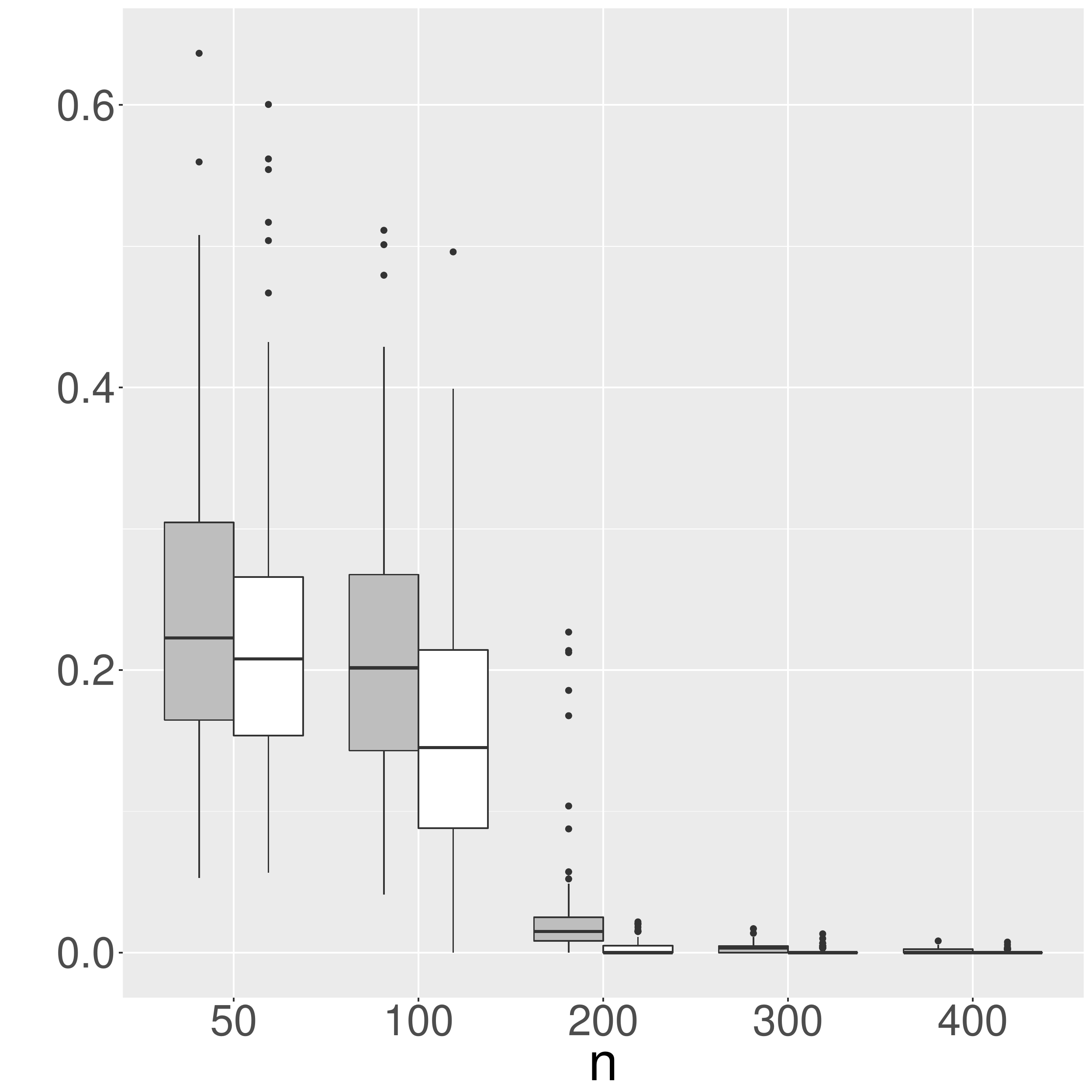}\\
$\sigma=5$&\hspace{-48mm}$a=5$&\hspace{-50mm}$\lambda=4$\\
\includegraphics*[height=4cm,width=4.5cm]{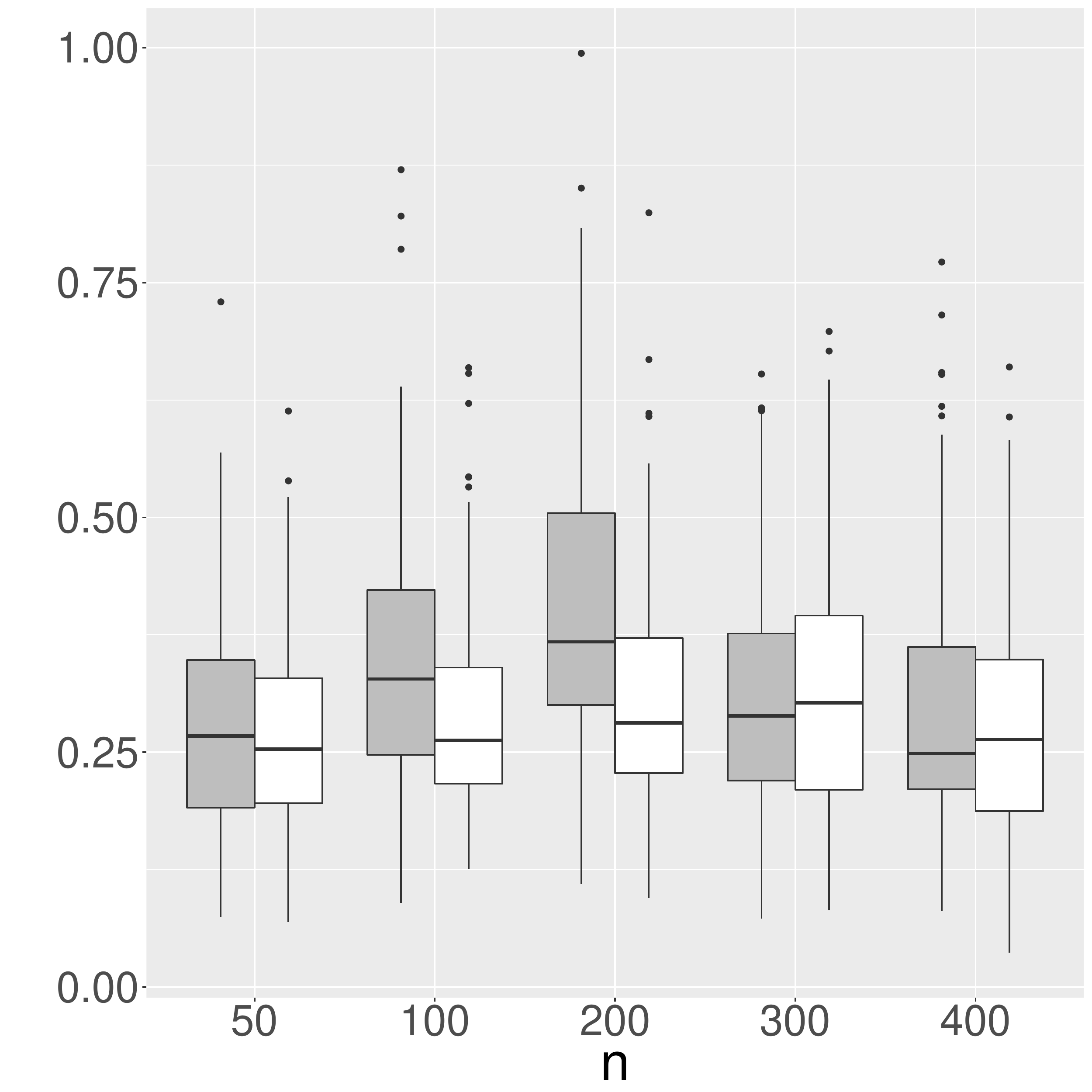}
&\includegraphics*[height=4cm,width=4.5cm]{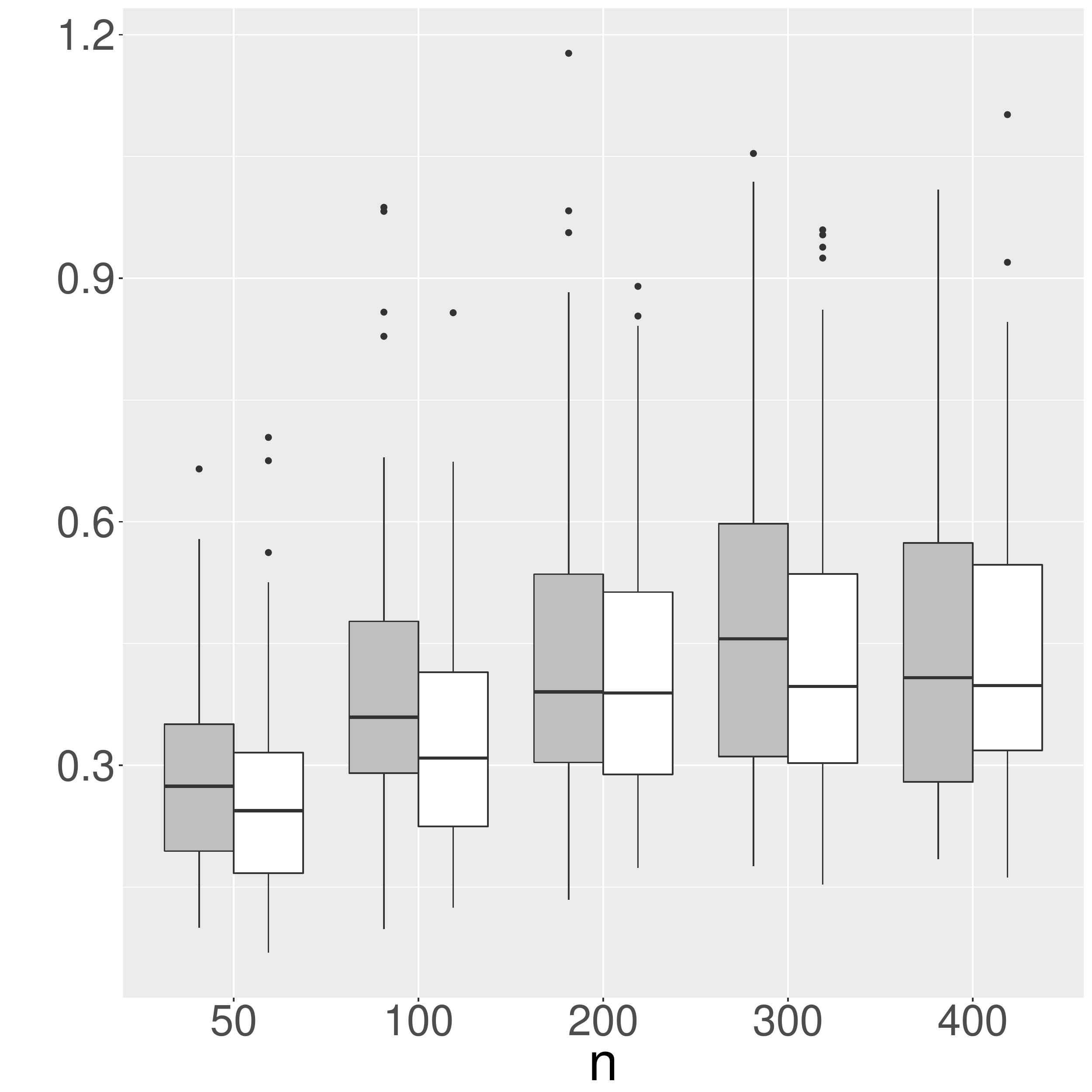}
\includegraphics*[height=4cm,width=4.5cm]{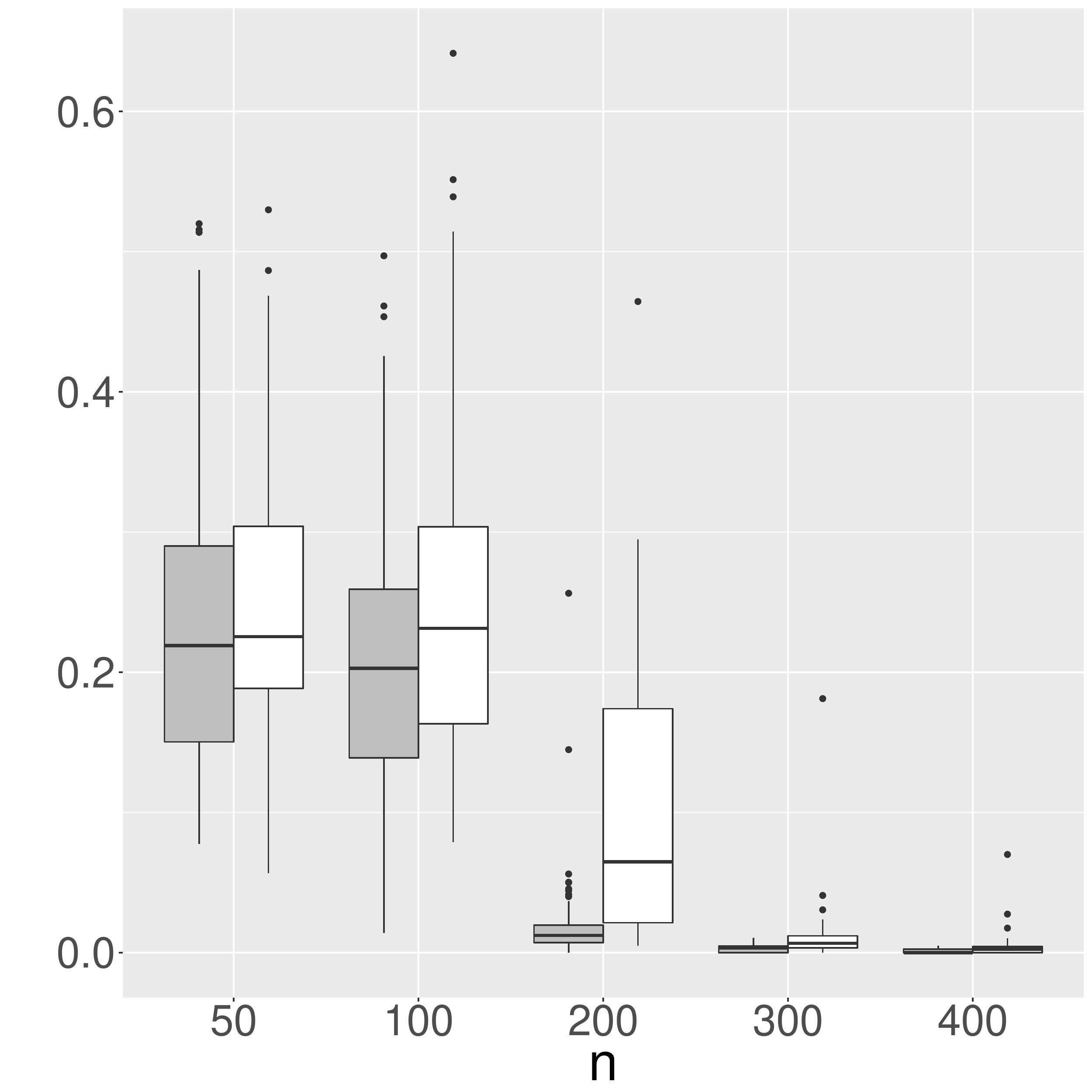}\\
\end{tabular}
\caption{Boxplots of the distances $D$ for \textsf{MuChPoint} and
  \textsf{ecp} in the \textsf{Block Diagonal} configuration.  Left: $\mathcal{L}_1=\mathcal{N}(1,\sigma^2)$,
  $\mathcal{L}_2=\mathcal{N}(0,\sigma^2)$, middle: $\mathcal{L}_1=\mathcal{C}au(1,a)$,
$\mathcal{L}_2=\mathcal{C}au(0,a)$ and right:
  $\mathcal{L}_1=\mathcal{E}xp(2)$,
  $\mathcal{L}_2=\mathcal{E}xp(\lambda)$ for different values of
  $\sigma$, $\lambda$ and $a$. The boxplots associated with
  \textsf{MuChPoint} are displayed in gray and the ones of
  \textsf{ecp} in white.\label{fig:comparaisons_diag}}
\end{center}
\end{figure}

We observe from Figures \ref{fig:comparaisons} and
\ref{fig:comparaisons_diag} that both approaches have similar
statistical performance. However, \textsf{MuchPoint}  performs
better than \textsf{ecp} in the Cauchy case. In the Gaussian
framework, the performance of \textsf{ecp} are a little bit better
for small $n$ and large $\sigma$.

%%% Local Variables:
%%% mode: latex
%%% eval: (TeX-PDF-mode 1)
%%% TeX-master: "Test_HiC.tex"
%%% ispell-local-dictionary: "en_US"
%%% eval: (flyspell-mode 1)
%%% End:

%% file: Tableau.tex
\begin{table}[!h]
\begin{center}
\begin{tabular}{c|*{3}{c}|*{3}{c}}
&\multicolumn{3}{c|}{$\nun=\lfloor
  0.1\n\rfloor$}&\multicolumn{3}{c}{$\nun=\lfloor 0.5\n\rfloor$}\\
\cline{2-7}
&$\mathcal{N}(0,1)$&$\mathcal{C}au(0,1)$&$\mathcal{E}xp(2)$&$\mathcal{N}(0,1)$&$\mathcal{C}au(0,1)$&$\mathcal{E}xp(2)$\\
\hline
$n=50$&0.83&0.83&0.82&0.78&0.79&0.76\\\hline
$n=100$&0.81&0.8&0.82&0.78&0.8&0.78\\\hline
$n=500$&0.78&0.8&0.81&0.8&0.78&0.77\\\hline
$n=1000$&0.79&0.78&0.79&0.78&0.77&0.79\\
\end{tabular}
\caption{\label{Table:Quantiles} Estimation of the empirical $0.95$ quantiles of
$T_n(n_1)$.}
\end{center}
\end{table}

%% file: real.tex
In this section, we apply our methodology to publicly available Hi-C data
(\url{http://chromosome.sdsc.edu/mouse/hi-c/download.html})
already studied by \cite{dixon2012topological}.
This technology provides read pairs corresponding to pairs of genomic loci that
physically interacts in the nucleus, see \cite{lieberman2009comprehensive} for further details.
The raw measurements provided by Hi-C data is therefore a list of pairs of
locations along the chromosome, at the nucleotide resolution. These
measurements are often summarized by a symmetric matrix $\boldsymbol{X}$ where each entry $\Xij$
corresponds the total number of read pairs matching in position
$i$ and position $j$, respectively. Positions refer here to a sequence
of non-overlapping windows of equal sizes covering the genome.
The number of windows may vary from one study to another:
\cite{lieberman2009comprehensive} considered a Mb resolution, whereas
\cite{dixon2012topological} went deeper and used windows of 40kb
(called hereafter the resolution).

In the sequel, we analyze the interaction matrices
of Chromosome 19 of the mouse cortex at a resolution 40 kb and
we compare the location of the estimated change-points found by our approach with those obtained
by \cite{dixon2012topological} on the same data since no ground truth is available.
In this case, the matrix that has to be processed is a $n\times n$
symmetric matrix where $n=1534$.

We display in Figure \ref{fig:resume_hic} the estimated matrix $\widehat{\textbf{X}}$ obtained by
using our strategy for various numbers of estimated
change-points. This estimated matrix
is a block-wise constant matrix for which the block boundaries are
estimated by using \textsf{MuChPoint} and the values within each block
correspond to the empirical mean of the observations lying in it.
We can see from this figure that both the diagonal and the extra
diagonal blocks are properly retrieved even when the number of
estimated change-points is not that large.

\begin{figure}[!h]
\begin{center}
\begin{tabular}{ccc}
\hspace{-10mm}\includegraphics*[height=6cm,width=6cm]{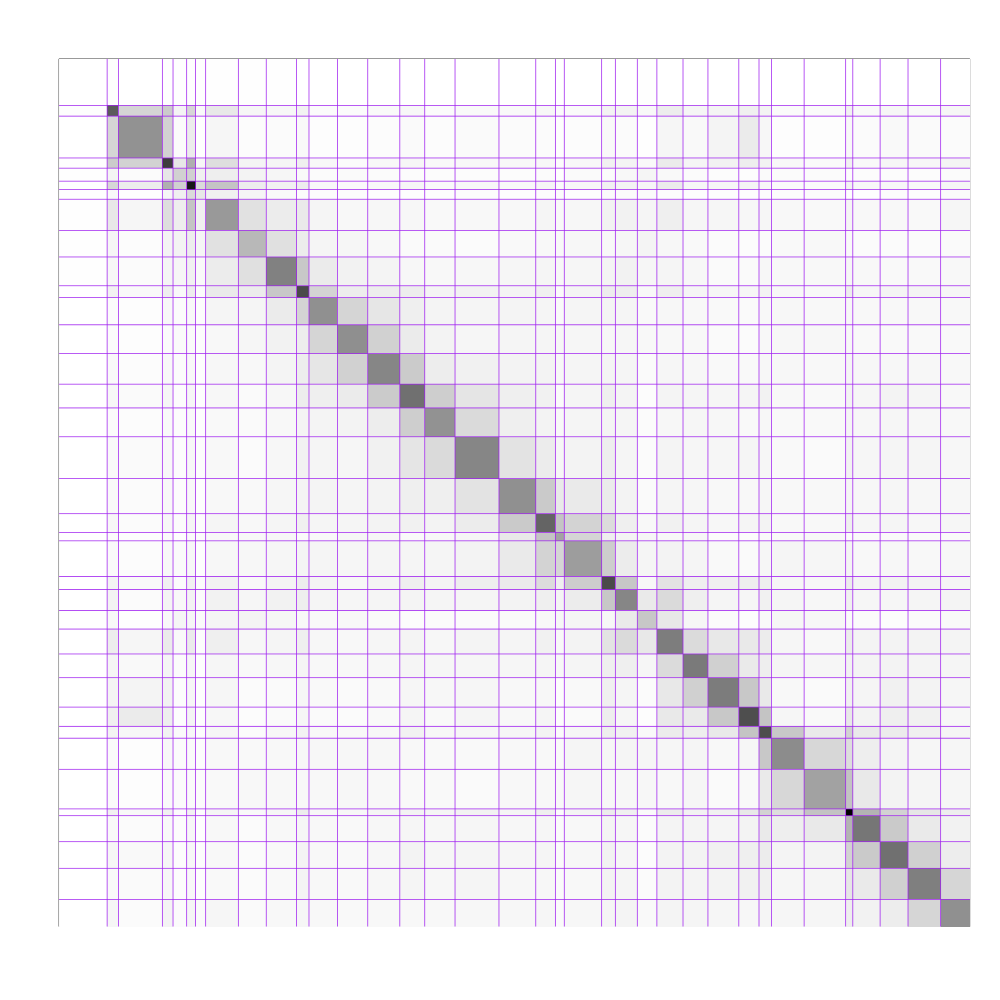}
&\hspace{-10mm} \includegraphics*[height=6cm,width=6cm]{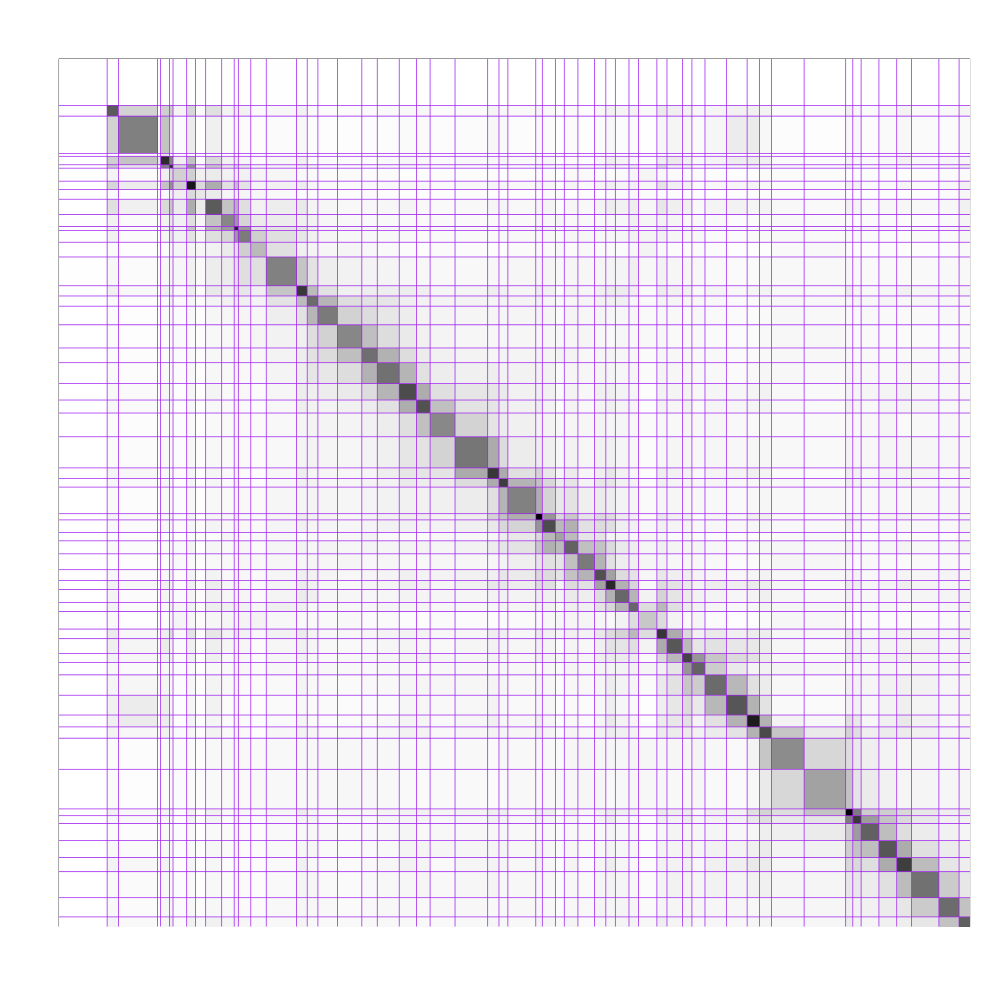}
&\hspace{-10mm}\includegraphics*[height=6cm,width=6cm]{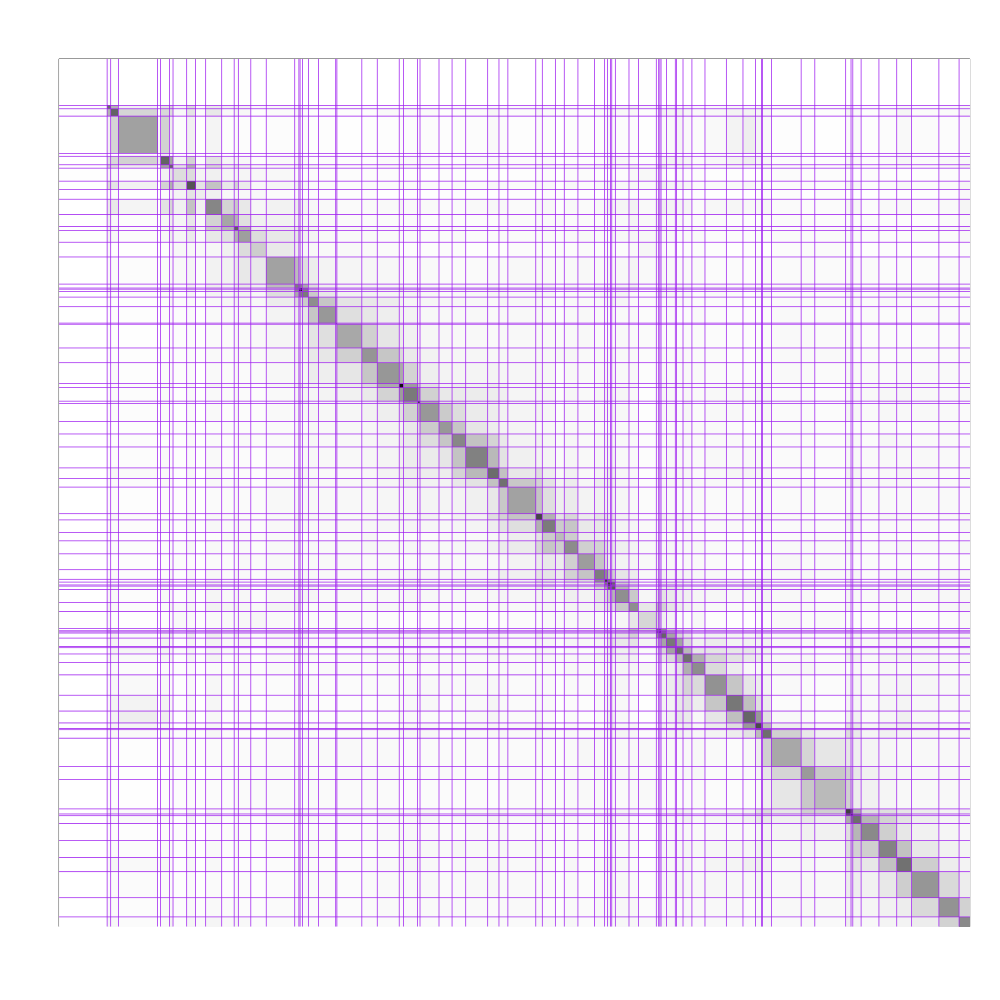}
\end{tabular}
\caption{Estimated matrices $\widehat{\textbf{X}}$ for different
  number of estimated change-points: 35 (left), 55 (middle) and 75 (right).\label{fig:resume_hic}}
\end{center}
\end{figure}

In order to further compare our approach with the one proposed by
\cite{dixon2012topological}, we computed the two parts of the Hausdorff distance
which is defined by
\begin{equation}\label{eq:hausdorff}
d\left(\widehat{\boldsymbol{t}}_B , \widehat{\boldsymbol{t}}\right)
= \max\left(
d_1 \left(\widehat{\boldsymbol{t}}_B, \widehat{\boldsymbol{t}}\right),
d_2 \left(\widehat{\boldsymbol{t}}_B, \widehat{\boldsymbol{t}}\right)
\right)\;,
\end{equation}
where $\widehat{\boldsymbol{t}}$ and $\widehat{\boldsymbol{t}}_B$ are
the change-points found
by our approach and \cite{dixon2012topological}, respectively. In (\ref{eq:hausdorff}),
\begin{eqnarray*}
d_1 \left(\mathbf{a},\mathbf{b}\right) & = & \sup_{b\in\mathbf{b}} \inf_{a\in\mathbf{a}} \left\vert a - b \right\vert \label{eq:hausd_1}, \\
d_2 \left(\mathbf{a},\mathbf{b}\right) & = & d_1 \left(\mathbf{b},\mathbf{a}\right)\label{eq:hausd_2}.
\end{eqnarray*}
More precisely, Figure \ref{fig:hausdorff} displays the boxplots of
the $d_1$ and $d_2$ parts of the Hausdorff distance without taking the
supremum in white and gray for different values of the estimated number of change-points, respectively.

 \begin{figure}[!h]
\begin{center}
   \includegraphics[height=6cm,width=9cm]{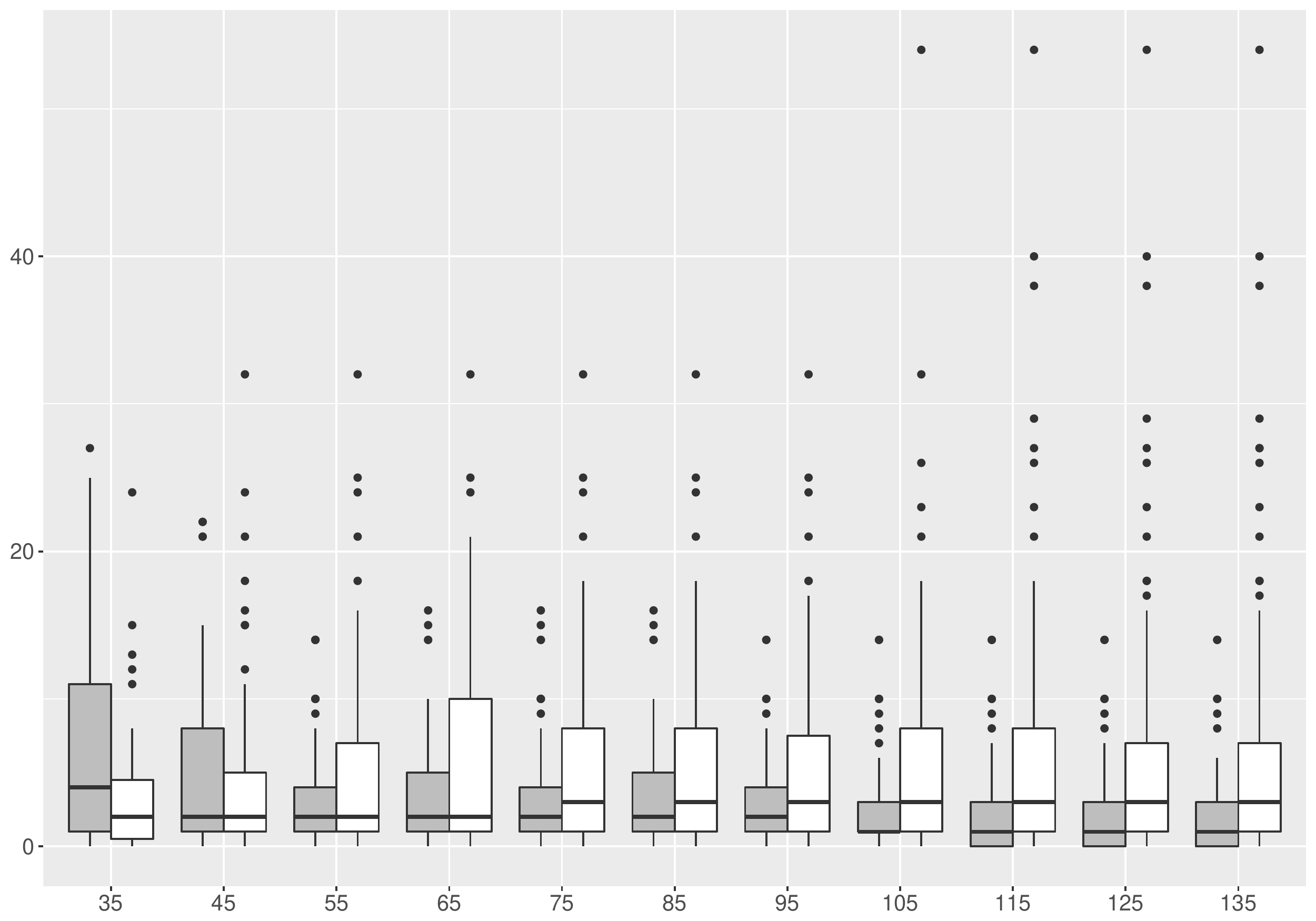}
\caption{\label{fig:hausdorff} Boxplots for the infimum parts of the Hausdorff distances $d_1$ (white) and $d_2$ (gray)
between the change-points found by \cite{dixon2012topological} and our approach
 for different values of the estimated number of change-points.}
\end{center}
\end{figure}

We can see from this figure that some differences exist between the two
approaches.  However, when the number of estimated change-points
considered in our methodology is on a par with the one of
\cite{dixon2012topological}, the position of the block boundaries are
very close as displayed in Figure \ref{fig:topological}.

 \begin{figure}[!h]
\begin{center}
   \includegraphics[height=9cm,width=9cm]{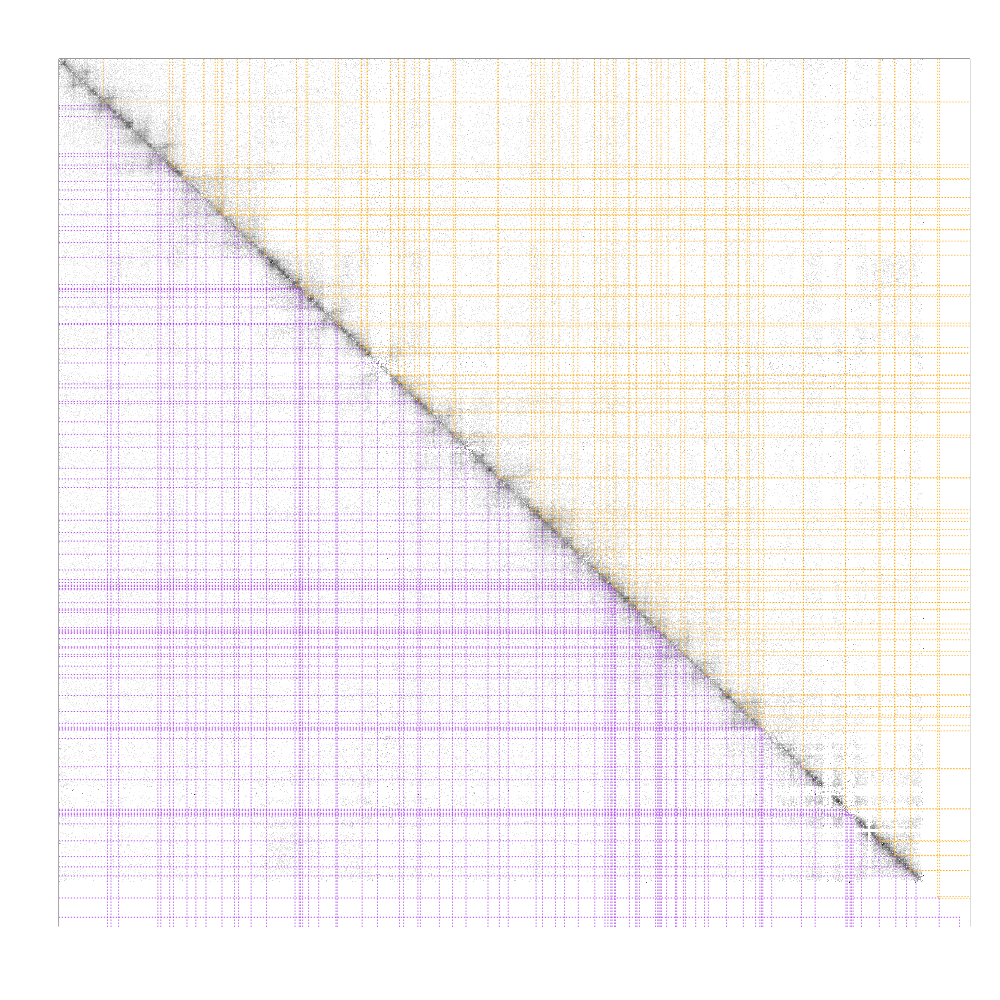}
\caption{\label{fig:topological} Topological domains detected by \cite{dixon2012topological} (upper
 triangular part of the matrix) and by our method (lower triangular part of
 the matrix).}
\end{center}
\end{figure}

%%% Local Variables:
%%% mode: latex
%%% eval: (TeX-PDF-mode 1)
%%% TeX-master: "Test_HiC.tex"
%%% ispell-local-dictionary: "en_US"
%%% eval: (flyspell-mode 1)
%%% End:

%% file: proof2_SL.tex
\subsection{Proof of Theorem \ref{thm_OP1_1rupture}}\label{proof2}

For proving Theorem~\ref{thm_OP1_1rupture}, we first compute the expectation of $S_n(n_1)$.

\begin{align*}
\Esp{S_n(n_1)} & = \sum_{i=1}^{\n} \Esp{U_{n,i}^2(n_1)} \\
& = \frac{1}{\n\nun(\n-\nun)} \sum_{i=1}^{\n} \Esp{\left(\sum_{j_0=1}^{n_1} \sum_{j_1=n_1+1}^{n} h(X_{i,j_0},X_{i,j_1})\right)^2} \\
& = \frac{1}{\n\nun(\n-\nun)} \sum_{i=1}^{\n} \sum_{1 \leq j_0,k_0 \leq n_1} \sum_{n_1+1 \leq j_1,k_1 \leq n} \Esp{h(X_{i,j_0},X_{i,j_1}) h(X_{i,k_0},X_{i,k_1})} \\
& = \frac{1}{\n\nun(\n-\nun)} \sum_{i=1}^{\n} \left\{ \sum_{j_0=1}^{n_1} \sum_{j_1=n_1+1}^{n} \Esp{h^2(X_{i,j_0},X_{i,j_1})} \right. \\
& \quad + \sum_{j_0=1}^{n_1} \sum_{n_1+1 \leq j_1 \neq k_1 \leq n} \Esp{h(X_{i,j_0},X_{i,j_1}) h(X_{i,j_0},X_{i,k_1})} \\
& \quad + \sum_{1 \leq j_0 \neq k_0 \leq n_1} \sum_{j_1=n_1+1}^{n} \Esp{h(X_{i,j_0},X_{i,j_1}) h(X_{i,k_0},X_{i,j_1})} \\
& \quad \left. + \sum_{1 \leq j_0 \neq k_0 \leq n_1} \sum_{n_1+1 \leq j_1 \neq k_1 \leq n} \Esp{h(X_{i,j_0},X_{i,j_1}) h(X_{i,k_0},X_{i,k_1})} \right\}.
\end{align*}

By using Lemma~\ref{lemma_1rupture}, we get that

\begin{align*}
\Esp{S_n(n_1)} & = \frac{1}{\n\nun(\n-\nun)} \sum_{i=1}^{\n} \left\{ n_1(n-n_1)+ \frac{1}{3}n_1 (n-n_1)(n-n_1-1) + \frac{1}{3}n_1(n_1-1) (n-n_1)\right\} \\
& = 1 + \frac{n-n_1-1}{3} + \frac{n_1-1}{3} = \frac{n+1}{3}.
\end{align*}

In order to derive the asymptotic behavior of $S_n(n_1)$ we write  the centered version of $S_n(n_1)$ as follows:
\begin{align*}
S_n(n_1) - \Esp{S_n(n_1)} & =  \frac{1}{\n\nun(\n-\nun)} \sum_{i=1}^{\n} \left(\sum_{j_0=1}^{n_1} \sum_{j_1=n_1+1}^{n} h(X_{i,j_0},X_{i,j_1})\right)^2 - \frac{n+1}{3} \\
& = \frac{1}{\n\nun(\n-\nun)} \left\{ \sum_{i=1}^{\n} \sum_{j_0=1}^{n_1} \sum_{j_1=n_1+1}^{n} \left[h^2(X_{i,j_0},X_{i,j_1})-1\right] \right. \\
& \quad + \sum_{i=1}^{\n} \sum_{j_0=1}^{n_1} \sum_{n_1+1 \leq j_1 \neq k_1 \leq n} \left[h(X_{i,j_0},X_{i,j_1}) h(X_{i,j_0},X_{i,k_1})-1/3\right] \\
& \quad + \sum_{i=1}^{\n} \sum_{1 \leq j_0 \neq k_0 \leq n_1} \sum_{j_1=n_1+1}^{n} \left[h(X_{i,j_0},X_{i,j_1}) h(X_{i,k_0},X_{i,j_1})-1/3\right] \\
& \quad \left. + \sum_{i=1}^{\n} \sum_{1 \leq j_0 \neq k_0 \leq n_1} \sum_{n_1+1 \leq j_1 \neq k_1 \leq n} h(X_{i,j_0},X_{i,j_1}) h(X_{i,k_0},X_{i,k_1}) \right\}\\
&=: \frac{1}{\n\nun(\n-\nun)} \left\{A+B+C+D\right\},
\end{align*}
where each term of this equality is centered. First, we observe that $A=0$ a.s.\,(almost surely) by Assertion~\ref{lem1:ii} of Lemma~\ref{lemma_1rupture}.

By using the Markov inequality we get that for all $\varepsilon>0$,

\begin{align*}
&\P\left(\left|\frac{B}{\sqrt{n}}\right|>\frac{6n^3}{\varepsilon}\right)\\
&\leq\varepsilon n^{-7/2}\Esp{|B|}/6\\
&\leq
\frac{\varepsilon} {6n^{7/2}}\sum_{i=1}^{\n}\Esp{\left| \sum_{j_0=1}^{n_1} \sum_{n_1+1 \leq j_1 \neq k_1 \leq n} \left[h(X_{i,j_0},X_{i,j_1}) h(X_{i,j_0},X_{i,k_1})-1/3\right]\right|}.
\end{align*}

By using the Cauchy-Schwarz inequality, we thus get that

\begin{align*}
&\P\left(\left|\frac{B}{\sqrt{n}}\right|>\frac{6n^3}{\varepsilon}\right)\\
&\leq\frac{\varepsilon} {6n^{7/2}}\sum_{i=1}^{\n}\left(\Esp{\left( \sum_{j_0=1}^{n_1} \sum_{n_1+1 \leq j_1 \neq k_1 \leq n} \left[h(X_{i,j_0},X_{i,j_1}) h(X_{i,j_0},X_{i,k_1})-1/3\right]\right)^2} \right)^{1/2}\\
&=\frac{\varepsilon} {6n^{7/2}}\sum_{i=1}^{\n}\Bigg( \sum_{1 \leq j_0,j_0^{\prime} \leq n_1} \sum_{n_1+1 \leq j_1 \neq k_1 \leq n} \sum_{n_1+1 \leq j_1^{\prime} \neq k_1^{\prime} \leq n}\E\Big[\left(h(X_{i,j_0},X_{i,j_1}) h(X_{i,j_0},X_{i,k_1})-1/3\right)\\
&\hskip210pt\times\left(h(X_{i,j_0^{\prime}},X_{i,j_1^{\prime}})h(X_{i,j_0^{\prime}},X_{i,k_1^{\prime}})-1/3\right)\Big]\Bigg)^{1/2}.
\end{align*}

By Assertion~\ref{lem1:iii} of Lemma~\ref{lemma_1rupture}, the above expectation is equal to zero when the cardinality of the set of indices $\left\{j_0,j_0^{\prime},j_1,j_1^{\prime},k_1,k_1^{\prime}\right\}$ equals 6. Indeed, the right-hand and left-hand side of the product in the expectation are independent in that case. Thus, only the cases where the cardinality of the set is smaller or equal to 5 have to be considered.
%and we easily bound the expectation by an absolute constant $\kappa>0$ considering that $h$ is bounded.
Moreover, note that 
$$
\big|(h(x,y)h(z,t)-1/3)\times(h(x',y')h(z',t')-1/3)\big| \leq 16/9 \leq 2,
$$ 
for all $x,y,z,t,x',y',z',t'$. Hence we get that, for all $\varepsilon>0$,

\begin{align}\label{Thm:B}
\P\left(\left|\frac{B}{\sqrt{n}}\right|>\frac{6n^3}{\varepsilon}\right)\leq
\frac{\varepsilon} {6n^{7/2}}\sum_{i=1}^{\n}2n^{5/2}=\varepsilon/3.
\end{align}
%
% Version avec M
%\begin{align*}
%&\P\left(\left|\frac{B}{\sqrt{n}}\right|>M\right)\\
%&\leq\Esp{B}/M\sqrt{n}\\
%&\leq
%\frac{1}{M\sqrt{n}}\sum_{i=1}^{\n}\Esp{\left| \sum_{j_0=1}^{n_1} \sum_{n_1+1 \leq j_1 \neq k_1 \leq n} \left[h(X_{i,j_0},X_{i,j_1}) h(X_{i,j_0},X_{i,k_1})-1/3\right]\right|} \\
%&\leq
%\frac{1}{M\sqrt{n}}\sum_{i=1}^{\n}\left(\Esp{\left( \sum_{j_0=1}^{n_1} \sum_{n_1+1 \leq j_1 \neq k_1 \leq n} \left[h(X_{i,j_0},X_{i,j_1}) h(X_{i,j_0},X_{i,k_1})-1/3\right]\right)^2} \right)^{1/2}\\
%&=\frac{1}{M\sqrt{n}}\sum_{i=1}^{\n}\Bigg( \sum_{1 \leq j_0,j_0^{\prime} \leq n_1} \sum_{n_1+1 \leq j_1 \neq k_1 \leq n} \sum_{n_1+1 \leq j_1^{\prime} \neq k_1^{\prime} \leq n}\E\Big[\left(h(X_{i,j_0},X_{i,j_1}) h(X_{i,j_0},X_{i,k_1})-1/3\right)\\
%&\hskip250pt\times\left(h(X_{i,j_0^{\prime}},X_{i,j_1^{\prime}})h(X_{i,j_0^{\prime}},X_{i,k_1^{\prime}})-1/3\right)\Big]\Bigg)^{1/2},
%\end{align*}
%\begin{align}\label{B}
%\P\left(\left|\frac{B}{\sqrt{n}}\right|>M\right)\leq\frac{1}{M\sqrt{n}}\sum_{i=1}^{\n}n^{5/2}\kappa=\kappa n^3/M.
%\end{align}
%
Using similar arguments, we get that for all $\varepsilon>0$,

\begin{align}\label{Thm:C}
\P\left(\left|\frac{C}{\sqrt{n}}\right|>\frac{6n^3}{\varepsilon}\right)\leq\varepsilon/3.
\end{align}
By using the Markov and the Cauchy-Schwarz inequalities as previously, we get that, for all $\varepsilon>0$,
\begin{align*}
&\P\left(\left|\frac{D}{\sqrt{n}}\right|>\frac{3n^3}{\varepsilon}\right)\\
&\leq\varepsilon n^{-7/2}\Esp{|D|}/3\\
&\leq\frac{\varepsilon}{3n^{7/2}}
\Esp{\left|\sum_{i=1}^{\n} \sum_{1 \leq j_0 \neq k_0 \leq n_1} \sum_{n_1+1 \leq j_1 \neq k_1 \leq n} h(X_{i,j_0},X_{i,j_1}) h(X_{i,k_0},X_{i,k_1})\right|} \\
&\leq\frac{\varepsilon}{3n^{7/2}}
\left(\Esp{\left(\sum_{i=1}^{\n} \sum_{1 \leq j_0 \neq k_0 \leq n_1} \sum_{n_1+1 \leq j_1 \neq k_1 \leq n} h(X_{i,j_0},X_{i,j_1}) h(X_{i,k_0},X_{i,k_1})\right)^2} \right)^{1/2}\\
&=\frac{\varepsilon}{3n^{7/2}}
\Big(\E\Big[\sum_{i=1}^{\n} \sum_{1 \leq j_0 \neq k_0 \leq n_1} \sum_{n_1+1 \leq j_1 \neq k_1 \leq n} h(X_{i,j_0},X_{i,j_1}) h(X_{i,k_0},X_{i,k_1})\\
&\hskip80pt\times\sum_{i^{\prime}=1}^{\n} \sum_{1 \leq j_0^{\prime} \neq k_0^{\prime} \leq n_1} \sum_{n_1+1 \leq j_1^{\prime} \neq k_1^{\prime} \leq n} h(X_{i^{\prime},j_0^{\prime}},X_{i^{\prime},j_1^{\prime}}) h(X_{i^{\prime},k_0^{\prime}},X_{i^{\prime},k_1^{\prime}})\Big] \Big)^{1/2}.
\end{align*}
The above expectation is equal to zero when the cardinality of $\left\{i,i',j_0,j_0^{\prime},k_0,k_0^{\prime},j_1,j_1^{\prime},k_1,k_1^{\prime}\right\}$ is greater than 8 and smaller than 10 by Assertion~\ref{lem1:v} of Lemma~\ref{lemma_1rupture}. Only the cases where the cardinality of the set is smaller
than 7 have to be considered. Observe moreover that 
$$
\big|h(x,y)h(z,t)h(x',y')h(z',t')\big| \leq 1, \textrm{ for all } x,y,z,t,x',y',z',t'\in\mathbb{R}. 
$$
Therefore, for all $\varepsilon>0$, we get,

\begin{align}\label{Thm:D}
\P\left(\left|\frac{D}{\sqrt{n}}\right|>\frac{3n^3}{\varepsilon}\right)\leq
\frac{\varepsilon}{3n^{7/2}}\times n^{7/2}=\varepsilon/3.
\end{align}
Finally, by combining \eqref{Thm:B}, \eqref{Thm:C} and \eqref{Thm:D}, we obtain that, for all $\varepsilon>0$,
\begin{eqnarray*}
\P\left(\n\nun(\n-\nun)\times\frac{\left|S_n(n_1) - \Esp{S_n(n_1)}\right|}{\sqrt{n}}>\frac{15n^3}{\varepsilon}\right)&\leq&
\varepsilon,
\end{eqnarray*}
which can be rewritten as
\begin{eqnarray*}
\P\left(\frac{\left|S_n(n_1) - \Esp{S_n(n_1)}\right|}{\sqrt{n}}>\frac{15n^2}{\varepsilon\nun(\n-\nun)}\right)&\leq&
\varepsilon.
\end{eqnarray*}
Since we assumed that $n_1/n \rightarrow \tau_1$ as $n\to\infty$, we get that
\[\frac{S_n(n_1) - \Esp{S_n(n_1)}}{\sqrt{n}}=O_{\mathbb{P}}(1),\]
which concludes the proof of Theorem \ref{thm_OP1_1rupture}.

%%% Local Variables:
%%% mode: latex
%%% eval: (TeX-PDF-mode 1)
%%% TeX-master: "Test_HiC.tex"
%%% ispell-local-dictionary: "en_US"
%%% eval: (flyspell-mode 1)
%%% End:

%% file: proofL_SL.tex
\subsection{Proof of Theorem \ref{Thm:Multiple}}\label{proofL}
%\LS{}{peut-etre a reecrire avec $\Delta$ en fonction des changements que l'on fera aux sections 2.2 et 2.3}
Let us start with the computation of the expectation of $\SnunL$. 
First observe that, for any $i\in\{1,\ldots,n\}$ and $\ell\in\{0,\ldots,L\}$,
\begin{eqnarray}\label{eq:esp_rank_1}
\left(\Rbarli-\frac{\n+1}{2}\right)^2
&=&\left(\frac{1}{\nlun-\nl}\sum_{j=\nl+1}^{\nlun}\Rji-\frac{\n+1}{2}\right)^2\nonumber\\
&=&\frac{1}{(\nlun-\nl)^2}\sum_{j=\nl+1}^{\nlun}\left(\Rji-\frac{\n+1}{2}\right)^2\nonumber\\
&&+\frac{1}{(\nlun-\nl)^2}\sum_{\nl+1\leq j\neq\jp\leq\nlun}\left(\Rji-\frac{\n+1}{2}\right)\left(R_{\jp}^{(i)}-\frac{\n+1}{2}\right)\nonumber\\
&=&\frac{1}{(\nlun-\nl)^2}\left(\sum_{j=\nl+1}^{\nlun}A_j^{(i)}+\sum_{\nl+1\leq j\neq\jp\leq\nlun}B_{jj^{\prime}}^{(i)}\right),
\end{eqnarray}
where
\begin{equation*}
A_j^{(i)} =\left(\Rji-\frac{\n+1}{2}\right)^2 \textrm{ and }
B_{jj^{\prime}}^{(i)}=\left(\Rji-\frac{\n+1}{2}\right)\left(R_{\jp}^{(i)}-\frac{\n+1}{2}\right).
\end{equation*}
By using the definition \eqref{Def:Rang} of $\Rji$, we get that, 
\begin{eqnarray}\label{aji}
A_j^{(i)}&=&\left(\sum_{k=1}^{\n}\1_{\{\Xik\leq\Xij\}}-\frac{\n+1}{2}\right)^2
=\left(1+\sum_{\substack{k=1 \\ k\neq j}}^{\n}
\1_{\{\Xik\leq\Xij\}}-\frac{\n+1}{2}\right)^2\nonumber\\
&=&\left(\sum_{\substack{k=1 \\ k\neq j}}^{\n}
\left(\1_{\{\Xik\leq\Xij\}}-\frac{1}{2}\right)\right)^2\nonumber\\
&=&\sum_{\substack{k=1 \\ k\neq j}}^{\n}
g(\Xik,\Xij)^2+\sum_{\substack{k=1 \\ k\neq j}}^{\n}\sum_{\substack{\kp=1 \\ \kp\neq k \\ k\neq j}}^{\n}g(\Xik,\Xij)g(\Xikp,\Xij),
\end{eqnarray}
where $g(x,y)=\1_{x\leq y}-\frac{1}{2}$
and, by Assertions~\ref{lem2:i} and \ref{lem2:ii} of Lemma~\ref{Lem:Esp}, we get
\begin{eqnarray}\label{eq:esp_Aij}
\Esp{A_j^{(i)}}=\frac{1}{4}(n-1)+\frac{1}{12}(n-1)(n-2)=\frac{(n-1)(n+1)}{12}.
\end{eqnarray}
Then, we decompose $B_{jj^{\prime}}^{(i)}$ in the four following terms. 
\begin{eqnarray}\label{bjji}
B_{jj^{\prime}}^{(i)}%&=&\Esp{\left(\Rji-\frac{\n+1}{2}\right)\left(R_{\jp}^{(i)}-\frac{\n+1}{2}\right)}\\
&=& \left(\Rji-\frac{\n+1}{2}\right)\left(R_{\jp}^{(i)}-\frac{\n+1}{2}\right)\nonumber\\
&=& \left(\sum_{k=1}^{\n}\1_{\{\Xik\leq\Xij\}}-\frac{\n+1}{2}\right)\left(\sum_{k'=1}^{\n}\1_{\{X_{i,k'}\leq X_{i,j'}\}}-\frac{\n+1}{2}\right)\nonumber\\
&=&\left(1+\sum_{\substack{k=1 \\ k\neq j}}^{\n}
\1_{\{\Xik\leq\Xij\}}-\frac{\n+1}{2}\right)\left(1+\sum_{\substack{k'=1 \\ k'\neq j'}}^{\n}
\1_{\{X_{i,k'}\leq X_{i,j'}\}}-\frac{\n+1}{2}\right)\nonumber\\
&=&\left(\sum_{\substack{k=1 \\ k\neq j}}^{\n}\left(
\1_{\{\Xik\leq\Xij\}}-\frac{1}{2}\right)\right)\left(\sum_{\substack{k'=1 \\ k'\neq j'}}^{\n}
\left(\1_{\{X_{i,k'}\leq X_{i,j'}\}}-\frac{1}{2}\right)\right)\nonumber\\
&=&\sum_{\substack{k=1 \\ k\neq j}}^{\n}\sum_{\substack{\kp=1 \\ \kp\neq j^{\prime}}}^{\n}g(\Xik,\Xij)g(\Xikp,\Xijp)\nonumber\\
&=&g(\Xijp,\Xij)g(\Xij,\Xijp)\nonumber\\
&&+\sum_{\substack{k=1 \\ k\neq j \\ k\neq j^{\prime}}}^{\n}
g(\Xik,\Xij)g(\Xij,\Xijp)\nonumber\\
&&+\sum_{\substack{\kp=1 \\ \kp\neq j^{\prime} \\ \kp\neq j}}^{\n}
g(\Xijp,\Xij)g(\Xikp,\Xijp)\nonumber\\
&&\sum_{\substack{k=1 \\ k\neq j \\ k\neq j^{\prime}}}^{\n}\sum_{\substack{\kp=1 \\ \kp\neq j^{\prime} \\ \kp\neq j}}^{\n}
g(\Xik,\Xij)g(\Xikp,\Xijp)\nonumber\\
&=:&B_1+B_2+B_3+B_4.
\end{eqnarray}
By Lemma \ref{Lem:Esp},
we obtain that 
$$
\Esp{B_1}=-\frac{1}{4},\; \Esp{B_2}=\Esp{B_3}=-\frac{n-2}{12}
\textrm{ and }\Esp{B_4}=\frac{n-2}{12},
$$
since the only term in the sum defining $B_4$ having a non
null expectation is the one for which $k=\kp$.
Hence,
\begin{equation}\label{eq:esp_Bij}
\Esp{B_{jj^{\prime}}^{(i)}}=-\frac{1}{4}-2\times\frac{n-2}{12}+\frac{n-2}{12}=-\frac14-\frac{n-2}{12}=-\frac{n+1}{12}.
\end{equation}
By (\ref{eq:esp_rank_1}), (\ref{eq:esp_Aij}) and (\ref{eq:esp_Bij}),
\begin{eqnarray*}
&&\Esp{\left(\Rbarli-\frac{\n+1}{2}\right)^2}\\
&=&\frac{1}{(\nlun-\nl)^2}\left\{\sum_{j=\nl+1}^{\nlun}\frac{(n-1)(n+1)}{12}
-\sum_{\nl+1\leq
  j\neq\jp\leq\nlun}\frac{(n+1)}{12}\right\}\\
&=&\frac{1}{(\nlun-\nl)}\frac{(n-1)(n+1)}{12}-\frac{(\nlun-\nl)(\nlun-\nl-1)}{(\nlun-\nl)^2}\times\frac{(n+1)}{12}\\
&=&\frac{1}{(\nlun-\nl)}\left\{\frac{(n-1)(n+1)}{12}-\frac{(n+1)(\nlun-\nl-1)}{12}\right\}.
\end{eqnarray*}
By (\ref{Def:Stat}), we get that
\begin{eqnarray*}
\Esp{\SnunL}&=&
\frac{4}{\n^2}
\sum_{\ell=0}^{L}\left(\nlun-\nl\right)
\sum_{i=1}^{\n}\Esp{\left(\Rbarli-\frac{\n+1}{2}\right)^2}\\
&=&\frac{4}{\n}\sum_{\ell=0}^{L}\left\{\frac{(n-1)(n+1)}{12}-\frac{(n+1)(\nlun-\nl-1)}{12}\right\}\\
&=&\frac{4(n+1)}{12\n}\left\{(L+1)(n-1)-(n-L-1)\right\}=\frac{L(\n+1)}{3}.
\end{eqnarray*}
%We finally obtain that $\Esp{\SnunL}=\frac{L(\n+1)}{3}$, by combining the just above equality with definition \eqref{Def:Stat} of $\SnunL$.\\
Now we focus on the asymptotic behavior of $\SnunL$. For this, we decompose the centered version of $\SnunL$ as follows.
\begin{align*}
&\SnunL-\Esp{\SnunL}\\
&=\frac{4}{\n^2}
\sum_{\ell=0}^{L}\left(\nlun-\nl\right)
\sum_{i=1}^{\n}\left(\Rbarli-\frac{\n+1}{2}\right)^2-\frac{L(n+1)}{3}\\
&=\frac{4}{\n^2}\sum_{\ell=0}^{L}\left(\nlun-\nl\right)
\sum_{i=1}^{\n}\left[
\frac{1}{(\nlun-\nl)^2}
\left(
\sum_{j=\nl+1}^{\nlun}\left(A_j^{(i)}-\Esp{A_j^{(i)}}\right)\right.\right.\\
&\hspace{7cm}+\left.\left.\sum_{\nl+1\leq j\neq\jp\leq\nlun}\left(B_{jj^{\prime}}^{(i)}-\Esp{B_{jj^{\prime}}^{(i)}}\right)
\right)
\right]\\
&=\frac{4}{\n^2}\sum_{\ell=0}^{L}\frac{1}{\nlun-\nl}
\sum_{i=1}^{\n}
\sum_{t=1}^7Z_i^{(t)},
\end{align*}
where $A_j^{(i)}$ and $B_{jj^{\prime}}^{(i)}$ are defined in
\eqref{aji} and \eqref{bjji}, and the $Z_i^{(t)}$ are defined as follows:
\begin{eqnarray*}
Z_i^{(1)}&=&\sum_{j=\nl+1}^{\nlun}\sum_{\substack{k=1 \\ k\neq j}}^{\n}
\left\{
g(\Xik,\Xij)^2-\frac{1}{4}\right\},\\
Z_i^{(2)}&=&\sum_{j=\nl+1}^{\nlun}\sum_{\substack{k=1 \\ k\neq j}}^{\n}\sum_{\substack{\kp=1 \\ \kp\neq k \\ k\neq j}}^{\n}\left\{g(\Xik,\Xij)g(\Xikp,\Xij)-\frac{1}{12}\right\},\\
Z_i^{(3)}&=&\sum_{\nl+1\leq j\neq\jp\leq\nlun}\left\{g(\Xijp,\Xij)g(\Xij,\Xijp)+\frac{1}{4}\right\},\\
Z_i^{(4)}&=&\sum_{\nl+1\leq j\neq\jp\leq\nlun}\sum_{\substack{k=1 \\ k\neq j \\ k\neq j^{\prime}}}^{\n}
\left\{g(\Xik,\Xij)g(\Xij,\Xijp)+\frac{1}{12}\right\},\\
Z_i^{(5)}&=&\sum_{\nl+1\leq j\neq\jp\leq\nlun}\sum_{\substack{\kp=1 \\ \kp\neq j^{\prime} \\ \kp\neq j}}^{\n}
\left\{g(\Xijp,\Xij)g(\Xikp,\Xijp)+\frac{1}{12}\right\},\\
Z_i^{(6)}&=&\sum_{\nl+1\leq j\neq\jp\leq\nlun}\sum_{\substack{k=1 \\ k\neq j \\ k\neq j^{\prime}}}^{\n}
\left\{g(\Xik,\Xij)g(\Xik,\Xijp)-\frac{1}{12}\right\},\\
Z_i^{(7)}&=&\sum_{\nl+1\leq j\neq\jp\leq\nlun}\sum_{\substack{k=1 \\ k\neq j \\ k\neq j^{\prime}}}^{\n}\sum_{\substack{\kp=1 \\ \kp\neq j^{\prime} \\ \kp\neq j \\ \kp\neq k}}^{\n}
g(\Xik,\Xij)g(\Xikp,\Xijp).
\end{eqnarray*}
Then, we get that, for all $M>0$,
\begin{align*}
&\P\left(\left|\frac{\SnunL-\Esp{\SnunL}}{\sqrt{n}}\right| > M\right) 
 \leq \sum_{\ell=0}^{L}\sum_{t=1}^7\P\left(\frac{4}{\n^2}\frac{1}{\nlun-\nl}
\left|\sum_{i=1}^{\n}
Z_i^{(t)}\right| > \frac{M\sqrt{n}}{7(L+1)}\right) \\
& \leq \sum_{\ell=0}^{L}\sum_{t=1}^7\P\left(
\left|\sum_{i=1}^{\n}
Z_i^{(t)}\right| > \frac{M(\nlun-\nl)n^{5/2}}{28(L+1)}\right).
\end{align*}
Using the Markov inequality we get that
\begin{align*}
\P\left(\left|\frac{\SnunL-\Esp{\SnunL}}{\sqrt{n}}\right| > M\right) 
 \leq \sum_{\ell=0}^{L}\sum_{t=1}^7 \frac{28(L+1)}{M(\nlun-\nl)n^{5/2}}\Esp{
\left|\sum_{i=1}^{\n}
Z_i^{(t)}\right|}.
\end{align*}
By using the Cauchy-Schwarz inequality we obtain that
\begin{align*}
\P\left(\left|\frac{\SnunL-\Esp{\SnunL}}{\sqrt{n}}\right| > M\right) 
 \leq \sum_{\ell=0}^{L}\sum_{t=1}^7 \frac{28(L+1)}{M(\nlun-\nl)n^{5/2}}\left(\Esp{
\left(\sum_{i=1}^{\n}
Z_i^{(t)}\right)^2}\right)^{1/2}.
\end{align*}
We shall now give upper bounds for $\Esp{
\left(\sum_{i=1}^{\n}
Z_i^{(t)}\right)^2}$ for all $t\in\{1,\ldots,7\}$. 
First, by using Assertion~\ref{lem2:i} of Lemma \ref{Lem:Esp}, we get
\begin{align*}
&\Esp{\left(\sum_{i=1}^{\n}Z_i^{(1)}\right)^2}
=\sum_{i=1}^{\n}\sum_{i^{\prime}=1}^{\n}\Esp{Z_i^{(1)}Z_{i^\prime}^{(1)}}\\
&=\sum_{i=1}^{\n}\sum_{i^{\prime}=1}^{\n}
\sum_{j=\nl+1}^{\nlun}\sum_{\substack{k=1 \\ k\neq j}}^{\n}\sum_{r=\nl+1}^{\nlun}\sum_{\substack{s=1 \\s\neq r}}^{\n}
\Esp{\left\{
g(\Xik,\Xij)^2-\frac{1}{4}\right\}
\left\{
g(X_{i^{\prime}s},X_{i^{\prime}r})^2-\frac{1}{4}\right\}}=0.
\end{align*}
Then,
\begin{align*}
&\Esp{\left(\sum_{i=1}^{\n}Z_i^{(2)}\right)^2}
=\sum_{i=1}^{\n}\sum_{i^{\prime}=1}^{\n}\Esp{Z_i^{(2)}Z_{i^\prime}^{(2)}}\\
&=\sum_{i=1}^{\n}\sum_{i^{\prime}=1}^{\n}
\sum_{j=\nl+1}^{\nlun}\sum_{\substack{k=1 \\ k\neq j}}^{\n}\sum_{\substack{\kp=1 \\ \kp\neq k \\ k\neq j}}^{\n}
\sum_{r=\nl+1}^{\nlun}\sum_{\substack{s=1 \\ s\neq r}}^{\n}\sum_{\substack{s^{\prime}=1 \\ s^{\prime}\neq s \\ s\neq r}}^{\n}
\E\Big[\left\{g(\Xik,\Xij)
g(\Xikp,\Xij)-\frac{1}{12}\right\}\\
&\hskip180pt\left\{g(X_{i^{\prime}s}, X_{i^{\prime}r}) g(X_{i^{\prime}s^{\prime}}, X_{i^{\prime}r})-\frac{1}{12}\right\}\Big].
\end{align*}
The above expectation is equal to zero when the cardinality of the set
of indices $\left\{i,i^{\prime},j,k,k^{\prime},r,s,s^{\prime}\right\}$
equals 8 by Assertion~\ref{lem2:ii} of Lemma~\ref{Lem:Esp}. 
Hence, only the cases where the cardinality of this set is smaller or
equal to 7 have to be considered. 
Since 
$$
\left|\left(g(x,y)g(z,t)-\frac{1}{12}\right)\left(g(x^{\prime},y^{\prime})g(z^{\prime},t^{\prime})-\frac{1}{12}\right)\right|
\leq 1/9\leq 1,
$$ 
for all $x,y,z,t,x',y',z',t'\in\mathbb{R}$, we get that,
\[\Esp{\left(\sum_{i=1}^{\n}Z_i^{(2)}\right)^2} \leq n^7.\]
By using similar arguments and Assertion~\ref{lem2:iii} of Lemma \ref{Lem:Esp},
we get that
$$
\Esp{\left(\sum_{i=1}^{\n}Z_i^{(6)}\right)^2}\leq n^7.
$$
By using similar arguments as those used for bounding $\Esp{\left(\sum_{i=1}^{\n}Z_i^{(2)}\right)^2}$
and by Assertion~\ref{lem2:ii} of Lemma \ref{Lem:Esp}, we get that $\Esp{g(X,Y)g(Y,Z)}=-\Esp{g(X,Y)g(Z,Y)}=-1/12$. Hence,
$$
\Esp{\left(\sum_{i=1}^{\n}Z_i^{(4)}\right)^2}\leq n^7 \;\textrm{ and }\; \Esp{\left(\sum_{i=1}^{\n}Z_i^{(5)}\right)^2}\leq n^7.
$$
% For the quantities $Z_i^4, Z_i^5$ and $Z_i^6$, with similar arguments, we find exactly the same upper bound. 
By using Assertion~\ref{lem2:i} of Lemma \ref{Lem:Esp}, %and the symmetry of the matrix $\textbf{X}=(X_{i,j})_{1\leq i,j\leq n}$,
we obtain that
\begin{align*}
&\Esp{\left(\sum_{i=1}^{\n}Z_i^{(3)}\right)^2}
=\sum_{i=1}^{\n}\sum_{i^{\prime}=1}^{\n}\Esp{Z_i^{(3)}Z_{i^\prime}^{(3)}}\\
&=\sum_{i=1}^{\n}\sum_{i^{\prime}=1}^{\n}
\sum_{\nl+1\leq j\neq\jp\leq\nlun}
\sum_{\nl+1\leq r\neq r^{\prime}\leq\nlun}
\E\Bigg[
\left\{g(\Xijp,\Xij)g(\Xij,\Xijp)+\frac{1}{4}\right\}\\
&\hskip180pt\left\{g(X_{i^{\prime}r^{\prime}},X_{i^{\prime}r})g(X_{i^{\prime}r},X_{i^{\prime}r^{\prime}})+\frac{1}{4}\right\}\Bigg]=0
\end{align*}
Finally, 
\begin{align*}
&\Esp{\left(\sum_{i=1}^{\n}Z_i^{(7)}\right)^2}
=\sum_{i=1}^{\n}\sum_{i^{\prime}=1}^{\n}\Esp{Z_i^{(7)}Z_{i^\prime}^{(7)}}\\
&=\sum_{i=1}^{\n}\sum_{i^{\prime}=1}^{\n}
\sum_{\nl+1\leq j\neq\jp\leq\nlun}\sum_{\substack{k=1 \\ k\neq j \\ k\neq j^{\prime}}}^{\n}\sum_{\substack{\kp=1 \\ \kp\neq j^{\prime} \\ \kp\neq j \\ \kp\neq k}}^{\n}
\sum_{\nl+1\leq r\neq r^{\prime}\leq\nlun}\sum_{\substack{s=1 \\ s\neq r \\ s\neq r^{\prime}}}^{\n}\sum_{\substack{s^{\prime}=1 \\ s^{\prime}\neq r^{\prime} \\ s^{\prime}\neq r \\ s^{\prime}\neq s}}^{\n}
\E\Big[
g(\Xik,\Xij)g(\Xikp,\Xijp)\\
&\hskip180pt
g(X_{i^{\prime}s},X_{i^{\prime}r})g(X_{i^{\prime}s^{\prime}}, X_{i^{\prime}r^{\prime}})\Big].%\leq n^7.
\end{align*}
The above expectation is null when the the cardinality of the set of indices 
$\left\{i,i^{\prime},j,j^{\prime},k,k^{\prime},r,r^{\prime},s,s^{\prime}\right\}$ is equal or greater than 8, by using Assertion \ref{lem2:0} of Lemma \ref{Lem:Esp}. 
Observe moreover that 
$$
\left|g(x,y)g(z,t)g(x^{\prime},y^{\prime})g(z^{\prime},t^{\prime})\right|| \leq 1/16\leq 1,
$$ 
for all $x,y,z,t,x',y',z',t'\in\mathbb{R}$. Therefore, we get,
\[\Esp{\left(\sum_{i=1}^{\n}Z_i^{(7)}\right)^2} \leq n^7.\]
Thus, we obtain that, for all $M>0$,
\begin{align*}
\P\left(\left|\frac{\SnunL-\Esp{\SnunL}}{\sqrt{n}}\right| > M\right)
& \leq \frac{1}{M}\sum_{\ell=0}^{L}\frac{5\times 28(L+1)n^{7/2}}{(\nlun-\nl)n^{5/2}}.
\end{align*}
Since for any $\ell$, $\frac{n}{\nlun-\nl}$ converges to $\frac{1}{\tau_{\ell+1}-\tau_{\ell}}$, the right-hand side of the above inequality tends to 0 
when $M\to\infty$, which concludes the proof.

%%% Local Variables:
%%% mode: latex
%%% eval: (TeX-PDF-mode 1)
%%% TeX-master: "Test_HiC.tex"
%%% ispell-local-dictionary: "en_US"
%%% eval: (flyspell-mode 1)
%%% End:

%% file: proof_Sec2-3_CV.tex
\subsection{Proof of Equation (\ref{eq:dyn_recurs})} \label{proof_Sec.2.3}

By (\ref{eq:dyn:def}), 
$$
I_0(p)=\max_{1<n_1=p}\Delta(1:n_{1})=\Delta(1:p)
$$
and 
$$
I_{1}(p)=\max_{1<n_1<n_{2}=p}\{\Delta(1:n_{1})+\Delta(n_{1}+1:p)\}=\max_{1<n_1<n_{2}=p}\{I_0(n_1)+\Delta(n_{1}+1:p)\},
$$
which is (\ref{eq:dyn_recurs}) when $L=1$. By (\ref{eq:dyn:def}), 
$$
I_{2}(p)=\max_{1<n_1<n_{2}<n_{3}=p}\{\Delta(1:n_{1})+\Delta(n_{1}+1:n_2)+\Delta(n_{2}+1:p)\}.
$$
By using the previous expression of $I_{1}(p)$, we get that
$$
I_{2}(p)=\max_{1<n_{2}<p}\{I_1(n_2)+\Delta(n_{2}+1:p)\},
$$ 
which is (\ref{eq:dyn_recurs}) when $L=2$ and so on, which gives (\ref{eq:dyn_recurs}).

%%% Local Variables:
%%% mode: latex
%%% eval: (TeX-PDF-mode 1)
%%% TeX-master: "Test_HiC.tex"
%%% ispell-local-dictionary: "en_US"
%%% eval: (flyspell-mode 1)
%%% End:

%% file: lemma_SL.tex
\begin{lem}\label{lemma_1rupture}
Let $h$ be defined by $h(x,y)=\1_{\{x \leq y\}}-\1_{\{y \leq x\}}$. Then,
\begin{enumerate}
% \item\label{lem1:i} $\Esp{h(X,Y)}=0$,
% \item \label{lem1:ii} $h^2(X,Y)=1$ a.s.,
% \item \label{lem1:iii} $\Esp{h(X,Y) h(X,Z)} = 1/3$,
% \item \label{lem1:iv} $\Esp{h(X,Y) h(Z,Y)} = 1/3$,
% \item \label{lem1:v} $\Esp{h(X,Y) h(Z,T)} = 0$,
\labitem{(i)}{lem1:i} $\Esp{h(X,Y)}=0$,
\labitem{(ii)}{lem1:ii} $h^2(X,Y)=1$ a.s.,
\labitem{(iii)}{lem1:iii} $\Esp{h(X,Y) h(X,Z)} = 1/3$,
\labitem{(iv)}{lem1:iv} $\Esp{h(X,Y) h(Z,Y)} = 1/3$, 
\labitem{(v)}{lem1:v} $\Esp{h(X,Y) h(Z,T)} = 0$,
\end{enumerate}
where $X$, $Y$, $Z$ and $T$ are i.i.d. random variables having a continuous distribution function.
\end{lem}

\begin{proof}
\begin{enumerate}
\item[$(i)$] Let $X$ and $Y$ be i.i.d.\,random variables with cumulative distribution function $F$. We have:
$$\Esp{h(X,Y)}=\Esp{\1_{\{X \leq Y\}}}-\Esp{\1_{\{Y \leq
    X\}}}=\Esp{1-2F(X)}=0,$$
where we used that $F(X)$ is a uniform random variable on $[0,1]$.
\item[$(ii)$]  For all $x \neq y$ in $\R$, $h^2(x,y)=\left(\1_{\{x \leq y\}}-\1_{\{y \leq x\}}\right)^2 = \1_{\{x \leq y\}} + \1_{\{y\leq x\}} - 2\1_{\{x \leq y\}}\1_{\{y \leq x\}} = 1$. Consequently, $h^2(X,Y)=1$ a.s..
\item[$(iii)$] Let $X$, $Y$ and $Z$ be i.i.d.\,random variables with cumulative distribution function $F$. We have:
\begin{align*}
\Esp{h(X,Y) h(X,Z)} & = \Esp{\left(\1_{\{X \leq Y\}}-\1_{\{Y \leq X\}}\right)\left(\1_{\{X \leq Z\}}-\1_{\{Z \leq X\}}\right)} \\
& =  \Esp{\1_{\{X \leq Y\}} \1_{\{X \leq Z\}}} - \Esp{\1_{\{X \leq Y\}} \1_{\{Z \leq X\}}} \\
& \quad - \Esp{\1_{\{Y \leq X\}} \1_{\{X \leq Z\}}}+ \Esp{\1_{\{Y \leq X\}} \1_{\{Z \leq X\}}}\\
& = \Esp{(1-F(X))^2} - 2(\Esp{F(X)}-\Esp{F(X)^2}) + \Esp{F(X)^2} \\
& = 1/3 - 2(1/2-1/3) + 1/3 = 1/3,
\end{align*}
where we used that $F(X)$ is a uniform random variable on $[0,1]$.
\item[$(iv)$] Since $\Esp{h(X,Y) h(Z,Y)} =\Esp{h(Y,X)h(Y,Z)}=1/3$, the
  result comes from $(iii)$.
\item[$(v)$] By independance of $(X,Y)$ with $(Z,T)$,
$$\Esp{h(X,Y) h(Z,T)}=\Esp{h(X,Y)} \Esp{ h(Z,T)} = 0.$$
\end{enumerate}
\end{proof}

\begin{lem}\label{Lem:Esp} 
Let us define the function $g$ as $g(x,y)=\1_{\{x \leq y\}}-\frac{1}{2}$. 
Let $X$, $Y$ and $Z$ be i.i.d. random variables having a continuous distribution function. Then
\begin{enumerate}
\labitem{(i)}{lem2:0}  $\Esp{g(X,Y)}=0$,
\labitem{(ii)}{lem2:i} $g(X,Y)^2=\frac{1}{4}$ a.s.,
\labitem{(iii)}{lem2:ii} $\Esp{g(X,Y)g(Z,Y)}=\frac{1}{12}$,
\labitem{(iv)}{lem2:iii} $\Esp{g(X,Y)g(X,Z)}=\frac{1}{12}$.
\end{enumerate}
\end{lem}
\begin{proof}
\begin{enumerate}
\item[$(i)$] $\Esp{g(X,Y)}=\Esp{F(Y)}-1/2=0$, since $F(Y)$ is a uniform random variable on $[0,1]$.
\item[$(ii)$] For all $x,y$ in $\R$, $g(x,y)^2=\left(\1_{\{x \leq y\}}-\frac{1}{2}\right)^2 = \1_{\{x \leq y\}} + \frac{1}{4} -\1_{\{x \leq y\}}=\frac{1}{4}$. Consequently, $g^2(X,Y)=\frac{1}{4}$ a.s..
\item[$(iii)$] Let $X$, $Y$ and $Z$ be i.i.d.\,random variables with cumulative distribution function $F$. We have:
\begin{align*}
\Esp{g(X,Y) g(Z,Y)} & = \Esp{\left(\1_{\{X \leq Y\}}-\frac{1}{2}\right)\left(\1_{\{Z \leq Y\}}-\frac{1}{2}\right)}\\
&= \Esp{\1_{\{X \leq Y\}}\1_{\{Z \leq Y\}}}-\frac{1}{2}\Esp{\1_{\{Z \leq Y\}}}-\frac{1}{2}\Esp{\1_{\{X \leq Y\}}}+\frac{1}{4}\\
&=\Esp{F(Y)^2}-\Esp{F(Y)}+\frac{1}{4}\\
&=\frac{1}{3}-\frac{1}{2}+\frac{1}{4}=\frac{1}{12},
\end{align*}
where we used that $F(X)$ is a uniform random variable on $[0,1]$.
\item[$(iv)$] Note that
\begin{align*}
\Esp{g(X,Y) g(X,Z)} & = \Esp{\left(\1_{\{X \leq Y\}}-\frac{1}{2}\right)\left(\1_{\{X \leq Z\}}-\frac{1}{2}\right)}\\
&=\Esp{\left(1-\1_{\{Y \leq X\}}-\frac{1}{2}\right)\left(1-\1_{\{Z \leq X\}}-\frac{1}{2}\right)}\\
&=\Esp{\left(\frac{1}{2}-\1_{\{Y \leq X\}}\right)\left(\frac{1}{2}-\1_{\{Z \leq X\}}\right)}\\
&=\Esp{g(Y,X) g(Z,X)}=\frac{1}{12},
\end{align*}
by $(iii)$.
\end{enumerate}
\end{proof}

%%% Local Variables:
%%% mode: latex
%%% eval: (TeX-PDF-mode 1)
%%% TeX-master: "Test_HiC.tex"
%%% ispell-local-dictionary: "en_US"
%%% eval: (flyspell-mode 1)
%%% End: